\documentclass[oneside,english]{amsart}
\usepackage[T1]{fontenc}
\usepackage[latin9]{inputenc}
\usepackage{graphicx}
\usepackage{color}
\usepackage{amstext}
\usepackage{amsthm}
\usepackage{amssymb}
\usepackage{mathtools}
\usepackage{tikz}
\usepackage{stmaryrd}
\usepackage{subcaption, float}

\makeatletter

\usepackage[a4paper]{geometry}
\usepackage[unicode=true,pdfusetitle,
 bookmarks=true,bookmarksnumbered=false,bookmarksopen=false,
 breaklinks=false,pdfborder={0 0 0},pdfborderstyle={},backref=false,colorlinks=true]
 {hyperref}
\definecolor{myblue}{rgb}{0,0,0.6}
\hypersetup{colorlinks=true, linkcolor=myblue,citecolor=myblue,filecolor=myblue,urlcolor=myblue}

\newtheorem{theorem}{Theorem}[section]
\newtheorem{cor}[theorem]{Corollary}
\newtheorem{definition}[theorem]{Definition}
\newtheorem{notations}[theorem]{Notation}
\newtheorem{lem}[theorem]{Lemma}
\newtheorem{prop}[theorem]{Proposition}
\newtheorem{rem}[theorem]{Remark}

\newtheorem{remark}[theorem]{Remark}

\numberwithin{equation}{section}

\newcommand{\hsc}{\hbar}
\newcommand{\muin}{\mu_{\rm in}}
\newcommand{\muout}{\mu_{\rm out}}

\usepackage{soul}

\newcommand{\tfa}{\text{ for all }}

\newcommand{\tand}{\text{ and }}
\newcommand{\ton}{\text{ on }}
\newcommand{\tin}{\text{ in }}

\newcommand{\beq}{\begin{equation}}
\newcommand{\eeq}{\end{equation}}
\newcommand{\beqs}{\begin{equation*}}
\newcommand{\eeqs}{\end{equation*}}

\newcommand{\bbU}{\mathbb{U}}
\newcommand{\cT}{{\mathcal T}}
\newcommand{\cI}{{\mathcal I}}
\newcommand{\bcT}{\boldsymbol{\mathcal{T}}}

\newcommand{\be}{\boldsymbol{e}}
\newcommand{\bF}{\boldsymbol{F}}
\newcommand{\pdiff}[2]{\frac{\partial #1}{\partial #2}}
\newcommand*{\N}[1]{\left\|#1\right\|}
\newcommand{\bv}{\boldsymbol{v}}
\newcommand{\bu}{\boldsymbol{u}}
\newcommand{\imp}{\mathrm{imp}} 
\newcommand{\Imap}[4]{\cI_{\Gamma_{#1.#2} \ra \Gamma_{#3,#4}}}
\newcommand{\Imaps}[4]{\cI_{\Gamma_{#1}^{#2} \ra \Gamma_{#3}^{#4}}}
\newcommand{\ra}{\to}
\newcommand{\hgamma}{{\gamma}}
\newcommand{\hrho}{{\rho}}

 \newcommand{\wOmega}{\widehat{\domain}}
 \newcommand{\pwOmega}{\partial \wOmega}

\newcommand{\hL}{{L}}

\newcommand{\bit}{\begin{itemize}}
\newcommand{\eit}{\end{itemize}}
\newcommand{\ben}{\begin{enumerate}}
\newcommand{\een}{\end{enumerate}}

\newcommand{\leftsubscript}{l}
\newcommand{\middlesubscript}{i}%{r}
\newcommand{\rightsubscript}{r}%{i}
\newcommand{\goodmap}{\mathcal I^-_{1}}%1}}
\newcommand{\badmap}{\mathcal I^+_{1}}%1}}
\newcommand{\maptwonewplus}{\mathcal I^+_{2}}%\circled{2}}}
\newcommand{\maptwonewminus}{\mathcal I^-_{2}}%\circled{2}}}

\newcommand{\domain}{D}
\newcommand{\height}{{\mathsf h}}

\DeclareMathSymbol{\shortminus}{\mathbin}{AMSa}{"39}

\title[Sharp bounds on Helmholtz impedance-to-impedance maps]{Sharp bounds on Helmholtz impedance-to-impedance maps and application to overlapping domain decomposition}
\author{David Lafontaine} 
\address{
CNRS and Institut de Math\'ematiques de Toulouse, UMR5219; Universit\'e de Toulouse, CNRS; UPS, F-31062 Toulouse Cedex 9, France}\email{david.lafontaine@math.univ-toulouse.fr}

\author{Euan A.~Spence}
\address{Department of Mathematical Sciences, University of Bath, Bath, BA2 7AY, UK}
\email{E.A.Spence@bath.ac.uk}

\begin{document}

\maketitle

\begin{abstract}
We prove sharp bounds on certain impedance-to-impedance maps (and their compositions) for the Helmholtz equation with large wavenumber (i.e., at high-frequency) using semiclassical defect measures.
The paper \cite{GoGaGrLaSp:22} recently showed that the behaviour of these impedance-to-impedance maps (and their compositions) dictates the convergence of the 
parallel overlapping Schwarz domain-decomposition method with impedance boundary conditions on the subdomain boundaries. 
For a model decomposition with two subdomains and sufficiently-large overlap, the results of this paper combined with those in  \cite{GoGaGrLaSp:22} show 
that the parallel Schwarz method is power contractive, independent of the wavenumber.
For strip-type decompositions with many subdomains, the results of this paper show that the composite impedance-to-impedance maps, in general, behave ``badly'' with respect to the wavenumber; nevertheless, by proving results about the composite maps applied to a restricted class of data, we give insight into 
the wavenumber-robustness of the parallel Schwarz method observed in the numerical experiments in \cite{GoGaGrLaSp:22}.
\end{abstract}

\section{Introduction}\label{sec:1}

\subsection{Motivation and outline}\label{sec:motivation}

Over the last 30 years, there has been sustained interest in computing approximations to solutions of the Helmholtz equation $\Delta u + k^2 u=0$ with wavenumber $k\gg 1$ using domain-decomposition (DD) methods; see the recent review papers \cite{GaZh:19, GaZh:22}. However, it still remains an open problem to provide a $k$-explicit convergence theory, valid for arbitrarily-large $k$, for any practical DD method for computing approximations to Helmholtz solutions.

Working towards this goal, the paper \cite{GoGaGrLaSp:22} studied a parallel overlapping Schwarz method for the Helmholtz equation, where impedance boundary conditions are imposed on the subdomain boundaries. This method 
can be thought of as the overlapping analogue of the parallel non-overlapping method introduced in Despr\'es' thesis \cite{De:91, BeDe:97}, which was the first method in the Helmholtz context to demonstrate the benefits of using impedance boundary conditions
 on the subdomain problems (the analogous algorithm for Laplace's equation, with Robin boundary conditions on the subdomains, was introduced by P.-L.~Lions \cite{Li:90}). Indeed, with or without overlap, the parallel Schwarz method with Dirichlet boundary conditions on the subdomains need not converge when applied to the Helmholtz equation (see, e.g., \cite[\S2.2.1]{DoJoNa:15}), and is not even well posed if $k^2$ is a Dirichlet eigenvalue of the Laplacian on one of the subdomains. 
In contrast, the method with impedance boundary conditions is well posed, and always converges in the non-overlapping case by \cite[Theorem 1]{BeDe:97} (see, e.g., \cite[\S2.2.2]{DoJoNa:15}).

The paper \cite{GoGaGrLaSp:22} studied 
the parallel overlapping Schwarz method  at the continuous level, i.e., without discretisation of the subdomain problems, with then the follow-up paper \cite{GoGrSp:22} showing that, at least for certain 2-d decompositions and provided the discretisation is sufficiently fine, the discrete method inherits the properties of the method at the continuous level. 

The main result of \cite{GoGaGrLaSp:22} was the expression of the error-propagation operator (i.e., the operator describing how the error propagates from one iteration to the next) in 
terms of certain impedance-to-impedance maps. This result paved the way for $k$-explicit results about the convergence of the DD method to be obtained from $k$-explicit bounds on the norms of the  impedance-to-impedance maps (and/or their compositions), and the present paper  provides such $k$-explicit bounds for a model set-up in 2-d.

The structure of the paper is as follows. \S\ref{sec:1.2} defines the impedance-to-impedance maps for a certain 2-d model set-up, and \S\ref{sec:1.3} states bounds on these maps that are sharp in the limit $k\to\infty$. \S\ref{sec:1.4} states sharp bounds on compositions of these maps.
\S\ref{sec:2} discusses the implications of the main results for the parallel Schwarz method, recapping the necessary material from \cite{GoGaGrLaSp:22}; \S\ref{sec:2} also recaps other literature on impedance-to-impedance maps in the context of domain-decomposition, including the recent paper \cite{BeCaMa:22} (see Remark \ref{rem:literature}).
\S\ref{sec:3} recaps known material about semiclassical defect measures (mainly from \cite{Miller}). \S\ref{sec:4} proves propagation results at the level of measures for the models we are interested in.
\S\S\ref{sec:6}-\S\ref{sec:7} prove the main results. 
\S\ref{sec:5} shows wellposedness results for our models.

\subsection{Definition of the impedance-to-impedance maps}\label{sec:1.2}

Given cartesian coordinates $(x_1,x')$ in $\mathbb R^2$ and $\height , \mathsf {d_{\leftsubscript}}, \mathsf {d_{\rightsubscript}} > 0$, we let
\begin{gather*}
\domain := (0,\mathsf {d_{\leftsubscript}} + \mathsf {d_{\rightsubscript}}) \times (0,\height)\\
\Gamma_{\leftsubscript} := \{0\} \times [0,\height],\qquad \Gamma_{\middlesubscript} := \{\mathsf {d_{\leftsubscript}}\} \times [0,\height], \qquad \Gamma_{\rightsubscript} := \{\mathsf {d_{\leftsubscript}} + \mathsf {d_{\rightsubscript}} \} \times [0,\height], \\ 
\Gamma_t := (0,\mathsf {d_{\leftsubscript}} + \mathsf {d_{\rightsubscript}}) \times \{ \height  \}, \qquad \Gamma_b := (0,\mathsf {d_{\leftsubscript}} + \mathsf {d_{\rightsubscript})} \times \{ 0 \}, \\
\Gamma_s:= \Gamma_t \cup \Gamma_b \cup \Gamma_{\rightsubscript}, \qquad\Gamma'_s := \Gamma_t \cup \Gamma_b;
\end{gather*} 
see Figure \ref{fig:model}.

\newcommand{\drawcell}{

\draw [thick] (0,0) -- (5,0)  node[midway, below] {$\Gamma_b$};
\draw [thick] (5,0) -- (5,3)  node[midway, right] {$\Gamma_{\rightsubscript}$};
\draw [thick] (5,3) -- (0,3) node[midway, above] {$\Gamma_t$};
\draw [thick] (0,3) -- (0,0) node[midway, left] {$\Gamma_{\leftsubscript}$};

\draw [thick, dashed]  (3,0) -- (3,3)  node[midway, right] {$\Gamma_{\middlesubscript}$} ;

\draw [<->] (1/20,3-1/5) -- (3-1/20,3-1/5)  node[midway, below] {$\mathsf {d_{\leftsubscript}}$};
\draw [<->] (3+1/20,3-1/5) -- (5-1/20,3-1/5)  node[midway, below] {$\mathsf{d_{\rightsubscript}}$};
\draw [<->] (1/5,1/20) -- (1/5, 3-1/20)  node[midway, right] {$\height $};
}

\begin{figure}[h!]
\begin{center}
\begin{tikzpicture}
\drawcell
\end{tikzpicture}
\caption{The domain $\domain$, the boundaries $\Gamma_\leftsubscript, \Gamma_\rightsubscript, \Gamma_t$, and $\Gamma_b$, and the interior interface $\Gamma_\middlesubscript$}
\label{fig:model}
\end{center}
\end{figure}

We are interested in the two following model problems.

\noindent\begin{minipage}{0.5\linewidth}
\begin{equation} \label{eq:model_cell} \tag{Model 1}
\begin{cases}
(\Delta + k^2) u = 0 \text{ in } \domain \\
(\frac{1}{i}\partial_{x_1} + k) u = g \; \text{ on } \Gamma_{\leftsubscript}, \\
\text{$u$ is outgoing near $\Gamma_s$},
\end{cases}
\end{equation}
\end{minipage}%
\begin{minipage}{0.5\linewidth}
\begin{equation} \label{eq:model_cell_timp} \tag{Model 2}
\begin{cases}
(\Delta + k^2) u = 0 \text{ in } \domain \\
(\frac{1}{i}\partial_{x_1} + k) u = g \; \text{ on } \Gamma_{\leftsubscript}, \\
(-\frac{1}{i}\partial_{x_1} + k) u = 0 \; \text{ on } \Gamma_{\rightsubscript}, \\
\text{$u$ is outgoing near $\Gamma'_s$},
\end{cases}
\end{equation}
\end{minipage}\par\vspace{\belowdisplayskip}
\noindent where the outgoingness near $\Gamma_s$ (resp.~$\Gamma'_s$) is defined in Definition \ref{def:outgo} below. 
Since the outgoing boundary condition is a high-frequency ($k\rightarrow \infty$) assumption, there might exist more than one
solution to such a model for a given $k$; however, all solutions will have the same high-frequency behaviour. 
Indeed, 
for both models, an admissible 
(in the sense of Definition \ref{def:adm} below)
solution operator $g \mapsto u$ exists,
and two solutions coming from distinct admissible solution operators always coincide modulo $O(k^{-\infty})$ (see Lemmas \ref{lem:modelwp} and \ref{lem:modelHF} below). Since we are interested in the high-frequency 
behaviour of the solutions, we therefore fix from now on an admissible solution operator to both models.

Observe that in \ref{eq:model_cell_timp}, the impedance data $(-\frac{1}{i}\partial_n +k)u$ is specified on both $\Gamma_\leftsubscript$ and $\Gamma_\rightsubscript$, where $\partial_n$ is the outward normal derivative to $\domain$. 
 If $g \in L^2(\Gamma_{\leftsubscript})$ and $u$ is the associated solution of (\ref{eq:model_cell}) in $\domain$, we let
$$
\goodmap(\height , \mathsf {d_{\leftsubscript}}, k) g := \gamma_{\Gamma_{\middlesubscript}}\left(\frac{1}{i}\partial_{x_1} - k\right)u, \hspace{0.6cm}\badmap(\height , \mathsf {d_{\leftsubscript}}, k)  g := \gamma_{\Gamma_{\middlesubscript}}\left(\frac{1}{i}\partial_{x_1} + k\right)u;
$$
i.e., $\goodmap$ and $\badmap$ are the two different impedance traces on $\Gamma_\middlesubscript$, with the plus/minus superscripts correspond to the plus/minus in the impedance traces,
where $\gamma_{\Gamma_{\middlesubscript}}$ denotes the trace operator onto $L^2(\Gamma_{\middlesubscript})$. In the same way, if $g \in L^2(\Gamma_{\leftsubscript})$, and $u$ is the associated solution of (\ref{eq:model_cell_timp}) in $\domain $, we let
$$
\maptwonewminus (\height , \mathsf {d_{\leftsubscript}}, \mathsf {d_{\rightsubscript}}, k) g := \gamma_{\Gamma_{\middlesubscript}}
\left(\frac{1}{i}\partial_{x_1} - k\right)u,
\hspace{0.6cm}\maptwonewplus (\height , \mathsf {d_{\leftsubscript}}, \mathsf {d_{\rightsubscript}}, k) g := \gamma_{\Gamma_{\middlesubscript}}\left(\frac{1}{i}\partial_{x_1} + k\right)u.
$$

\subsection{Upper and lower bounds on the impedance operators}\label{sec:1.3}
\begin{theorem}[Upper and lower bounds for (\ref{eq:model_cell})] \label{th:model1}
 Let $\height , \mathsf {d_\leftsubscript} > 0$ and let $\theta_{\rm max}\in (0,\pi/2)$ be defined by
$$
\theta_{\rm max}(\height , \mathsf {d_\leftsubscript}) := \arctan \frac {\height } {\mathsf{d_\leftsubscript}}.
$$
$\;$
\begin{enumerate}
\item For any $\epsilon>0$, there exists $k_0(\epsilon)>0$ such that, for all $k\geq k_0$
$$
\Vert \goodmap(\height , \mathsf {d_{\leftsubscript}}, k)  \Vert_{L^2 \rightarrow L^2}\leq   \frac{1-\cos( \theta_{\rm max})}{1 + \cos (\theta_{\rm max})} + \epsilon
$$
In addition,
$$
\limsup_{k \rightarrow \infty }\Vert \goodmap(\height , \mathsf {d_{\leftsubscript}}, k)  \Vert_{L^2 \rightarrow L^2} \geq  \frac{1-\cos( \theta_{\rm max})}{1 + \cos( \theta_{\rm max})}.
$$
\item For any $\epsilon>0$, there exists $k_0(\epsilon)>0$ such that, for all $k\geq k_0$
$$
\Vert \badmap(\height , \mathsf {d_{\leftsubscript}}, k)  \Vert_{L^2 \rightarrow L^2}\leq 1 + \epsilon
$$
In addition,
$$
\limsup_{k \rightarrow \infty }\Vert \badmap(\height , \mathsf {d_{\leftsubscript}}, k)  \Vert_{L^2 \rightarrow L^2} \geq  1.
$$
\end{enumerate}
\end{theorem}

 \subsubsection*{The heuristic behind Theorem \ref{th:model1}.} In Fourier space, $\frac{1}{i}\partial_{x_1} \pm k$ acts as multiplication by $\xi_1 \pm k$, where $\xi_1$ is the first component of the dual variable to $x$ 
 (that is, the frequency); 
thus $\goodmap$ is governed by $(\xi_1-k)/(\xi_1+k)$ and $\badmap$ is governed by $(\xi_1+k)/(\xi_1+k)$.
This ratio is one in the case of $\badmap$, leading to the lower bound in Part 2 of Theorem \ref{th:model1}. To bound the ratio sharply in the case of $\goodmap$, we observe the following:~since the solution is outgoing on the borders of the cell, the mass reaching $\Gamma_{\middlesubscript}$ comes necessarily from $\Gamma_{\leftsubscript}$. To go from $\Gamma_{\leftsubscript}$ to $\Gamma_{\middlesubscript}$ at frequency $k$, this mass must travel with a direction $\xi = k(\cos \theta, \sin \theta)$, where $0 \leq \theta \leq \theta_{\rm max}$ (with mass travelling from the very bottom of $\Gamma_\leftsubscript$ to the very top of $\Gamma_\middlesubscript$ traveling with angle $\theta_{\rm max}$).
This direction $\xi$ therefore satisfies
\beqs
|\xi_1 -k| = k|\cos\theta-1|= k\big(1-\cos\theta\big) \leq k\big(1-\cos(\theta_{\max})\big)
\eeqs
and
\beqs
|\xi_1+k| = k |\cos\theta+1|= k\big(1+\cos\theta\big) \geq k\big(1+\cos(\theta_{\max})\big),
\eeqs  
  from which the upper bounds in Theorem \ref{th:model1} follow.
Our proof implements these heuristic arguments in a rigorous way, using so-called \emph{semiclassical defect measures}. In particular, to prove the lower bounds in Theorem \ref{th:model1}, we take data $g$ 
from a so-called \emph{coherent state}, chosen so that all the mass of the solution concentrates on a ray 
 coming from one point $x_0\in\Gamma_{\leftsubscript}$ and travels in one direction $\xi_0 = k(\cos \theta_0, \sin \theta_0)$; we then take 
 $x_0$ to be very close to the origin (i.e., the bottom of $\Gamma_\leftsubscript$) and  $\theta_0$ to be as close as $\theta_{\rm max}$ as possible.

{
%\begin{remark*}
We highlight that the right-hand side of the upper bound on $\mathcal I^-_{1}$ \emph{decreases} as $\mathsf{d_\leftsubscript}$ increases, and $\|\mathcal{I}^-_{1}\|$ can be made arbitrarily small for sufficiently large $\mathsf{d_\leftsubscript}$; this property is used in the application to domain decomposition, where $\mathsf{d_\leftsubscript}$ corresponds to the overlap of the subdomains.
%\end{remark*}
}

\begin{theorem}[Lower bounds for (\ref{eq:model_cell_timp})] \label{th:model2} Let $\height , \mathsf {d_\leftsubscript}, \mathsf{d_{\rightsubscript}} > 0$. Then,
$$
\limsup_{k \rightarrow \infty }\Vert \mathcal I^\pm_{2}(\height , \mathsf {d_\leftsubscript}, \mathsf{d_{\rightsubscript}}, k)  \Vert_{L^2 \rightarrow L^2} \geq  1.
$$
\end{theorem}

 \subsubsection*{The heuristic behind Theorem \ref{th:model2}.} 
The difference between Model 2 and Model 1 is that now rays are reflected from $\Gamma_\rightsubscript$. Arguing as above, the map $\maptwonewplus$ is bounded below by one on rays travelling directly from $\Gamma_\leftsubscript$ to $\Gamma_\middlesubscript$. Provided these rays are not horizontal (i.e., $\xi_2\neq 0$), their subsequent reflections from $\Gamma_\rightsubscript$ and $\Gamma_\leftsubscript$ do not interfere with the initial ray (they move either up or down in the vertical direction, until being absorbed by the outgoing conditions at the top and bottom), and thus $\maptwonewplus$ is bounded below by one.

To see that  $\maptwonewminus$ is bounded below by one, we need to consider the first reflected ray. 
 When the wave $\exp(i \xi\cdot x)$ with $|\xi|=k$ hits the impedance boundary on the right, a reflected wave $R \exp(-i \xi_1 x_1 + i \xi_2 x_2)$ is created, where the reflection coefficient $R := (\xi_1 -k)/(\xi_1+k)$. For a wave $\exp(-i \xi_1 x_1 + i \xi_2 x_2)$, $\maptwonewminus$ is governed by $(-\xi_1-k)/(-\xi_1+k)$, and thus the contribution to $\maptwonewminus$ from the first reflected ray is 
 $|R(-\xi_1-k)/(-\xi_1+k)| = 1$.
 As discussed in the previous paragraph, provided that $\xi_2\neq 0$, further reflections do not interfere with this first reflected ray, and thus 
 $\maptwonewminus$ is bounded below by one.

\subsection{The behaviour of the composite impedance map in (\ref{eq:model_cell_timp})}\label{sec:1.4}

Given $\height , \mathsf {d_\leftsubscript}^+, \mathsf {d_{\rightsubscript}}^+, \mathsf {d_\leftsubscript}^-, \mathsf {d_{\rightsubscript}}^- > 0$, we 
consider arbitrary compositions of the two maps 
\beq\label{eq:two_maps}
\mathcal I^{-}_2 (\height , \mathsf {d_\leftsubscript}^{-}, \mathsf {d_{\rightsubscript}}^{-}, k)
\quad\tand\quad
\mathcal I^{+}_2 (\height , \mathsf {d_\leftsubscript}^{+}, \mathsf {d_{\rightsubscript}}^{+}, k);
\eeq
we allow the two maps $\cI_2^-$ and $\cI_2^+$ to have different  arguments $\mathsf d_{\leftsubscript}$ and $\mathsf d_{\rightsubscript}$ because of the application of these results in domain decomposition -- see Remark \ref{rem:twod} below.

An arbitrary composition of the two maps \eqref{eq:two_maps} can be written as the following:~given $n\geq 0$ and 
 $\sigma \in \{ +, -\}^n$, let 
$$
\mathcal I ^\sigma(\height , \mathsf{d^+_\leftsubscript}, \mathsf {d_{\rightsubscript}}^+, \mathsf{d^-_\leftsubscript}, \mathsf {d_{\rightsubscript}}^-, k) := \prod_{\ell=0, \cdots, n} \mathcal I^{\sigma(\ell)}_2 (\height , \mathsf {d_\leftsubscript}^{\sigma(\ell)}, \mathsf {d_{\rightsubscript}}^{\sigma(\ell)}, k),
$$
where the product denotes composition of the maps.

In addition, for any $\lambda >0$, we define the projection $\lambda$-away from zero frequency  $\Pi^k_\lambda$  as
$$
\Pi^k_\lambda g  := \mathcal F_k^{-1} \Big( \Big(1 - \psi\Big(\frac{\cdot}{\lambda}\Big)\Big) \mathcal F_kg\Big),  \hspace{0.6cm}\mathcal F_k g (\zeta) := \frac{k}{2 \pi} \int e^{- ik y \zeta} g(y) \; dy,
$$
where $\psi \in C^\infty_c(\mathbb R; [0,1])$ is equal to one on $[-1,1]$, 
and $\mathcal F_k$ is the Fourier transform at scale $k$. 
Written with $\mathcal F_1$, the non-scaled Fourier-transform,
$$
\Pi^k_\lambda g  = \mathcal F_1^{-1} \Big( \Big(1 - \psi\Big(\frac{\cdot}{\lambda k}\Big)\Big) \mathcal F_1g\Big);
$$
that is, $\Pi^k_\lambda$ is a projection $\lambda k$-away from zero in the Fourier variable.

The projection $\Pi^k_\lambda$ is applied below to impedance data; the heuristic interpretation of this is the following:
~since a Helmholtz solution is, in the high-frequency limit, supported in Fourier space where $|\xi_1|^2+|\xi'|^2 = k^2$ (with $\xi_1$ the dual variable of $x_1$ and $\xi'$ of $x'$), truncating the impedance data $\lambda k$-away from $\xi' = 0$ produces a solution supported at high frequencies where $|\xi_1|^2 \leq k^2(1-\lambda^2)$, hence away from the horizontal direction corresponding to $|\xi_1|^2 = k^2$.

\begin{theorem}[The composite impedance map] \label{th:compo} Let
$\height , \mathsf{d^+_\leftsubscript}, \mathsf {d_{\rightsubscript}}^+, \mathsf{d^-_\leftsubscript}, \mathsf {d_{\rightsubscript}}^- > 0$.
\begin{enumerate}
\item \label{i:compo:lower} For any $n \geq 1$ and any $\sigma \in \{ +, -\}^n$, 
$$
\limsup_{k \rightarrow \infty }\Vert \mathcal I^\sigma(\height , \mathsf{d^+_\leftsubscript}, \mathsf {d_{\rightsubscript}}^+, \mathsf{d^-_\leftsubscript}, \mathsf {d_{\rightsubscript}}^-, k)  \Vert_{L^2 \rightarrow L^2} \geq  1.
$$
\item \label{i:compo:upper} Let $\sigma \in \{+, -\}^{\mathbb N}$ and, for any $n\geq 1$, $\sigma_n := (\sigma(1), \cdots, \sigma(n)) \in \{ +, - \}^n$. 
Given $\lambda>0$, let $n_0(\lambda) := \height  (\min (\mathsf{d^+_\leftsubscript}, \mathsf{d^-_\leftsubscript}))^{-1}\lambda^{-1}\sqrt{1-\lambda^2}$.
Then 
$$
\tfa  n \geq n_0(\lambda), \hspace{0.5cm} \lim_{k \rightarrow \infty }\Vert \mathcal I ^{\sigma_n}(\height , \mathsf{d^+_\leftsubscript}, \mathsf {d_{\rightsubscript}}^+, \mathsf{d^-_\leftsubscript}, \mathsf {d_{\rightsubscript}}^-, k) \Pi^k_\lambda \Vert_{L^2 \rightarrow L^2} = 0 .
$$
\end{enumerate}
\end{theorem}

\subsubsection*{The heuristic behind Theorem \ref{th:compo}}
The main idea behind Theorem \ref{th:compo} is that, in the high-frequency limit, the impedance-to-impedance map associated with (\ref{eq:model_cell_timp}) pushes the mass emanating from $\Gamma_{\leftsubscript}$ with an angle $\theta$ to the horizontal up and down by a distance proportional to $\theta^{-1}$, while preserving its mass. 
Therefore, if $g$ creates a Helmholtz solution emanating from $\Gamma_{\leftsubscript}$ with angles $>0$ 
 to the horizontal, all its mass is pushed off the domain in a finite number of iterations.
 After applying $\Pi^k_\lambda$, the data creates a solution emanating from $\Gamma_\leftsubscript$ with angles $\geq \arctan \frac{\lambda}{\sqrt{1-\lambda^2}}>0$;~hence the high-frequency nilpotence of the composite impedance map, i.e., Part  (\ref{i:compo:upper}) of the theorem. On the other hand, the same idea allows us to construct the lower bound in Part  (\ref{i:compo:lower}):~taking coherent-state data $g$ creating a Helmholtz solution concentrating all its mass on an arbitrarily small angle to the horizontal, the image by the composite impedance map after the corresponding, arbitrarily high, number of iterations, is still in the domain and hence has order one mass. 
 
\subsection{The semiclassical notation and definition of outgoingness} \label{sec:1.5}
It is convenient to work with the semiclassical small parameter $\hsc := k^{-1}$. In addition, we let $D_\bullet := \frac 1i \partial_\bullet$. Then, (\ref{eq:model_cell}) and  (\ref{eq:model_cell_timp}) become, with $g$ replaced by $\hsc g$,

\noindent\begin{minipage}{0.5\linewidth}
\begin{equation} \label{eq:modelh} \tag{M1}
\begin{cases}
(-\hsc^2 \Delta - 1) u = 0 \text{ in }\domain \\
(\hsc D_{x_1}+1) u = g \; \text{ on } \Gamma_{\leftsubscript}, \\
\text{$u$ is outgoing near $\Gamma_s$}.
\end{cases}
\end{equation}
\end{minipage}%
\begin{minipage}{0.5\linewidth}
\begin{equation} \label{eq:modelh_timp} \tag{M2}
\begin{cases}
(-\hsc^2 \Delta - 1) u = 0 \text{ in }\domain \\
(\hsc D_{x_1}+1) u = g \; \text{ on } \Gamma_{\leftsubscript}, \\
\text{$u$ is outgoing near $\Gamma'_s$}, \\
(- \hsc D_{x_1}+1) u = 0 \; \text{ on } \Gamma_{\rightsubscript}. \\
\end{cases}
\end{equation}
\end{minipage}\par\vspace{\belowdisplayskip}
We can now define the outgoingness near the border of the cell:
\begin{definition} \label{def:outgo}
We say that a $\hsc$-family of solutions $u$ to $(-\hsc^2 \Delta - 1) u = 0$ in $\domain$ is \emph{outgoing} near $\Gamma \in \{\Gamma_s, \Gamma'_s \}$ if there exists an open set $\domain^+ \supset \domain \cup \Gamma$ (independent of $\hsc>0$) with $\partial \domain^+ \cap(\partial\domain \backslash \Gamma) = \partial\domain \backslash \Gamma$, such that $u$ can be extended to a $\hsc$-tempered
 solution of $(-\hsc^2 \Delta - 1) u = 0$ in $\domain^+$ and
$$
\forall(x, \xi) \in \operatorname{WF}_\hsc u \cap \{ x \in \Gamma \}, \hspace{0.3cm} \xi \cdot n(x) \geq 0,
$$
where the normal $n(x)$ points 
out of the domain $\domain$ depicted in Figure \ref{fig:model}.
\end{definition}
The wavefront set $\operatorname{WF}_\hsc$ of an $\hsc$-tempered family of functions is defined in Definition \ref{def:WF} below.
It describes where the non-negligible mass of an $\hsc$-dependent family of  functions lies in
 phase-space (that is, in both position and direction)
 in the high-frequency limit $\hsc \rightarrow 0$;
Definition \ref{def:outgo} therefore means 
the solution $u$ has only mass pointing outside the
cell -- hence outgoing.

\subsection{Wellposedness results}\label{sec:wp}

We say that $S = S(\hsc) : g\in L^2(\Gamma_l) \mapsto u \in H^1(D)$ is a solution operator associated to model (\ref{eq:modelh})/(\ref{eq:modelh_timp}), if $S$ is linear and for any $g\in L^2(\Gamma_l)$, $u(\hsc) := S(\hsc)g$ is solution to model (\ref{eq:modelh})/(\ref{eq:modelh_timp}). The following results are shown in \S \ref{sec:5}.

\begin{lem} \label{lem:modelwp}
There exists an admissible (in the sense of Definition \ref{def:adm}) solution operator to (\ref{eq:modelh})/(\ref{eq:modelh_timp}).
\end{lem}

The admissibility condition of Definition \ref{def:adm} corresponds to requiring that any solution can be extended in a slightly bigger 
domain, where it is bounded, and has bounded traces where an impedance boundary condition is imposed ($\Gamma_l$ for (\ref{eq:modelh}) and $\Gamma_l\cup \Gamma_r$ for (\ref{eq:modelh_timp})). 
All admissible solution operators have the same high-frequency behavior in the following sense, where 
$\Vert f \Vert_{H^s_\hsc} := \Vert f \Vert_{L^2} + \hsc^s \Vert f \Vert_{\dot H^s}$. 

\begin{lem} \label{lem:modelHF}
If $S_1$, $S_2$ are two admissible solutions operators, then, for any bounded $g \in L^2(\Gamma_l)$, any $N>0$, and any $\chi \in C^\infty_c(D)$, there exists $C_N>0$ such that $\Vert \chi(S_1 g - S_2 g) \Vert_{H^{N}_\hsc} \leq C_N \hsc^N$.
\end{lem}

\section{Implications of the main results for the parallel overlapping Schwarz method}\label{sec:2}

The plan of this section is to 
\begin{itemize}
\item define the parallel overlapping Schwarz method studied in \cite{GoGaGrLaSp:22} (\S\ref{sec:2.1}),
\item summarise the performance of the parallel overlapping Schwarz method as illustrated in the numerical experiments in \cite{GoGaGrLaSp:22} (\S\ref{sec:summary}),
\item show how impedance-to-impedance maps govern the behaviour of the error of this method (\S\S\ref{sec:2.2}-\ref{sec:2.4}) and define these impedance-to-impedance maps for general decompositions (\S\ref{sec:2.5}), 
\item show how the maps of Models 1 and 2 in \S\ref{sec:1.2} are the relevant impedance-to-impedance maps for 2-d strip decompositions when the boundary condition on the whole domain (approximating the Sommerfeld radiation condition) is the outgoing condition described in \S\ref{sec:1.5} (\S\ref{sec:2.6}-\S\ref{sec:2.8}), and 
\item 
explain the relevance of Theorems \ref{th:model1}, \ref{th:model2}, and \ref{th:compo} to the analysis of the parallel Schwarz method (\S\ref{sec:2.9}).
\end{itemize}

\subsection{Definition of the parallel overlapping Schwarz method studied in \cite{GoGaGrLaSp:22}}\label{sec:2.1}

The paper \cite{GoGaGrLaSp:22} considers the Helmholtz interior impedance problem, i.e., given a bounded domain $\Omega\subset \mathbb{R}^d$ $d\geq 2$, $f\in L^2(\Omega)$, and $g\in L^2(\partial \Omega)$, find $u\in H^1(\Omega)$ satisfying 
\begin{align} \label{eq:Helmholtz}
(\Delta +k^2) u =-f \quad \text{on } \Omega \quad\tand \quad\left(\frac{1}{i} \frac{    \partial } {\partial n} - k  \right)u = g \quad \text{on }
     \partial \Omega
  \end{align}  
(where $\partial/\partial n$ denotes the outward normal derivative on $\partial \Omega$)  and considers its solution via the following parallel overlapping Schwarz method with impedance transmission conditions. Let  $\{\Omega_j\}_{j=1}^N$  
form  an overlapping cover of $\Omega$ with each $\Omega_j \subset \Omega$ and Lipschitz polyhedral. 
If  $u$ solves (\ref{eq:Helmholtz}), then $u_j:=u|_{\Omega_j}$ satisfies
\begin{align}
    (\Delta + k^2)u_j & = - f  \quad  &\text{in } \  \Omega_j, \label{eq11} \\
  \left(
  \frac{1}{i} 
  \frac{\partial }{\partial n_j} -  k \right) u_j & = \left(\frac{1}{i} \frac{\partial }{\partial n_j} -  k \right) u \quad  &\text{on }  \ \partial \Omega_j\backslash \partial \Omega,  \label{eq12} \\
  \left(
  \frac{1}{i} 
  \frac{\partial }{\partial n_j} -  k \right) u_j & = g  \quad  &\text{on } \  \partial \Omega_j  \cap \partial \Omega,\label{eq13}
\end{align}
where $\partial /\partial n_j$ denotes the outward normal derivative on $\partial \Omega_j$.
\footnote{The impedance boundary conditions in \cite{GoGaGrLaSp:22} are written in the form $\partial_n - i k$ (with this form more commonly-used in numerical analysis). In this section we write the results of \cite{GoGaGrLaSp:22} using the impedance condition $\frac{1}{i} \partial_n -k$; the two conventions are equivalent up to multiplication/division by $i$ of the data $g$ on $\partial \Omega$.
}

Let  $\{\chi_j\}_{j = 1}^N$ be such that $\chi_j\in C^{1,1}(\Omega;[0,1])$,
$\chi_j\equiv 0$ in $\Omega\cap (\overline{(\Omega_j)^c})$ (and thus, in particular, on $\partial\Omega_j\setminus \partial \Omega$) and $\sum_{j=1}^N \chi_j(x) =1$ for all $x\in \overline{\Omega}$.
The parallel Schwarz method is:~given an iterate $u^n$ defined on $\Omega$, let $u^{n+1}_j$ be the solution of 
\begin{align}
  (\Delta + k^2)u_j^{n+1}  & =  -f  \quad  &\text{in } \  \Omega_j,  \label{eq21}\\
  \left(
  \frac{1}{i} 
  \frac{\partial }{\partial n_j} -  k \right) u_j^{n+1}  & = \left(\frac{1}{i} \frac{\partial }{\partial n_j} -  k \right) u^n \quad  &\text{on }  \ \partial \Omega_j\backslash \partial \Omega, \label{eq22}  \\
    \left(
    \frac{1}{i} 
    \frac{\partial }{\partial n_j} -  k \right) u_j^{n+1}  & = g  \quad  &\text{on } \  \partial \Omega_j  \cap \partial \Omega;  \label{eq23} 
\end{align}
finally, let 
\begin{align} \label{star}
  u^{n+1} := \sum_\ell \chi_\ell  u_\ell ^{n+1}.
\end{align}
This method is well-defined since 
if $u^n \in U(\Omega)$ then $u^{n+1}\in U(\Omega)$, where
\beqs
 U (\Omega) := \big\{ u \in H^1(\Omega): \,  (\Delta +k^2) u \in L^2(\Omega), \, (-i\partial  / \partial n -  k) u 
   \in L^2(\partial \Omega) \big\};
\eeqs
see \cite[Theorem 2.12]{GoGaGrLaSp:22}.

We highlight that the impedance boundary condition enters in \eqref{eq11}-\eqref{star} in two ways
\begin{enumerate}
\item as the boundary condition on $\partial \Omega$ (\eqref{eq13}, \eqref{eq23}), and  
\item as the boundary conditions on the subdomains \eqref{eq22}.
\end{enumerate}
Regarding 1:~the motivation for imposing an impedance boundary condition on $\partial \Omega$ is that it is the simplest-possible approximation to the Sommerfeld radiation condition, and the interior impedance problem is a ubiquitous model problem in the numerical analysis of the Helmholtz equation (see, e.g., the discussion and references in \cite[\S 1.1]{GLS1}). However, the analysis in \cite{GoGaGrLaSp:22} is, in principle, applicable to other boundary conditions on $\partial \Omega$, and we discuss this further in Remark \ref{rem:other_bc} below.

Regarding 2:~as discussed in \S\ref{sec:motivation}, the advantage of using impedance boundary conditions on the subdomains was recognised in Despr\'es' thesis \cite{De:91, BeDe:97}, and \eqref{eq11}-\eqref{star} is the overlapping analogue of the non-overlapping method in \cite{De:91, BeDe:97} (see, e.g., the discussion in  \cite[\S2.3]{DoJoNa:15}).

We see later (in \S\ref{sec:2.8}) that the impedance-to-impedance maps in Theorems \ref{th:model1}-\ref{th:compo} are those dictating the behaviour of the parallel Schwarz method in the idealised case where the boundary condition on $\partial\Omega$ is the outgoing condition. The rationale for considering this idealised case is that it allows us to focus on the impedance boundary conditions imposed in the domain decomposition method itself (i.e., in Point (2) above), and ignore the influence of the impedance boundary condition imposed as an approximation of the Sommerfeld radiation condition (i.e., in Point (1) above).
{The implications of Theorems \ref{th:model1}-\ref{th:compo} for the parallel overlapping Schwarz method studied in \cite{GoGaGrLaSp:22} (i.e., with impedance boundary conditions everywhere) are discussed further in \S\ref{sec:2.9}.
}

\subsection{Summary of the numerical experiments in \cite{GoGaGrLaSp:22} on the performance of the parallel Schwarz method}\label{sec:summary}

The experiments in \cite[\S6]{GoGaGrLaSp:22} considered two situations.
\ben
\item Strip decompositions (described in \S\ref{sec:2.6} below) in 2-d rectangular domains with height one and maximum length $64/3$ with $k\in [20,80]$
 (so that at the highest $k$ there were approximately $267$ wavelengths in the domain).
\item Uniform (``checkerboard'') and non-uniform (created by the mesh partitioning software METIS) decompositions of the 2-d unit square with $k\in [40,160]$ (so that at the highest $k$ there were approximately $25$ wavelengths in the domain).
\een
The experiments in \cite[\S6]{GoGaGrLaSp:22} showed the following three features of the parallel Schwarz method.
\bit
\item[(a)] For a fixed number of subdomains with fixed overlap proportional to the subdomain length, the number of iterations required to achieve a fixed error tolerance decreases as $k$ increases (in the ranges above) -- this was shown for the strip decompositions in \cite[Experiment 6.2 and Table 2]{GoGaGrLaSp:22} and for the square in \cite[Tables 7, 8, 10, 11]{GoGaGrLaSp:22} (with a similar result seen for a different parallel DD method in \cite[Table 3]{GrSpZo:20}).
\item[(b)] For the strip decomposition with fixed number of subdomains, the convergence rate of the method increases as the length of subdomains increases with the overlap proportional to the subdomain length (so that the overall length of the domain increases) \cite[Experiment 6.1 and Figure 5]{GoGaGrLaSp:22}.
\item[(c)] For the strip decomposition with an increasing number of subdomains and fixed subdomain length and overlap (so the length of domain increases), at fixed $k$, one needs roughly $O(N)$ iterations to obtain a fixed error tolerance; see \cite[Experiment 6.2]{GoGaGrLaSp:22}.
\eit
We highlight that using the method with a fixed number of subdomains, as in (a), is not completely practical, since the subproblems have the same order of complexity as the global problem. Nevertheless, this situation provides a useful starting point for methods based on recursion; see the discussion in \cite[Section 1.4]{GrSpZo:20}. Furthermore, we see in \S\ref{sec:2.9} how the results of the present paper imply that analysing the method even in this idealised case is very challenging.

\subsection{The error propagation operator $\bcT$}\label{sec:2.2}

We consider the vector of errors
 \begin{align} \label{31} \mathbf{e}^n = (e_1^n, e_2^n, \ldots e_N^n)^\top , \quad \text{where} \quad  e_\ell^n := u_\ell - u_\ell^n  = u\vert_{\Omega_\ell}  - u_\ell^n,  \quad \ell = 1, \ldots, N. \end{align}  
By the definition of $u^n$ \eqref{star} and the fact that $\{\chi_\ell\}_{\ell=1}^N$ is a partition of unity,
\begin{align} \label{eq:global} 
e^n:= u-u^n =\sum_\ell \chi_\ell   u\vert_{\Omega_\ell }  -\sum_\ell    \chi_\ell u_\ell^n = \sum_\ell \chi_\ell  e_\ell^n . 
\end{align}
Thus, subtracting  \eqref{eq21}-\eqref{eq23} from  \eqref{eq11}-\eqref{eq13}, we obtain 
\begin{align}
  (\Delta + k^2)e_j^{n+1}  & = 0  \quad  \text{in } \  \Omega_j,  \label{eq31}\\
  \left(\frac{1}{i}\frac{\partial }{\partial n_j} -  k \right)  e_j^{n+1}  & = \left(\frac{1}{i}\frac{\partial }{\partial n_j} -  k \right) e^n \ = \ \sum_\ell \left(\frac{1}{i}\frac{\partial }{\partial n_j} -  k \right) \chi_\ell e_\ell^n,   \quad  \text{on }  \ \partial \Omega_j\backslash \partial \Omega, \label{eq32}  \\
    \left(\frac{1}{i}\frac{\partial }{\partial n_j} -  k \right) e_j^{n+1}  & = 0  \quad  \text{on } \  \partial \Omega_j  \cap \partial \Omega.   \label{eq33} 
\end{align} 
The map from $e^n$ to $e^{n+1}$ can be written in a convenient way using the operator-valued matrix $\bcT = (\cT_{j,\ell})_{j,\ell = 1}^N$, defined as follows.  For $v_\ell \in U(\Omega_\ell)$,  and any $j\in \{1,\ldots,N\}$, let $\cT_{j,\ell}v_\ell \in U(\Omega_j)$ be the solution of 
\begin{align}
  (\Delta + k^2)(\cT_{j,\ell} v_\ell)  & = 0  \quad  \text{in } \  \Omega_j,  \label{eq41}\\
  \left(\frac{1}{i}\frac{\partial }{\partial n_j} -  k \right) (\cT_{j,\ell} v_\ell) & = \left(\frac{1}{i}\frac{\partial }{\partial n_j} -  k \right) (\chi_\ell v_\ell) \quad  \text{on }  \ \partial \Omega_j\backslash \partial \Omega, \label{eq42}  \\
    \left(\frac{1}{i}\frac{\partial }{\partial n_j} -  k \right) ({\cT_{j,\ell} v_\ell}) & = 0  \quad  \text{on } \  \partial \Omega_j \cap \partial \Omega.   \label{eq43} 
\end{align}
Therefore,
 \begin{align}\label{eq:errit} 
  e_j^{n+1} = \sum_\ell \cT_{j,\ell} e_\ell^n, \quad \text{and thus} \quad  \be^{n+1} = \bcT \be^{n}.   
 \end{align}
Observe that 
     (i) if $\Omega_j \cap\Omega_\ell = \emptyset$, then $\cT_{j,\ell}=0$ (since the right-hand side of \eqref{eq42} is zero on $\partial\Omega_j\setminus\partial\Omega$), (ii) since $\chi_\ell$ vanishes on $\partial\Omega_\ell\setminus \partial\Omega$,  $(-i\partial /\partial n_\ell - k )(\chi_\ell v_\ell)$ vanishes on $\partial \Omega_\ell\setminus\partial \Omega$,  and thus  $\cT_{\ell, \ell} \equiv 0$ for all $\ell$.

 It is convenient here to introduce  the notation
     \begin{align} 
\Gamma_{j,\ell}:= (\partial \Omega_j \setminus \partial \Omega)\cap \Omega_\ell,\label{defGamma}
     \end{align} 
     so that  \eqref{eq42} holds on $\Gamma_{j,\ell}$ and \eqref{eq43} holds on $\partial \Omega_j \backslash \Gamma_{j,\ell}$.   

\begin{rem}
Using the definition of $\bcT$, the parallel Schwarz method \eqref{eq21}-\eqref{eq23} can be written as 
\beqs
\bu^{n+1} = \bcT \bu^n + \bF,
\eeqs
where $(\bu^n)_j := u^n_j$, $(\bF)_j:=F_j$,
where $F_j\in U(\Omega_j)$ satisfies
\beqs
(\Delta+k^2) F_j = -f \,\,\tin \Omega_j,
\eeqs
\beqs
 \left(\frac{1}{i}\pdiff{}{n_j}-k\right)F_j= 0 \,\,\ton \partial\Omega_j\setminus \partial \Omega, \quad\tand\quad
 \left(\frac{1}{i}\pdiff{}{n_j}-k\right) F_j= g \,\,\ton \partial\Omega_j\cap \partial \Omega.
\eeqs
\end{rem}

\subsection{The goal:~proving power contractivity of $\bcT$}\label{sec:2.3}

The paper \cite{GoGaGrLaSp:22} sought to prove that $\bcT^M$ is a contraction, for some appropriate $M\geq 1$, in an appropriate norm. The motivation for this is that, in 1-d with a strip decomposition (i.e., the 1-d analogue of the decompositions considered in \S\ref{sec:2.6} below), $\bcT^N=0$ where $N$ is the number of subdomains; see \cite[Propositions 2.5 and 2.6]{NaRoSt:94}. This property holds because, in 1-d, the impedance boundary condition is the exact Dirichlet-to-Neumann map for the Helmholtz equation.

To define an appropriate norm, let 
\beqs
 U_0 (\Omega_j) := \Big\{ u \in H^1(\Omega_j): \,  (\Delta  +k^2) u=0, \, (-i\partial/\partial_n -k)u
   \in L^2(\partial \Omega_j) \Big\}\subset U(\Omega_j).
\eeqs
 Since $e_j^n \in U_0(\Omega_j)$ for each $j$,  \cite{GoGaGrLaSp:22} analyses convergence of \eqref{eq:errit} in the space $\bbU_0:=\prod_{\ell=1}^N U_0(\Omega_\ell)$.

 \begin{lem}[Norm on $U_0(\Omega_j)$]\label{lem:norm}
For $\Omega_j$ a bounded Lipschitz domain, let  
$\|\cdot\|_{1,k,\partial \Omega_j}$ be defined by
   \begin{align}\label{eq:pseudo-energy}
\|v\|_{1,k,\partial \Omega_j}^2 \ := \  
\N{\pdiff{v}{n}}^2_{L^2(\partial \Omega_j)} + k^2 \N{v}^2_{L^2(\partial \Omega_j)},
   \end{align}
   where  $\partial/\partial n$ denotes the outward normal derivative on $\partial \Omega_j$. Then $\|\cdot\|_{1,k,\partial \Omega_j}$ is a norm on $U_0(\Omega_j)$ and 
     \begin{align} \label{eq:norm2}   \|v\|_{1,k,\partial \Omega_j}^2 \ = \ 
\N{\frac{1}{i}\pdiff{v}{n}- k v }^2_{L^2(\partial \Omega_j)}\ =\  \N{\frac{1}{i}\pdiff{v}{n}+ k v }^2_{L^2(\partial \Omega_j)},
     \end{align} 
\end{lem}
The norm on $\bbU_0$ is then defined by 
   \begin{align}\label{eq:pseudo-energy1}
\|\bv\|_{1,k,\partial }^2 \ := \sum_{\ell=1}^N \|v_\ell \|_{1,k,\partial \Omega_\ell}^2 \quad \text{for} \quad \bv \in \bbU_0.
\end{align}

For the proof of Lemma \ref{lem:norm}, see \cite[Lemma 3.3]{GoGaGrLaSp:22}. We highlight that (i) the norm \eqref{eq:pseudo-energy} was used in the non-overlapping analysis in \cite{De:91,BeDe:97} (see \cite[Equation 12]{BeDe:97}) and (ii) the relation \eqref{eq:norm2} is a well-known ``isometry'' result about impedance traces; see, e.g., \cite[Lemma 6.37]{Sp:15}, \cite[Equation 3]{ClCoJoPa:21}. {The relation \eqref{eq:norm2} and the fact that the $L^2$ norm is much easier to compute than the $H^{-1/2}$ norm are the reasons why \cite{GoGaGrLaSp:22} considers the impedance-to-impedance map in $L^2$, rather than the trace space $H^{-1/2}$ (although we expect that the theory in \cite{GoGaGrLaSp:22} can be suitably modified to hold in $H^{-1/2}$).}

\subsection{How impedance-to-impedance maps arise in studying $\bcT^M$ for $M\geq 1$}\label{sec:2.4}

When studying $\bcT^M$ for $M\geq 1$, compositions such as 
$\cT_{j,\ell} \cT_{\ell, j'}$ naturally arise. Indeed,
\begin{align} \label{eq:square} 
(\bcT^2)_{j,j'} \ = \ \sum_\ell \cT_{j,\ell} \cT_{\ell, j'},  
\end{align}
where the sum is over all $\ell \in  \{1, 2, \ldots , N \} \backslash \{ j,j'\}$, with  
neither $\Gamma_{j,\ell}$ nor $\Gamma_{\ell, j'}$ empty (equivalently, neither 
$\Omega_j\cap \Omega_{\ell}$ nor $\Omega_\ell\cap \Omega_{j'}$ empty).

To condense notation, let 
\beqs
\imp_j := \frac{1}{i}\pdiff{}{n_j} -k,
\eeqs
so that the boundary condition \eqref{eq42} becomes
\beq
\imp_j (\cT_{j,\ell} v_\ell)  = \imp_j (\chi_\ell v_\ell) \quad  \text{on }  
\Gamma_{j,\ell}.
 \label{eq42new}  
\eeq

A useful expression  for the action of   \eqref{eq:square} can be obtained by  inserting $v_\ell = \cT_{\ell,j'} z_{j'}$,  with $z_{j'} \in U(\Omega_{j'})$,  into the boundary condition \eqref{eq42new},  to obtain
   \begin{align}
    \imp_j \big(\cT_{j,\ell} \cT_{\ell,j'} z_{j'}\big)  &=  \imp_j \big(\chi_\ell\cT_{\ell,j'} z_{j'}\big) =  \chi_\ell\imp_j \big(\cT_{\ell,j'} z_{j'}\big) + \left(\pdiff{\chi_\ell}{n_j}\right)(\cT_{\ell,j'} z_{j'}\big)
        \quad \ton 
        \Gamma_{j,\ell}.
    \label{eq:relation2}
    \end{align}
    Observe that to find the first term on the right-hand side of \eqref{eq:relation2} one (i) finds $\cT_{\ell, j'} {z_{j'}}$, i.e.,  the unique function  in $U_0(\Omega_\ell)$ with impedance data given by $\imp_\ell(\chi_{j'} z_{j'})$ on $\Gamma_{\ell,j'}$ and zero otherwise,  (ii) evaluates $\imp_j(\cT_{\ell,j'}z_{j'})$    on $\Gamma_{j,\ell}$  and (iii) multiplies the result by $\chi_\ell$. The combination of steps (i) and (ii) naturally lead us to consider impedance-to-impedance maps.
    
      \subsection{ The impedance-to-impedance map for general decompositions}\label{sec:2.5}

\begin{definition}[Impedance map] \label{def:impmap}
Let $\ell, j, j' \in \{ 1, \ldots, N\}$
  be such that 
  neither $\Gamma_{j,\ell}$ nor $\Gamma_{\ell, j'}$ is empty (equivalently, neither 
$\Omega_j\cap \Omega_{\ell}$ nor $\Omega_\ell\cap \Omega_{j'}$ is empty).
    Given $g \in L^2(\Gamma_{\ell,j'})$, let $v_\ell$ be the unique element of
    $U_0(\Omega_\ell)$ with impedance data 
\begin{align}  
   \imp_\ell (v_\ell) & = \left\{ \begin{array}{ll} g & \text{on} \quad \Gamma_{\ell,j'} \\  0 & 
  \text{on} \quad \partial \Omega_\ell \backslash \Gamma_{\ell,j'}  \end{array} \right. . \label{H2} \end{align}
Then the impedance-to-impedance map $ \Imap{\ell}{j'}{j}{\ell}: L^2(\Gamma_{\ell,j'})
    \rightarrow L^2(\Gamma_{j,\ell})$ is defined by 
\begin{align}   \Imap{\ell}{j'}{j}{\ell} g := \imp_j(v_\ell), \quad \text{on} \quad \Gamma_{j,\ell}.     \label{H3}
 \end{align}   
  \end{definition}
  \begin{figure}[H]
    \centering
    \begin{subfigure}[t]{0.33\textwidth}
        \centering
      \includegraphics[width=1.1\textwidth]{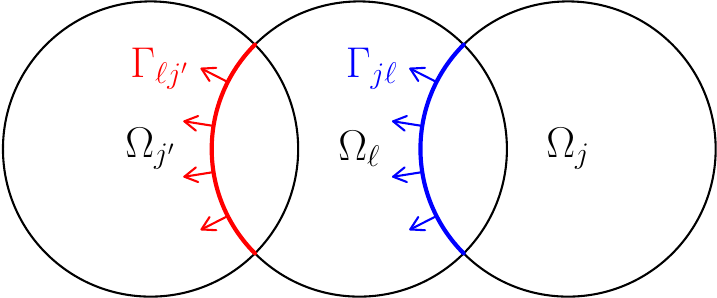}
        \caption{$\Omega_j \cap \Omega_{j'} = \emptyset$  }
         \end{subfigure}%
     \hfill
    \begin{subfigure}[t]{0.33\textwidth}
        \centering
        \includegraphics[width=.75\textwidth]{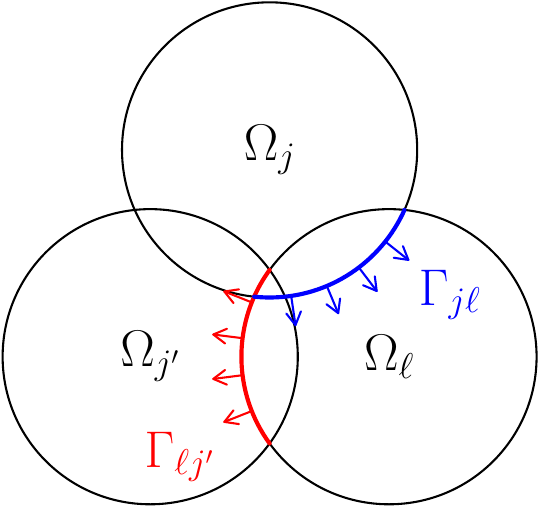}
        \caption{$\Omega_j \cap \Omega_{j'} \not = \emptyset$}
        \end{subfigure}%
           \begin{subfigure}[t]{0.33\textwidth}
        \centering
        \includegraphics[width=.95\textwidth]{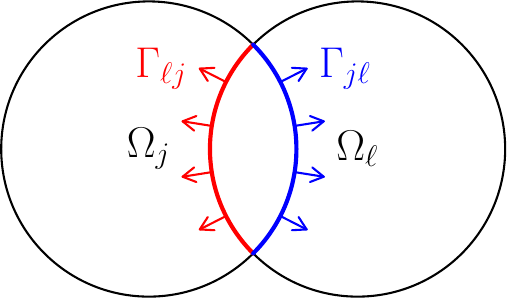}
        \caption{$\Omega_j = \Omega_{j'} $}
         \end{subfigure}%
           \caption{Illustrations of the domain (red) and co-domain (blue) of $\Imap{\ell}{j'}{j}{\ell}$   in 2d (in the case when none of the subdomains $\Omega_j,\Omega_{j'},$ and $\Omega_\ell$ touch $\partial \Omega$)}\label{fig:illustrations}
\end{figure}

Figure \ref{fig:illustrations} illustrates the domain (in red) and co-domain (in blue) of the impedance-to-impedance map, with arrows indicating the direction of the normal derivative. 

The following is a simplified version of \cite[Theorem 3.9]{GoGaGrLaSp:22}.

\begin{lem}[Connection between $\bcT^2$ and the impedance-to-impedance map] \label{thm:main_T2} 
Let $\ell, j, j' \in \{ 1, \ldots, N\}$
  be such that $\Omega_\ell\cap\Omega_{j'} \neq \emptyset$ and $\Omega_\ell\cap\Omega_{j} \neq \emptyset$.
If $z_{j'} \in U(\Omega_{j'})$, then
  \begin{align}
  \imp_j (\cT_{j,\ell} \cT_{\ell,j'} z_{j'} ) = \chi_\ell \, \Imap{\ell}{j'}{j}{\ell}   \big(\imp_\ell (\cT_{\ell,j'} z_{j'})\big)  +  \left(\pdiff{\chi_\ell}{n_j}\right) \left(\cT_{\ell,j'} z_{j'} \right) \quad\ton \Gamma_{j,\ell}. 
\label{eq:impit}    \end{align}
\end{lem}

\begin{proof}
 Since $\cT_{\ell, j'} z_{j'} \in U_0(\Omega_{\ell})$ and $\imp_\ell (\cT_{\ell,j'} z_{j'} ) $ vanishes on $\partial \Omega_\ell \backslash  \Gamma_{\ell,j'}$, the definition of $\Imap{\ell}{j'}{j}{\ell}$ implies that
\beqs
\imp_j \big(\cT_{\ell,j'} z_{j'}\big) 
=\Imap{\ell}{j'}{j}{\ell}\big(\imp_\ell \cT_{\ell,j'} z_{j'}\big) 
\quad \ton \Gamma_{j,\ell};
\eeqs
the result then follows from \eqref{eq:relation2}.
 \end{proof} 
 
\begin{rem}[Imposing other boundary conditions on $\partial \Omega$]\label{rem:other_bc}
\S\ref{sec:2.1} mentioned how, in principle, the analysis of \cite{GoGaGrLaSp:22} can be repeated with boundary conditions on $\partial \Omega$ other than the impedance boundary condition. 
Suppose  the boundary condition in \eqref{eq:Helmholtz} is replaced by $\mathcal{B}u=g$ and \eqref{eq23} replaced by $\mathcal{B}u_j^{n+1}=g$ where $\mathcal{B}$ is an arbitrary operator.
Proceeding formally (i.e., ignoring the issue of identifying the correct analogue of $U(\Omega)$ in which the method is well-posed) we now outline the changes to the above.
The analogues of equations \eqref{eq41}-\eqref{eq43} are the following
\begin{align*}
  (\Delta + k^2)(\cT_{j,\ell} v_\ell)  & = 0  \quad  \text{in } \  \Omega_j,\\
\imp_j(\cT_{j,\ell} v_\ell) & = \imp_j (\chi_\ell v_\ell) \quad  \text{on }  \ \partial \Omega_j\backslash \partial \Omega, \\\
    \mathcal{B}({\cT_{j,\ell} v_\ell}) & = 0  \quad  \text{on } \  \partial \Omega_j \cap \partial \Omega. 
\end{align*}
Just under \eqref{defGamma} above, we split the boundary condition on $\partial \Omega_j\backslash\partial \Omega$ into two (using the support properties of $\chi_\ell$) to obtain the boundary conditions 
\beq\label{eq:John1}
\imp_j \cT_{j,\ell} v_\ell = \imp_j \chi_\ell v_\ell \,\,\ton \Gamma_{j,\ell},\quad
\imp_j \cT_{j,\ell} v_\ell = 0 \,\,\ton (\partial \Omega_j\setminus \partial \Omega)\setminus \Omega_\ell,\quad
\imp_j\cT_{j,\ell} v_\ell = 0 \,\,\ton \partial\Omega_j \cap \partial \Omega,
\eeq
and then combined the latter two as one boundary condition on $\partial \Omega_j\setminus \Gamma_{j,\ell}$.
Now we cannot do this combination, and so have to keep the following analogue of \eqref{eq:John1}
\beqs
\imp_j \cT_{j,\ell} v_\ell = \imp_j \chi_\ell v_\ell \,\,\ton \Gamma_{j,\ell},\quad
\imp_j \cT_{j,\ell} v_\ell = 0 \,\,\ton (\partial \Omega_j\setminus \partial \Omega)\setminus \Omega_\ell,\quad
\mathcal{B}(\cT_{j,\ell} v_\ell) = 0 \,\,\ton \partial\Omega_j \cap \partial \Omega,
\eeqs
The boundary conditions for the function $v_\ell$ in the definition of the impedance-to-impedance map (Definition \ref{def:impmap}) then become 
\begin{align*}  
   \imp_\ell (v_\ell) & = \left\{ \begin{array}{ll} g & \text{on} \quad \Gamma_{\ell,j'}, \\  0 & 
  \text{on} \quad (\partial \Omega_\ell\setminus \partial\Omega) \backslash \Omega_{j'},  \end{array} \right. \\
  \mathcal{B}(v_\ell)&=0 \,\,\ton \partial \Omega_\ell \cap \partial \Omega;
   \end{align*}
i.e., the absorbing boundary condition $\mathcal{B}(v_\ell)=0$ is always imposed on the part of any subdomain boundary that intersects $\partial\Omega$.
   \end{rem}
 
\subsection{Specialisation to 2-d strip decompositions}\label{sec:2.6}
 
\begin{figure}[H]
  \begin{center}
  \scalebox{0.9}{
  \includegraphics{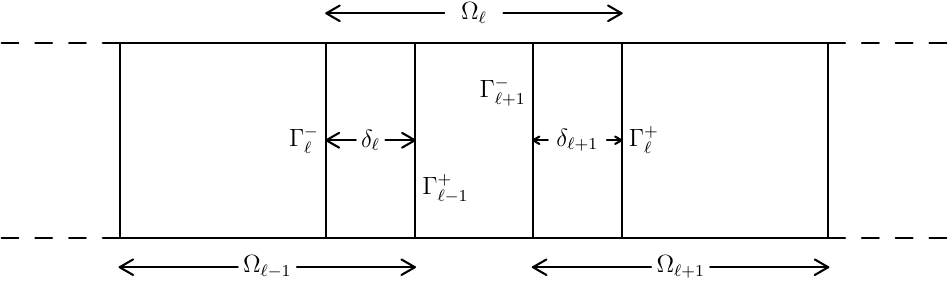}
  }
\end{center}   \caption{Three overlapping subdomains in the 2-d strip decomposition \label{fig:overlap}}
\end{figure}

As in \cite[\S4]{GoGaGrLaSp:22}, we now specialise to strip decompositions in 2-d.
    Without loss of generality, we assume the domain $\Omega$ is a rectangle of height one and that the  subdomains $\Omega_\ell$  have height one, width $L_\ell$, and are  bounded by vertical sides denoted $\Gamma_\ell^-$, $\Gamma_\ell^+$.
    We  assume that each  $\Omega_\ell$ is only overlapped by $\Omega_{\ell-1}$ and $\Omega_{\ell+1}$ (with   $\Omega_{-1} := \emptyset$ and $\Omega_{N+1}  :=\emptyset$); see  Figure \ref{fig:overlap}.

This set-up falls into the class of general decompositions considered above  with 
 \begin{align} \label{notstrip} \Gamma_\ell^-  = \Gamma_{\ell, \ell-1}  \quad \text{and}\quad \Gamma_\ell^+ = \Gamma_{\ell, \ell+1};\end{align} 
 i.e., the four interfaces shown in Figure \ref{fig:overlap} (from left to right) are
 \beqs
 \big\{ \Gamma_{\ell,\ell-1}, \Gamma_{\ell-1,\ell}, \Gamma_{\ell+1,\ell}, \Gamma_{\ell, \ell+1}\big\}
 \quad
 \tand
 \quad
 \big\{ \Gamma_{\ell}^-, \Gamma_{\ell-1}^+, \Gamma_{\ell+1}^-, \Gamma_\ell^+\big\}.
 \eeqs
 
 Under this set-up, the functions $\{\chi_j\}_{j=1}^N$ defined in \S\ref{sec:2.1} satisfy
        \begin{align} \label{eq:important} 
          \chi_{\ell}\vert_{\Gamma_{\ell-1}^+} = 1 = \chi_\ell\vert_{\Gamma_{\ell+1}^-},  \end{align}
and thus 
Theorem \ref{thm:main_T2} simplifies to the following (with this result a simplified version of \cite[Corollary 4.3]{GoGaGrLaSp:22}).

\begin{cor}\label{cor:main_T2strip} 
Let $\ell \in \{ 1, \ldots, N\}$ and $j,j' \in \{\ell-1,\ell+1\}$. 
If $z_{j'} \in U(\Omega_{j'})$, then
  \begin{align}
\imp_j (\cT_{j,\ell} \cT_{\ell,j'} z_{j'} ) = \Imap{\ell}{j'}{j}{\ell} \,    \left(\imp_\ell (\cT_{\ell,j'} z_{j'})\right)  \quad\ton \Gamma_{j,\ell}.
  \label{eq:impit2new}    
\end{align}
\end{cor}
 
Recall from the definition of $\cT_{j,\ell}$ \eqref{eq41}-\eqref{eq43} that
$\cT_{\ell,\ell} = 0$ and 
$  \cT_{j,\ell} = 0$ {if } $\Omega_j \cap \Omega_\ell =\emptyset.
$
Therefore, for these strip decompositions, $\bcT$ takes the block tridiagonal  form  
 
\begin{equation}\label{eq:splitT}
\bcT =\begin{pmatrix}
0		&	\cT_{1,2} \\
\cT_{2,1}	&	0		&\cT_{2,3}\\
		&	\cT_{3,2}	&0		&\cT_{3,4}\\
		&			&	\ddots	&	\ddots&	\ddots\\
		&			&	\cT_{N-1,N-2}&		0	&\cT_{N-1,N}\\
		&			&			&	\cT_{N,N-1}&0	
\end{pmatrix}.
\end{equation}

\subsection{The impedance-to-impedance maps for 2-d strip decompositions}\label{sec:2.7}

There are now four impedance-to-impedance maps associated with $\Omega_\ell$:
\begin{align}\nonumber
\Imap{\ell}{\ell-1}{\ell-1}{\ell} = \Imaps{\ell}{-}{\ell-1}{+},\quad
\Imap{\ell}{\ell+1}{\ell+1}{\ell} = \Imaps{\ell}{+}{\ell+1}{-},\\
\Imap{\ell}{\ell-1}{\ell+1}{\ell} = \Imaps{\ell}{-}{\ell+1}{-},\quad
\Imap{\ell}{\ell+1}{\ell-1}{\ell} = \Imaps{\ell}{+}{\ell-1}{+};\label{eq:strip_maps}
\end{align}
there are four maps because $\partial\Omega_\ell\setminus \partial\Omega$ has two disjoint components, $\Gamma_\ell^-$ and $\Gamma_\ell^+$, and there are two ``interior interface'' in $\Omega_\ell$, $\Gamma_{\ell-1}^+$ and $\Gamma_{\ell+1}^-$.

These four maps can in turn be expressed in terms of two maps.

 \begin{definition}\emph{(``Canonical'' impedance-to-impedance maps \cite[Definition 4.9]{GoGaGrLaSp:22})} \label{def:canonical}   For $\hL > 0$, let ${\wOmega} :=  [0,\hL] \times [0,1]$, let $\Gamma_{\leftsubscript} := \{0\} \times [0,1]$ and let $\Gamma_{\middlesubscript} := \{\mathsf {d_\leftsubscript}\} \times [0,1]$ for $\mathsf {d_\leftsubscript} \in (0,1)$ (compare to Figure \ref{fig:model}). 
  Let  ${u}$ be the solution to
\begin{equation}\label{eq:plane-wave-scattering1}
(  \Delta + {k}^2) {u} =0 \,\,\text{ in } \,\, {\wOmega},\quad
     \left(-\frac{1}{i}\partial_{x_1} -{k} \right)u = g\,\,\ton\,\, \Gamma_{\leftsubscript}, \quad
    \left(\frac{1}{i}\partial_n  - k\right) u =0\,\,\ton\,\,  \pwOmega\backslash \Gamma_{\leftsubscript}. 
  \end{equation}
  Let 
  \begin{align} \label{eq:maps} 
\cI^- g:= \left(\frac{1}{i}\partial_{x_1} -k\right) u , \qquad  \cI^+ g := \left(\frac{1}{i} \partial_{x_1} +k\right)u \quad \text{on} \quad \Gamma_{\middlesubscript} , \end{align} 
and let 
\begin{equation}\label{eq:rho_gamma}
 {\rho}({k},{\mathsf {d_\leftsubscript}}, L) = \sup_{{g}\in L^2(\Gamma_{\leftsubscript})}\frac{\|\cI^- g  \|_{L^{2}({\Gamma}_{\middlesubscript})} }
 {   \|{g}\|_{L^2(\Gamma_{\leftsubscript})}},
 \quad\quad
{\gamma}({k},{\mathsf {d_\leftsubscript}}, L)  = \sup_{{g}\in L^2(\Gamma_{\leftsubscript})}\frac{\| \cI^+g  \|_{L^{2}(\Gamma_{\middlesubscript})} }
{  \|{g}\|_{L^2(\Gamma_{\leftsubscript})}}.
\end{equation}
\end{definition}

The maps $\Imaps{\ell}{-}{\ell-1}{+}$ and $\Imaps{\ell}{+}{\ell+1}{-}$ can be written in terms of $\cI^-$ and $\cI^+$, as in the following result from \cite{GoGaGrLaSp:22}.

\begin{cor}\emph{(\cite[Corollary 4.11]{GoGaGrLaSp:22})}\label{cor:imp}
For the 2-d strip decomposition described above, 
  \begin{align} 
    \|\Imaps{\ell}{-}{\ell-1}{+}\|_{L^2\to L^2} = \hrho(k, \delta_\ell, L_\ell),  \qquad
                                   \|\Imaps{\ell}{+}{\ell+1}{-}\|_{L^2\to L^2}  =      \hrho(k, \delta_{\ell+1}, L_\ell),  
\label{eq:rho} \end{align}
and 
\begin{align}
  \|\Imaps{\ell}{-}{\ell+1}{-}\|_{L^2\to L^2}  =   \hgamma(k,  L_\ell-\delta_{\ell+1}, L_\ell),  \qquad
                                       \|\Imaps{\ell}{+}{\ell-1}{+}\|_{L^2\to L^2}  = \hgamma(k, L_\ell-\delta_\ell, L_{\ell} ) \label{eq:gamma}
\end{align}
(where $\delta_\ell$ and $\delta_{\ell+1}$ are shown in Figure \ref{fig:overlap}, and $L_\ell$ is the width of $\Omega_\ell$).
\end{cor}

\begin{rem}[Why the maps in \S\ref{sec:1.4} involve two different values of $\mathsf{d}_\leftsubscript$ and $\mathsf{d_{\rightsubscript}}$]\label{rem:twod}
If the subdomains all have equal lengths $L$ and equal overlaps $\delta$, then the map $\mathcal I^-$ appears in Corollary \ref{cor:imp} with $\mathsf{d}_\leftsubscript = \delta$ and $\mathsf{d_{\rightsubscript}}= L-\delta$ and the map $\mathcal I^+$ appears with $\mathsf{d}_\leftsubscript = L-\delta$ and $\mathsf{d_{\rightsubscript}}= \delta$. This is why  in \S\ref{sec:1.4} we consider $\mathcal{I}^+(\mathsf{d}^{+},\mathsf{d_{\rightsubscript}}^{+})$ and $\mathcal{I}^-(\mathsf{d}^{+},\mathsf{d_{\rightsubscript}}^{+})$ with $\mathsf{d}^+$ not necessarily equal to $\mathsf{d}^-$ (and similar for $\mathsf{d_{\rightsubscript}}^\pm$).
\end{rem}

\subsection{The relation of the impedance-to-impedance maps in Definition \ref{def:canonical} to those in \S\ref{sec:1.2}}\label{sec:2.8}

The maps in \S\ref{sec:1.2} are related to the impedance-to-impedance maps in  Definition \ref{def:canonical} for 
the 2-d strip decomposition when the boundary condition on $\partial\Omega$ (approximating the Sommerfeld radiation condition) is the outgoing condition described in \S\ref{sec:1.5}. Indeed, following Remark \ref{rem:other_bc}, 
the maps $\cI^{\pm}$ in Definition \ref{def:canonical} correspond to $\mathcal I^{\pm}_{2}$ when the subdomain is not $\Omega_1$ or $\Omega_N$
(i.e., when the subdomain does not touch the left- or right-ends of the strip)
 and $\mathcal I^{\pm}_{1}$ when the subdomain is one of $\Omega_1$ or $\Omega_N$.
In particular, in this set-up with only two subdomains (i.e., $N=2$), the maps $\cI^{\pm}$ are $\mathcal I^{\pm}_{1}$.

As mentioned in \S\ref{sec:2.2}, the rationale for considering the idealised case where the boundary condition on $\partial\Omega$ is the outgoing condition is that it allows us to focus on the impedance boundary conditions imposed in the domain decomposition method itself, and ignore the influence of the impedance boundary condition imposed as an approximation of the Sommerfeld radiation condition.

Perfectly matched layers approximate outgoing Helmholtz solutions with accuracy increasing exponentially in the limit $k\to \infty$; see \cite{GLS2}.
We therefore expect the behaviour of the maps in \S\ref{sec:1.2} to be similar to the analogous maps with a perfectly matched layer on the top and bottom of $\domain$. 
         
\subsection{The relevance of Theorems \ref{th:model1}-\ref{th:compo} to the analysis of the parallel Schwarz method for 2-d strip decompositions}\label{sec:2.9}

\

We now give five conclusions and one conjecture about the analysis of the parallel Schwarz method for 2-d strip decompositions. The justification/discussion of each is given after its statement.

\vspace{1ex}

\emph{Conclusion 1:~with outgoing boundary condition on $\partial \Omega$, the parallel Schwarz method with two subdomains is power contractive for sufficiently-large subdomain overlap.} 

\vspace{1ex}

By considering powers of $\bcT$ given by \eqref{eq:splitT}, using \eqref{eq:impit2new}, \eqref{eq:strip_maps}, and the definitions of $\rho$ \eqref{eq:rho} and $\gamma$ \eqref{eq:gamma}, \cite[Corollary 4.14]{GoGaGrLaSp:22} obtains the following sufficient condition for $\bcT^N$ to be a contraction

\begin{theorem}\emph{(Informal statement of \cite[Corollary 4.14]{GoGaGrLaSp:22})}\label{th:contraction}
If $\rho$ is sufficiently small relative to $\gamma$ and $N$, then 
$\big\|\bcT^N\big\|_{1,k,\partial }<1.$
\end{theorem}   
   
Theorem \ref{th:model1} and Corollary \ref{cor:imp} show that, with outgoing boundary conditions on $\partial \Omega$ and two subdomains (i.e., $N=2$), the analogue of $\rho$ \eqref{eq:rho_gamma} in this set-up can be made arbitrarily small by increasing the overlap of the subdomains, and the analogue of $\gamma$ is bounded below by one; i.e., in this set up, $\bcT^N$ is a contraction for sufficiently-large overlap (which necessarily means that the length of $\Omega$ must be sufficiently large relative to its height).

\vspace{1ex}

\emph{Conclusion 2:~one cannot obtain a $k$-explicit convergence theory, valid for arbitrarily-large $k$, of the parallel Schwarz method with more than 2 subdomains by looking at the norms of single impedance-to-impedance maps.}

\vspace{1ex}

Theorem \ref{th:model2} and Corollary \ref{cor:imp} show that, with outgoing boundary conditions on $\partial \Omega$ and more than two subdomains, the analogues of $\rho$ and $\gamma$ are both bounded below by one. Thus Theorem \ref{th:contraction} or its precursor \cite[Theorem 4.13]{GoGaGrLaSp:22}, which are based on bounding $\bcT^N$ by norms of single impedance-to-impedance maps, cannot be used to prove that $\bcT^N$ is a contraction.

\vspace{1ex}
   
\emph{Conclusion 3:~the values of $k$ chosen in the computations of impedance-to-impedance maps in \cite{GoGaGrLaSp:22} 
   were not large enough to see the asymptotic behaviour of $\rho$ (the norm of $\cI^-$).} 

\vspace{1ex}

The lower bounds on $\cI^{\pm}_{2}$ in Theorem \ref{th:model2} do not immediately give lower bounds for $\rho$ and $\gamma$ in Definition \ref{def:canonical}; this is because $\rho$ and $\gamma$ contain contributions from additional reflections from the impedance boundary conditions on the top and bottom, which are not present in  $\cI^{\pm}_{2}$.
 Although it is possible that these additional reflections create fortunate cancellation, this
 would be highly non-generic,
   and we therefore expect the lower bounds in Theorem \ref{th:model2} to carry over to $\rho$ and $\gamma$; i.e., we expect both $\rho$ and $\gamma$ to be bounded below by one at high frequency.

\cite[Table 1]{GoGaGrLaSp:22} computes $\rho(k, L/3,L)$ and $\gamma(k,L/3,L)$ for $k$ in the range $[10,80]$ and $L$ in the range $[1,16]$. At $k=80$, $\rho \leq 0.3$ and $\gamma \geq 0.97$, showing that $k=80$ is large enough to see the asymptotic behaviour of $\gamma$, but not the asymptotic behaviour of $\rho$.
   
 \vspace{1ex}  
   
\emph{Conclusion 4:~one cannot obtain a $k$-explicit convergence theory, valid for arbitrarily-large $k$, of the parallel Schwarz method by looking at the norms of composite impedance-to-impedance maps on the whole of the space on which they are defined.}
  
\vspace{1ex} 

Recognising that estimating $\|\bcT^N\|_{1,k,\partial}$ via norms of single impedance-to-impedance maps might be too crude, \cite[\S4.4.5]{GoGaGrLaSp:22} outlined how to estimate $\|\bcT^N\|_{1,k,\partial}$ via norms of compositions of impedance-to-impedance maps; essentially this involves using 
\eqref{eq:impit2new} iteratively and relating the impedance-to-impedance maps to the canonical maps via \eqref{eq:strip_maps}.

However, Part (\ref{i:compo:lower}) of Theorem \ref{th:compo} shows that the norm of any composition of $\cI^{\pm}$ is bounded below by one for sufficiently-large $k$, meaning that the strategy outlined in \cite[\S4.4.5]{GoGaGrLaSp:22} will not be able to show that $\|\bcT^N\|_{1,k,\partial}$ is small for arbitrarily-large $k$

\vspace{1ex} 
   
\emph{Conclusion 5:~the lower bound on the composite map in Theorem \ref{th:compo} is only reached for very large $k$, explaining why the computed values of the composite map in \cite[Table 3]{GoGaGrLaSp:22} are small.}

\vspace{1ex} 

As discussed in \S\ref{sec:1.4}, 
the lower bound in Part (\ref{i:compo:lower}) of Theorem \ref{th:compo} is based on choosing particular data corresponding to rays leaving $\Gamma_{\leftsubscript}$ close to horizontal; this is clear both from the proof of Part (\ref{i:compo:lower}) of Theorem \ref{th:compo}, and also from Part (\ref{i:compo:upper}) of Theorem \ref{th:compo}, which shows that if one projects away from zero frequency on the boundary (with zero frequency on the boundary corresponding to rays leaving horizontally) then the norm of the composite map on such data is zero.

Rays leaving $\Gamma_{\leftsubscript}$ close to the horizontal hit $\Gamma_{\rightsubscript}$ close to the horizontal, and thus have small reflection coefficient (since rays hitting horizontally, i.e., normally, have zero reflection coefficient). Therefore, the asymptotic lower bound on the map in Part (\ref{i:compo:lower}) of Theorem \ref{th:compo} will only be reached for very large $k$.

This is consistent with the computations of a particular composite impedance map in \cite[Tables 5 and 6]{GoGaGrLaSp:22}; even when the subdomains have small overlap, the norms of a composition of four impedance-to-impedance maps were less than $0.1$ for $k\in \{10,20,40,80\}$; see \cite[Table 6]{GoGaGrLaSp:22}.

\vspace{1ex} 

\emph{Conjecture 1:~a convergence theory for the parallel Schwarz method with $k\gg 1$ can be obtained by combining 
Part (\ref{i:compo:upper}) of Theorem \ref{th:compo} with the fact that normally incident waves are those for which the exact DtN map is the impedance boundary condition.}

\vspace{1ex} 

For a 2-d strip decomposition with more than two subdomains and outgoing boundary condition on $\Omega$, 
a possible route to prove that $\|\bcT^M\|_{1,k,\partial}<1$ for some $M>0$ is the following. 
Combine (i) the expression for 
$\bcT^M$ in terms of composite impedance-to-impedance maps (outlined in \cite[\S4.4.5]{GoGaGrLaSp:22}), (ii) 
Part (\ref{i:compo:upper}) of Theorem \ref{th:compo}, to show that the contribution from non-normally incident waves is small for $M$ large enough, and (iii) the fact that the normally incident waves not covered by Part (\ref{i:compo:upper}) of Theorem \ref{th:compo} are those for which the exact DtN map is the impedance boundary condition (i.e., the impedance boundary condition on the subdomains is the ``correct'' one for these waves). We expect the optimal $M$ will come from balancing the fact that to use (iii) we want $\lambda$ to be small (so that only rays very close to horizontal need to be dealt with via this mechanism), but the smaller $\lambda$, the larger the $n_0(\lambda)$ in Theorem \ref{th:compo} (since the smaller the $\lambda$, the closer to the horizontal the rays can be, and a larger number of reflections is needed to push their mass into the outgoing layers).
Since implementing these steps would require a more-refined analysis of the parallel Schwarz method than in \cite{GoGaGrLaSp:22}, this will be investigated elsewhere.

\begin{rem}\emph{(Discussion of other literature on impedance-to-impedance maps in the context of domain decomposition)}\label{rem:literature}
Whereas the paper \cite{GoGaGrLaSp:22} showed how impedance-to-impedance maps govern the error propagation in the non-overlapping parallel Schwarz method, but do not appear explicitly in the formulation of the method, 
impedance-to-impedance maps appear explicitly in the formulation of certain fast direct solvers \cite{GiBaMa:15}, \cite{BeGiHe:20} and non-overlapping domain decomposition methods \cite{PeTuBo:17}. 

The impedance-to-impedance maps associated with the fast direct solver in \cite{GiBaMa:15} were recently analysed in \cite{BeCaMa:22}
The prototypical example of the maps studied in \cite{BeCaMa:22} is the following. Given a square with zero impedance data on three sides, the considered map is the map from impedance data on the fourth side to the impedance data on that same side but with the sign swapped (i.e., the map from $\partial_n u+i ku$ to $\partial_n u -ik u$, where $u$ is a Helmholtz solution); see \cite[Definition 1.1]{BeCaMa:22}.
The main goal of \cite{BeCaMa:22} is to study $I-R_1R_2$, where $R_1,R_2$ are two of these impedance-to-impedance maps on squares.
If $I-R_1R_2$ is invertible, then the impedance-to-impedance map on a $2\times 1$ rectangle can be recovered from the impedance-to-impedance maps on two non-overlapping square subdomains (see \cite[Equations 7-9 and the surrounding text]{BeCaMa:22}).
Whereas $I-R_1R_2$ was assumed invertible in \cite{GiBaMa:15}, 
 \cite[Theorem 1.2]{BeCaMa:22} proved a $k$-explicit bound on $(I-R_1R_2)^{-1}$.

The techniques used in \cite{BeCaMa:22} are quite different from those in the present paper; indeed 
 \cite[Theorem 1.2]{BeCaMa:22} is proved using vector-field arguments and bounds on the Neumann Green's function. Nevertheless, the fact that the bound on $(I-R_1R_2)^{-1}$ in \cite[Theorem 1.2]{BeCaMa:22} allows for its norm to grow as $k\to \infty$ is thematically similar to the ``bad'' behaviour as $k\to\infty$ of our impedance-to-impedance maps proved in Theorem \ref{th:model2} and Part (\ref{i:compo:upper}) of Theorem \ref{th:compo}.
\end{rem}

\section{Definition and useful properties of defect measures} \label{sec:3}

\subsection{The local geometry and associated notation} \label{subsec:geonot}

\subsubsection{The geometry}
For $\domain^+$ as in Definition \ref{def:outgo}, we denote $\Gamma_\middlesubscript^+ := \big(\{\mathsf {d_{\leftsubscript}}\} \times \mathbb R\big)\cap \domain^+$. We denote the dual variables to $(x_1,x')$ by $(\xi_1,\xi')$.
Near the interfaces $\Gamma=\Gamma_{\leftsubscript}$, $\Gamma=\Gamma_{\leftsubscript}$, and $\Gamma=\Gamma_{\middlesubscript}^+$, 
we use Riemannian/Fermi normal coordinates $(s,x')$, in which $\Gamma$ is given by $\{s=0\}$ and 
$[0,\mathsf{d_\leftsubscript+d_{\rightsubscript}}]\times[0,\height ]$ is given by $\{ s>0\}$ for both $\Gamma_{\leftsubscript}$ and $\Gamma_{\rightsubscript}$, and 
$[0,\mathsf{d_\leftsubscript}]\times[0,\height ]$ is given by $\{ s>0\}$ for $\Gamma_{\middlesubscript}$; therefore
$$
s = x_1 \text{ on }\Gamma_{\leftsubscript}, \quad s=- x_1 \text{ on }\Gamma_{\rightsubscript}, \quad s=- x_1 \text{ on }\Gamma_{\middlesubscript}^+.
$$
The conormal variable to $s$ is denoted by $\zeta$; so that
$$
\zeta = \xi_1 \text{ on }\Gamma_{\leftsubscript},\quad \zeta=-\xi_1 \text{ on }\Gamma_{\rightsubscript}, \quad \tand\quad \zeta=-\xi_1 \text{ on }\Gamma_{\middlesubscript}^+.
$$
In addition, on $T^*(\Gamma_\leftsubscript\cup \Gamma_\middlesubscript^+ \cup \Gamma_\rightsubscript)$, we let 
\beq\label{eq:r}
r(x',\xi') := 1 - |\xi'|^2,
\eeq
in such a way that, near $\Gamma_\leftsubscript\cup \Gamma_\middlesubscript^+ \cup \Gamma_\rightsubscript$,
$-\hsc^2 \Delta - 1$ is given by
$$
-\hsc^2 \Delta - 1 = (\hsc D_{x_1})^2 - r(x',\hsc  D_{x'}).
$$
We define the \emph{hyperbolic set} in $T^*(\Gamma_\leftsubscript\cup \Gamma_\middlesubscript^+ \cup \Gamma_\rightsubscript)$ by 
$$
\mathcal H := \big\{ (x', \xi') \in T^*(\Gamma_\leftsubscript\cup \Gamma_\middlesubscript^+ \cup \Gamma_\rightsubscript), \; r(x', \xi') > 0 \big\},
$$
and for $(x, \xi') \in \mathcal H$, we denote
$$
\zeta_{\rm in} (x',\xi'):= - \sqrt{r(x',\xi')}, \hspace{0.3cm}\zeta_{\rm out} (x',\xi'):=+\sqrt{r(x',\xi')}.
$$
Finally, $\varphi_t$ is the \emph{Hamiltonian flow} associated to the equation, that is, for
$(x_0,\xi_0) \in T^* \mathbb R^2$ and $t\in \mathbb R$, $\varphi_t(x_0,\xi_0) = (x(t), \xi(t))$ is the solution to
$$
x(0) = x_0, \; \xi(0) = \xi_0, \; \dot \xi(t) = 0, \; \dot x(t) = 2 \xi(t).
$$

\subsubsection{Quantisation and semiclassical wavefront set}

For $b\in C^\infty_c(T^*\mathbb R^2)$, we recall that the (standard) \emph{semiclassical quantisation} of $b$, which we denote by 
$b(x,\hsc D_x)$ or $\operatorname{Op}_\hsc(b)$, is defined by  
\begin{equation} \label{eq:scquant}
\big(b(x,\hsc D_x) v\big)(x) := (2\pi \hsc)^{-2} \int_{T^*\mathbb R^2} 
e^{i(x-y)\cdot\xi/\hsc} \,
a(x,\xi) v(y) \, dy  d\xi ,
\end{equation}
see, for example, \cite[\S4.1]{Zw:12}, \cite[Page 543]{DyZw:19}.
In addition, for $\Gamma \in \{ \Gamma_{\middlesubscript}, \Gamma_{\leftsubscript}, \Gamma_{\rightsubscript}\}$, if $a \in C^\infty_c(T^*\Gamma)$, we
 define, as an operator acting on $L^2(\Gamma)$
$$
\big(a(x',\hsc D_{x'}) v\big)(x') := (2\pi \hsc)^{-1} \int_{T^*\Gamma} 
e^{i(x'-y')\cdot\xi'/\hsc} \,
a(x,\xi) v(y) \, dy'  d\xi'.
$$

We can now introduce the \emph{semiclassical wavefront-set} of an $\hsc$-tempered family of functions. An $\hsc$--family of functions $v(\hsc)\in L^2_{\rm loc}$ is $\hsc$-tempered if for each $\chi \in C^\infty_c$ there exists $C, N$ such that $\|\chi v(\hsc)\|_{L^2}\leq C\hsc^{-N}$.

\begin{definition}[Semiclassical wavefront set]\label{def:WF}
For an $\hsc$-tempered family of functions $v(\hsc)\in L_{\rm loc}^2(D^+)$, 
 we say that $(x_0,\xi_0)\in {T^*D^+}$ 
 is \emph{not} in the semiclassical wavefront-set of  $v(\hsc)$ and write $(x_0,\xi_0)\notin \operatorname{WF}_{\hsc}(v)$ if for any $b\in C^\infty_c(T^* D^+)$ 
 supported sufficiently close to $(x_0,\xi_0)$, and for any $N,s$, there is $C_{s,N}>0$ such that 
$$
\Vert b(x,\hsc D_{x}) v  \Vert_{H^s_\hsc} \leq C_{s,N} \hsc^N,
$$
for the semi-classical Sobolev norm
\begin{equation} \label{eq:scsob}
\Vert f \Vert_{H^s_\hsc} := \Vert f \Vert_{L^2} + \hsc^s \Vert f \Vert_{\dot H^s}
\end{equation}
(where the second term on the right-hand side is the $H^s$ semi norm).
\end{definition}

\subsection{Semiclassical defect measures}

\begin{theorem}[Existence of defect measures] \label{th:ex_defect}
Let $v(\hsc_k) \in L^2(\domain^+).$
\begin{enumerate}
\item 
If there exists $C>0$ so that
$$
\Vert v \Vert_{L^2(\domain^+)} \leq C,
$$
then, there exists a subsequence $h_{k_\ell}\to 0$ and a non-negative Radon measure $\mu$ on $T^*{\mathbb R^2}$, such that for any $b\in C_c^\infty(T^*{\mathbb R^2})$
$$
\big\langle b(x,h_{k_\ell}D_{x})v,v\big\rangle_{{\mathbb{R}^2}}\rightarrow\int b\ d\mu,
$$
where v has been extended by zero to $\mathbb R^2$.
\item Let $\Gamma \in \{ \Gamma_{\middlesubscript}^+, \Gamma_{\leftsubscript}, \Gamma_{\rightsubscript} \}$ {and identify functions on $\Gamma$ with functions on the corresponding subset of $\mathbb{R}$.}
If there exists $C>0$ so that
$$
\Vert v \Vert_{L^2(\Gamma)} + \Vert \hsc \partial_n v \Vert_{L^2(\Gamma)} \leq C,
$$
then, there exists a subsequence $h_{k_\ell}\to 0$, non-negative Radon measures $\nu_{d}$, $\nu_{n}$, and a Radon measure $\nu_{j}$ on $T^*\mathbb{R}$ such that for any  $a\in C_c^\infty(T^*\mathbb{R})$,
\begin{align*}
\big\langle a(x,h_{k_\ell}D_{x'})v,v\big\rangle_{{\mathbb{R}}}&\rightarrow\int a\ d\nu_{d},\\
\big\langle a(x,h_{k_\ell}D_{x'})h_{k_\ell}D_s v,v\big\rangle_{{\mathbb{R}}}&\rightarrow\int a\ d\nu_{j}, \\
\big\langle a(x,h_{k_\ell}D_{x'})h_{k_\ell}D_s v,h_{k_\ell}D_s v\big\rangle_{{\mathbb{R}}}&\rightarrow\int a\ d\nu_{n},
\end{align*}
where the traces of $v$ have been extended to $\mathbb R$ by zero.
\end{enumerate}
\end{theorem}
\begin{proof}[Reference for the proof]
See \cite[Theorem 5.2]{Zw:12}.
\end{proof}

In the rest of this section, we assume that $v$ satisfies
$$
(-\hsc^2 \Delta - 1)v = 0 \,\,\text{ in }\,\, \domain^+
$$
and has defect measure $\mu$, and boundary measures $\nu_d, \nu_n, \nu_j$ on $\Gamma$. 

We now introduce the outgoing and incoming boundary measures on $\Gamma$. To do so, we define geodesic coordinates on the flow--out and the flow--in of a subset of $\mathcal H$:

\begin{definition} \label{def:geocor}
For $\mathcal V \subset \mathcal H$, let $\mathcal B_{{\rm out}}(\mathcal V), \mathcal B_{{\rm in}}(\mathcal V) $ be defined by
\begin{align*}
\mathcal B_{{\rm out}}(\mathcal V)  &:= \bigcup_{(x',\xi')\in\mathcal V} \Big\{ \varphi_{t}\big(0,x',\xi',\zeta = \zeta_{{\rm out}}(x',\xi')\big), \; 
t \in \mathbb R \Big\} \: \cap \: T^*\domain^+ \\
B_{{\rm in}}(\mathcal V)  &:= \bigcup_{(x',\xi')\in\mathcal V} \Big\{ \varphi_{t}\big(0,x',\xi',\zeta =\zeta_{{\rm in}}(x',\xi')\big), \; 
t \in \mathbb R \Big\} \: \cap \: T^*\domain^+. 
\end{align*}
In $\mathcal B_{{\rm out}}(\mathcal V)$, we work in geodesic coordinates $(\rho,t) \in \Big( \mathcal \pi_\Gamma ^ {-1} \mathcal V \cap \big\{ \zeta = \zeta_{{\rm out}} \big\} \Big) \times \mathbb R$, defined by $\mathcal B \ni (x,\xi) = \varphi_t (\rho)$; and in $\mathcal B_{{\rm in}}(\mathcal V)$, we can in the same way work in geodesic coordinates $(\rho,t) \in \Big( \mathcal \pi_\Gamma ^ {-1} \mathcal V \cap \big\{ \zeta = \zeta_{{\rm in}} \big\} \Big) \times \mathbb R$.
\end{definition}

We recall that, for $f:X_1 \rightarrow X_2$ and a measure $\alpha$ on $X_1$, the pushforward measure $f_* \alpha$ on $X_2$ is defined by $f_* \alpha(\mathcal B) := \alpha(f^{-1}(\mathcal B))$, and the change of variable formula
\begin{equation} \label{eq:puschchange}
\int_{X_2} g \; d(f_*\alpha) = \int_{X_1} g \circ f \; d\alpha
\end{equation}
holds. In addition, we let $a \otimes b$ denote the product measure of $a$ and $b$, and we let $\pi_\Gamma$ be the projection map defined by
\begin{equation} \label{eq:defproj}
\pi_\Gamma:\big(s=0,x',\zeta,\xi'\big)\in T^{*}\mathbb{R}^{2}\cap\{ x\in\Gamma\}\rightarrow(x',\xi')\in T^{*}\Gamma.
\end{equation}

\begin{lem}[Relationship between boundary measures and the measure in the interior]
\label{lem:interpr}
There exist {unique} positive measures $\mu_{\rm in}$ and $\mu_{\rm out}$ on $T^* \Gamma$ and supported in $\mathcal{H}$ so that, if
$\mathcal V \subset \mathcal H$,
then, in geodesic coordinates defined above,
$$
\mu = \big(p^{{\rm out}}_*(2\sqrt r\muout) \big) \otimes dt \,\, \text{ in }\,\, \mathcal B_{{\rm out}}(\mathcal V), \hspace{0.5cm}
\mu  = \big(p^{{\rm in}}_*(2\sqrt r\muin) \big) \otimes  dt \,\,  \text{ in } \,\,\mathcal B_{{\rm in}}(\mathcal V),
$$
where $p^{{\rm in}}$ and $p^{{\rm out}}$ are defined by
$$
p^{\rm out/in}: (x',\xi') \in \mathcal H
\longrightarrow \big(x', \xi', \zeta = \zeta_{\rm out/in}(x',\xi')\big) \in  \mathcal \pi_\Gamma ^ {-1} \mathcal V \cap \big\{ \zeta = \zeta_{\rm out/in} \big\}.
$$
\end{lem}

\begin{proof}
Existence of $\muin$ and $\muout$ is proven in \cite[Proposition 1.7]{Miller}, and the rest of the result is proved in \cite[Lemma 2.16]{GLS1} {(with, in particular, uniqueness following from \cite[Equations 2.24/2.41]{GLS1})}.
In \cite[Lemma 2.16]{GLS1}, $\mathcal B_{\rm out/in}$ 
is defined only with $t\gtrless 0$. 
When $\Gamma \in \{ \Gamma_\leftsubscript, \Gamma_\rightsubscript\}$, the definitions in \cite{GLS1} coincide with Definition \ref{def:geocor}, and then the result follows immediately from \cite[Lemma 2.16]{GLS1}. When $\Gamma = \Gamma_\middlesubscript^+$, applying \cite[Lemma 2.16]{GLS1} once with measures taken from the left, and once with measures taken from the right gives the result for all $t\in \mathbb R$.
\end{proof}

 The relationship between  the outgoing and incoming boundary measures and the boundary measures $\nu_d, \nu_n, \nu_j$ is given in the following lemma. Both here and in the rest of the paper we use the notation (as in \cite{Miller}) that $a \mu (\mathcal B) := \int_{\mathcal B} a d\mu$ for a measure $\mu$ and a Borel set $\mathcal B$.
 
\begin{lem}\emph{\cite[Proposition 1.10]{Miller}}\label{lem:Millretations}
{With $r(x',\xi')$ defined by \eqref{eq:r}}, the incoming and outgoing measures satisfy the following.
\begin{enumerate}
\item In $\mathcal H$, %
\begin{align*}
2\mu_{\rm out} &=  \sqrt{r(x',\xi')}\nu_{d} + 2\Re \nu_{j}+\frac{1}{\sqrt{r(x',\xi')}}\nu_{n}, \\
2\mu_{\rm in}  &=\sqrt{r(x',\xi')}\nu_{d}-2\Re \nu_{j}+\frac{1}{\sqrt{r(x',\xi')}}\nu_{n}.
\end{align*}
\item If $\mu_{\rm in}=0$ on some Borel set $\mathcal{B}\subset \mathcal{H}$, then, on $\mathcal B$ 
$$
\mu_{\rm out}=2\Re \nu_{j}=2\sqrt{r(x',\xi')}\nu_{d}=\frac{2}{\sqrt{r(x',\xi')}}\nu_{n}.
$$
\item If $\mu_{\rm out}=0$ on some Borel set $\mathcal{B}\subset \mathcal{H}$, then, on $\mathcal B$ 
$$
\mu_{\rm in}= - 2\Re \nu_{j}=  2\sqrt{r(x',\xi')}\nu_{d}=\frac{2}{\sqrt{r(x',\xi')}}\nu_{n}.
$$
\item If 
$$
-2 \Re  \nu_j=(\Re \alpha)\nu_{d}=4(\Re \alpha)|\alpha|^{-2}\nu_{n}
$$
on some Borel set $\mathcal{B}\subset \mathcal{H}$ for $\alpha$ a complex valued function
such that $\alpha+2\sqrt{r(x',\xi')}$ is never zero
on $\mathcal{B}$ then
$$
\muout= \mathcal R \muin,
$$
where
$$
\mathcal R :=\left|\frac{2\sqrt{r(x',\xi')}-\alpha}{2\sqrt{r(x',\xi')}+\alpha}\right|^{2}\text{ on }\mathcal{B},
$$
If instead, $\alpha-2\sqrt{r}$ is never zero, then 
$$
\mathcal R^{-1}\muout=\muin.
$$

\end{enumerate}
\end{lem}

\begin{notations} \label{not:defect}
When they exist, we denote $\nu^{\mathrm x}_d$, $\nu^{\mathrm x}_n$, $\nu^{\mathrm x}_j$, $\mu^{\mathrm x}_{\rm in}$, $\mu^{\mathrm x}_{\rm out}$ the measures of $v$ on $\Gamma_{\mathrm x}$ with $\mathrm x \in \{ \leftsubscript,\rightsubscript \}$, and on $\Gamma_\middlesubscript^+$ when $\mathrm x = i$.
In these definitions, we take the measures on $\Gamma_{\middlesubscript}^+$ \emph{from the left}. 
When all these objects exist, we succinctly say that $u$ admits defect measures and boundary measures.
\end{notations}

\section{Invariance properties of outgoing Helmholtz solutions } \label{sec:4}

In this section, we assume that $v$ satisfies
$$
(-\hsc^2 \Delta - 1)v = 0 \,\,\text{ in } \domain, 
$$
and we investigate the consequences of $v$ being outgoing near $\Gamma \in \{ \Gamma_s, \Gamma'_s \}$. In the following,
$D^+$ is as in Definition \ref{def:outgo}.

We first prove two lemmas about location of $\operatorname{WF}_\hsc v$.

\begin{lem}\label{lem:outflowout}
Assume that $v$ is outgoing near $\Gamma \in \{ \Gamma_s, \Gamma'_s \}$ (in the sense of Definition \ref{def:outgo}) and let $\widehat \Gamma := \partial 
\domain
 \backslash \Gamma$. Then, 
 $$
 \operatorname{WF}_\hsc v \cap T^* \big(\overline \domain \cap \domain^+\big) \subset \Big\{ (x,\xi), \; \exists t<0, \; x + t \xi \in\widehat \Gamma \Big\}.
 $$
\end{lem}
\begin{proof}
If it is not the case, there exists $(x_0, \xi_0) \in  \operatorname{WF}_\hsc v \cap T^* \big(\overline \domain \cap \domain^+\big)$ so that $\forall t<0$, $x_0 + t \xi_0 \notin \widehat \Gamma$. Therefore, there exists $\tau < 0$ so that  $x_1 := x_0 + \tau \xi_0 \in \Gamma$. As $x_0 \in \overline\domain$, we have necessarily $\xi_0 \cdot n(x_1) <0$
(where, as in Definition \ref{def:outgo}, $n(x_1)$ is the normal at $x_1$ pointing out of $\domain$). On the other hand, as the wavefront set is invariant under the Hamilton flow (this follows, for example, from propagation 
of singularities, \cite[Theorem E.47]{DyZw:19}), $(x_1, \xi_0) \in \operatorname{WF}_\hsc v$. This is a contradiction with the outgoingness of $v$ (Definition \ref{def:outgo}).
\end{proof}

\begin{lem} \label{lem:hyperint}
Assume that $v$ is outgoing near $\Gamma \in \{ \Gamma_s, \Gamma'_s \}$ (in the sense of Definition \ref{def:outgo}). Then, for any $M>0$ and $\epsilon >0$, there exists $c>0$ so that
\begin{equation}
\operatorname{WF}_\hsc (v) \cap T^* \domain^+ \cap T^*\Big((\epsilon, \mathsf{d_l} + \mathsf{d_r} - \epsilon) \times (-M, \height + M)\Big) \subset \big\{(x, (\xi_1, \xi')), \; |\xi_1| \geq c \big\}.
\end{equation}
\end{lem}

\begin{proof}
By a similar proof to the one of Lemma \ref{lem:outflowout}, there exists $c(\epsilon, M)>0$ so that $\operatorname{WF}_\hsc (v) \cap T^* \domain^+ \cap T^*\big((\epsilon, \mathsf{d_l} + \mathsf{d_r} - \epsilon) \times (-M, \height + M)\big)\subset \{(x, (\xi_1, \xi')), \; |\xi_1|/|\xi| \geq c \}$. The result follows using the additional fact that $\operatorname{WF_\hsc} v \cap T^* \domain^+ \subset \{|\xi|^2 = 1\}$ since $(-\hsc^2 \Delta - 1)v = 0$ in $\domain^+$.
\end{proof} 

By the definition of $(\operatorname{WF}_\hsc v)^c$ and the definition of $\mu$, 
 $\operatorname{supp }\mu \subset \operatorname{WF}_\hsc v$ (if $(x_0, \xi_0) \notin \operatorname{WF}_\hsc v$, taking 
 $b$ as in Definition \ref{def:WF} gives $\mu(b) = 0$ using the fact that $b(x, \hsc D_x) v = O(\hsc^\infty)$ in $L^2$, and hence $(x_0, \xi_0) \notin \operatorname{supp} \mu$). Therefore, 
an immediate consequence of Lemma \ref{lem:outflowout} is the following, where
we recall the definition of $\mathcal B_{\rm out}$ from Definition \ref{def:geocor}.

\begin{cor} \label{cor:suppmu}
Assume that $v$ is outgoing near $\Gamma \in \{ \Gamma_s, \Gamma'_s \}$ and let $\widehat \Gamma := \partial \domain \backslash \Gamma$. 
Then, for any defect measure of $v$, 
$$
\operatorname{supp }\mu  \cap T^*\big(\overline \domain \cap \domain^+\big) \subset \mathcal B_{\rm out}(T^*\widehat \Gamma \cap \mathcal H).
$$
\end{cor}

We now define, for $(\mathrm x, \mathrm y) \in \{ i, r, l \}^2
$
\begin{align*}
\mathcal B^{\mathrm x}_{\mathrm y\rightarrow \mathrm x} &:= \big\{ \rho \in T^*\Gamma_{\mathrm x}\cap \mathcal H, \; \exists t<0, \; \pi_x \varphi_t(p^{\rm in})(\rho) \in \Gamma_{\mathrm y} \big\}, \\
\mathcal B^{\mathrm x}_{\mathrm x\rightarrow \mathrm y} &:= \big\{ \rho \in T^*\Gamma_{\mathrm x}\cap \mathcal H, \; \exists t>0, \; \pi_x \varphi_t(p^{\rm out})(\rho) \in \Gamma_{\mathrm y} \big\},
\end{align*}
with $p^{\rm in}$ and $p^{\rm out}$ exchanged in the case $(\mathrm x, \mathrm y) = (i, r)$. 

The sets $\mathcal B^{\mathrm x}_{\mathrm y\rightarrow \mathrm x}$ and $\mathcal B^{\mathrm x}_{\mathrm x\rightarrow \mathrm y}$ are, respectively, the elements of $T^*\Gamma_{\mathrm x}$ reached by the flowout of $\Gamma_{\mathrm y}$, and the elements of $T^*\Gamma_{\mathrm x}$
whose flowout reaches $\Gamma_{\mathrm y}$. The reason why we have to define the case $(\mathrm x, \mathrm y) = (i, r)$ separately is because the measures are taken from $[0,\height ]\times[0,\mathsf{d_\leftsubscript}]$ (i.e., from the left) on $\Gamma_i$, hence when working in $[0,\height ] \times [\mathsf{d_\leftsubscript}, \mathsf{d_\leftsubscript} + \mathsf{d_\rightsubscript}]$ the roles of $\zeta_{\rm in}$ and  $\zeta_{\rm out}$ on $T^*\Gamma_i$ (and therefore of $p^{\rm in}$ and $p^{\rm out}$) are exchanged.

\begin{lem}\label{lem:muinlup} 
$\;$
 \begin{enumerate}
\item If $v$ is outgoing near $\Gamma_s$ (as in Model 1), then, for any defect measures of $v$, 
\begin{enumerate}
\item $ \mu_{\rm in}^{\leftsubscript} = 0$, \label{i:muinlup1}
\item $\mu_{\rm out}^{\middlesubscript}= 0$, \label{i:muinlup2}
\item $\mu_{\rm out}^{\rightsubscript}= 0$, \label{i:muinlup3}
\end{enumerate}
\item If $v$ is outgoing near $\Gamma_t \cup \Gamma_b$  (as in Model 2),  then, for any defect measures of $v$, \\
for any $(\mathrm x, \mathrm y) \in \{ (\rightsubscript,\leftsubscript), (\leftsubscript, \rightsubscript), (\middlesubscript, \leftsubscript), (\leftsubscript, \middlesubscript), (\rightsubscript, i) \}$
\begin{enumerate}
\item $\operatorname{supp} \mu_{\rm in}^{\mathrm x} \subset \mathcal B^{\mathrm x}_{{\mathrm y}\rightarrow {\mathrm x}}$,  \label{i:muinlupa}
\item $\operatorname{supp}\mu_{\rm out}^{\mathrm x} \subset \mathcal B^{\mathrm x}_{\mathrm x\rightarrow \mathrm y}$, \label{i:muinlupd}
\end{enumerate}
with $\operatorname{supp} \mu_{\rm in, out}^{i}$ replaced with ${T^* \Gamma_i} \cap \operatorname{supp} \mu_{\rm in, out}^{i}$ in the case $\mathrm x = i$.
In addition,
\begin{enumerate}
\setcounter{enumii}{2}
\item $T^* \Gamma_i \,\cap\, \operatorname{supp} \mu_{\rm out}^{ i} \subset \mathcal B^{ i}_{{ r}\rightarrow { i}}$. \label{i:muinlupa2}
\end{enumerate}
\end{enumerate}
\end{lem}

\begin{proof}
Part  (\ref{i:muinlup1}) follows from Corollary \ref{cor:suppmu} together with Lemma \ref{lem:interpr}.
Part  (\ref{i:muinlup2}) follows from the same results and the additional observation that
$$
\mathcal B_{\rm out}(T^* \Gamma_{\middlesubscript}\cap \mathcal H) \cap \mathcal B_{\rm out}(T^*\Gamma_{\leftsubscript} \cap \mathcal H) = \emptyset,
$$
and Part  (\ref{i:muinlup3}) is proven in the same way.
Part (\ref{i:muinlupa}) and (\ref{i:muinlupd}) for $(\mathrm x, \mathrm y) \in \{ (\rightsubscript,\leftsubscript), (\leftsubscript, \rightsubscript)\}$ follow again from Corollary \ref{cor:suppmu} and Lemma \ref{lem:interpr}, together  with the facts that
$$
\begin{cases}
\mathcal B^{\mathrm x}_{\mathrm y\rightarrow \mathrm x} = \pi_{\Gamma_{\mathrm x}}\Big( \mathcal B_{\rm in}(T^* \Gamma_{\mathrm x} \cap \mathcal H) \cap \mathcal B_{\rm out}(T^* \Gamma_{\mathrm y}\cap \mathcal H) \cap \{ x\in \Gamma_{\mathrm x}\}\Big), \\
\mathcal B^{\mathrm x}_{\mathrm x\rightarrow \mathrm y} = \pi_{\Gamma_{\mathrm x}}\Big( \mathcal B_{\rm out}(T^* \Gamma_{\mathrm x} \cap \mathcal H) \cap \mathcal B_{\rm in}(T^* \Gamma_{\mathrm y}\cap \mathcal H) \cap \{ x\in \Gamma_{\mathrm x}\}\Big).
\end{cases}
$$
The case $(\mathrm x, \mathrm y) = (r, i)$ follows from the case $(\mathrm x, \mathrm y) = (r, l)$ and the observations that $\mathcal B^{r}_{l \rightarrow r} \subset \mathcal B^{r}_{i \rightarrow r}$ and $\mathcal B^{r}_{r \rightarrow l} \subset \mathcal B^{r}_{r \rightarrow i}$.
The cases $(\mathrm x, \mathrm y)\in \{(\middlesubscript, \leftsubscript), (\leftsubscript, \middlesubscript)\}$, as well as Part (\ref{i:muinlupa2}), are shown in the same way
and using the additional facts that
$$
\mathcal B_{\rm out}(T^* \Gamma_{ l} \cap \mathcal H) \cap \mathcal B_{\rm out}(T^* \Gamma_{ i}\cap \mathcal H) = \emptyset, \hspace{0.5cm}
\mathcal B_{\rm out}(T^* \Gamma_{ r}\cap \mathcal H) \cap \mathcal B_{\rm in}(T^* \Gamma_{ i}\cap \mathcal H) = \emptyset,
$$
while also exchanging the roles of $\mu_{\rm in}$ and $\mu_{\rm out}$ in the case of Part (\ref{i:muinlupa2}).
\end{proof}

We now investigate how defect measures propagate from one interface to another. To do so, we let,  for 
$(\mathrm x, \mathrm y) \in \{{\rightsubscript, \middlesubscript, \leftsubscript}\}^2$,
$d_{\leftsubscript} := 0$, $d_{\rightsubscript} := \mathsf {d_\leftsubscript} + \mathsf{d_{\rightsubscript}}$, $d_{\middlesubscript} := \mathsf {d_\leftsubscript}$, and 
\begin{align*}
\Phi^{\rm in }_{\mathrm x \rightarrow \mathrm y} : \;  & \; T^*(\{d_{\mathrm x}\} \times \mathbb R) &\longrightarrow  &T^*(\{d_{\mathrm y}\} \times \mathbb R) \\ 
&(x',\xi') &\longmapsto
&\; \pi_{T^*(\{ d_{\mathrm y}\} \times \mathbb R)}\Big(\{ \varphi_{t}(0,x',\xi', \zeta_{{\rm in}}(x',\xi')), \; t<0 \} \cap \{ x \in \{ d_{\mathrm y}\}\times\mathbb R \}\Big),
\end{align*}
where $\zeta_{{\rm in}}$ is replaced by $\zeta_{{\rm out}}$ in the case $(\mathrm x, \mathrm y) = (\middlesubscript, \rightsubscript)$ (i.e., $\Phi^{\rm in }_{\mathrm x \rightarrow \mathrm y}$ maps mass on the interface $\mathrm x$ to mass on the interface $\mathrm y$ under the backward flow from incoming mass on $\mathrm x$)
and
\begin{align*}
\Phi^{\rm out }_{\mathrm x \rightarrow \mathrm y} : \;  & \; T^*(\{d_{\mathrm x}\} \times \mathbb R) &\longrightarrow  &T^*(\{d_{\mathrm y}\} \times \mathbb R) \\ 
&(x',\xi') &\longmapsto
&\; \pi_{T^*(\{ d_{\mathrm y}\} \times \mathbb R)}\Big(\{ \varphi_{t}(0,x',\xi', \zeta_{{\rm out}}(x',\xi')), \; t>0 \} \cap \{ x \in \{ d_{\mathrm y}\}\times\mathbb R \}\Big)
\end{align*}
where $\zeta_{{\rm out}}$ is replaced by $\zeta_{{\rm in}}$ in the case $(\mathrm x, \mathrm y) = (\middlesubscript, \rightsubscript)$ 
(i.e., $\Phi^{\rm out }_{\mathrm x \rightarrow \mathrm y}$ maps mass on the interface $\mathrm x$ to mass on the interface $\mathrm y$ under the forward flow from outgoing mass on $\mathrm x$). 
Once again, the reason why we had to define the case $(\mathrm x, \mathrm y) = (i, r)$ separately is because the measures on $\Gamma_i$ are taken from the left.

In the following lemma, we recall the notation that $a \mu(\mathcal{C}) := \int_{\mathcal C} a d\mu$ for a measure $\mu$. 

\begin{lem} \label{lem:ltoi}
For any defect measure of $v$ and any $(\mathrm x, \mathrm y) \in \{\rightsubscript, \middlesubscript, \leftsubscript\}
\times\{{\rightsubscript, \middlesubscript, \leftsubscript}\},
$
$$
\begin{cases}
 \mu^{\mathrm x}_{\rm out}=   \mu^{\mathrm y}_{\rm in}\circ \Phi^{\rm out}_{\mathrm x \rightarrow \mathrm y} \hspace{0.3cm} &\text{ on  }   \mathcal B^{\mathrm x}_{\mathrm x \rightarrow \mathrm y}, \\
\mu^{\mathrm x}_{\rm in}=   \mu^{\mathrm y}_{\rm out} \circ \Phi^{\rm in}_{\mathrm x \rightarrow \mathrm y} \hspace{0.3cm} &\text{ on  }   \mathcal B^{\mathrm x}_{\mathrm y \rightarrow \mathrm x},
\end{cases}
$$
with $\mu^{i}_{\rm out}$ and  $\mu^{i}_{\rm in}$ exchanged in the cases $(\mathrm x, \mathrm y)\in \{ (\middlesubscript, \rightsubscript), (\rightsubscript, \middlesubscript) \}$.
In particular, for any $a\in C^\infty( T^*\mathbb R)$ not depending on $x$
$$
\begin{cases}
a \mu^{\mathrm x}_{\rm out}(\mathcal B) =  a \mu^{\mathrm y}_{\rm in}(\Phi^{\rm out}_{\mathrm x \rightarrow \mathrm y} \mathcal B), \hspace{0.3cm} \tfa  \mathcal B \subset \mathcal B^{\mathrm x}_{\mathrm x \rightarrow \mathrm y}, \\
a \mu^{\mathrm x}_{\rm in}(\mathcal B) =  a \mu^{\mathrm y}_{\rm out}(\Phi^{\rm in}_{\mathrm x \rightarrow \mathrm y} \mathcal B), \hspace{0.3cm} \tfa  \mathcal B \subset \mathcal B^{\mathrm x}_{\mathrm y \rightarrow \mathrm x},
\end{cases}
$$
with $\mu^{i}_{\rm out}$ and  $\mu^{i}_{\rm in}$ exchanged in the cases $(\mathrm x, \mathrm y)\in \{ (\middlesubscript, \rightsubscript), (\rightsubscript, \middlesubscript) \}$. 
\end{lem}

\begin{proof} We prove the lemma for $\Phi_{\mathrm x \rightarrow \mathrm y} := \Phi^{\rm out}_{\mathrm x \rightarrow \mathrm y}$, the proof for 
$\Phi^{\rm in}_{\mathrm x \rightarrow \mathrm y}$ being similar.
In addition, we assume that $(\mathrm x, \mathrm y) \in \{ (\rightsubscript,\leftsubscript), (\leftsubscript, \rightsubscript), (\middlesubscript, \leftsubscript), (\leftsubscript, \middlesubscript)\}$; the proof in the cases $(\mathrm x, \mathrm y)\in \{ (\middlesubscript, \rightsubscript), (\rightsubscript, \middlesubscript) \}$ is the same with the roles of $\zeta_{\rm in}$ and  $\zeta_{\rm out}$ exchanged on $T^*\Gamma_i$. Observe that
$$
\mathcal B_{\rm out}(\mathcal B^{\mathrm x}_{\mathrm x \rightarrow \mathrm y}) = \mathcal B_{\rm in}(\mathcal B^{\mathrm y}_{\mathrm x \rightarrow \mathrm y}).
$$
In $\mathcal V := \mathcal B_{\rm out}(\mathcal B^{\mathrm x}_{\mathrm x \rightarrow \mathrm y}) = \mathcal B_{\rm in}(\mathcal B^{\mathrm y}_{\mathrm x \rightarrow \mathrm y})$, we work in geodesic coordinates given by Definition \ref{def:geocor}, $(\rho_{\mathrm x},t_{\mathrm x}) \in \Big( \mathcal \pi_\Gamma ^ {-1} \mathcal B^{\mathrm x}_{{\mathrm x}\rightarrow {\mathrm y}} \cap \big\{ \zeta = \zeta_{{\rm out}} \big\} \Big) \times \mathbb R_+$, and in geodesic coordinates $(\rho_{\mathrm y},t_{\mathrm y}) \in \Big( \mathcal \pi_\Gamma ^ {-1} \mathcal B^{\mathrm y}_{{\mathrm x}\rightarrow {\mathrm y}} \cap \big\{ \zeta = \zeta_{{\rm in}} \big\} \Big) \times \mathbb R_-$.
Observe that
$$
t_{\mathrm y} = t_{\mathrm x} - \tau(\rho_1),\quad \rho_{\mathrm y}  = p^{\rm in} \Phi_{{\mathrm x}\rightarrow {\mathrm y}} \pi_{\Gamma_{\mathrm x}}(\rho_{\mathrm x}),
$$
for some $\tau(\rho_1)>0$, thus $dt_{\mathrm y} = dt_{\mathrm x} =: dt$. Now, by Lemma \ref{lem:interpr}, in $(\rho_{\mathrm x}, t_{\mathrm x})$ and $(\rho_{\mathrm y}, t_{\mathrm y})$ coordinates, $\mu$ can be written as 
$$
\mu = \big(p^{{\rm out}}_*(2\sqrt r \muout^{\mathrm x})\big)(\rho_{\mathrm x})  \otimes dt = \big(p^{{\rm in}}_*(2\sqrt r \muin^{\mathrm y}))(\rho_{\mathrm y}\big) \otimes dt.
$$
It follows that
$$
p^{{\rm in}}_*(2\sqrt r\muin^{\mathrm y}) = \Phi^{{\mathrm x}\rightarrow {\mathrm y}}_* p^{{\rm out}}_*(2\sqrt r \muout^{\mathrm x}), \quad\text{ where }\quad  \Phi^{{\mathrm x}\rightarrow {\mathrm y}} := p^{\rm in} \Phi_{{\mathrm x}\rightarrow {\mathrm y}} \pi_{\Gamma_{\mathrm x}}.
$$
Let now $a$ be arbitrary and $b := \frac{1}{2\sqrt r}a$. Then
\begin{align} \label{eq:inv2}
\int_{\Phi_{\mathrm x \rightarrow \mathrm y} \mathcal B} a \; d\muin^{\mathrm y} &= \int_{\Phi_{\mathrm x \rightarrow \mathrm y} \mathcal B} b \; d\big(2\sqrt r\muin^{\mathrm y}\big) \nonumber \\
&=  \int_{ p^{\rm in} (\Phi_{\mathrm x \rightarrow \mathrm y} \mathcal B)} b \circ \pi_{\Gamma_{\mathrm y}} \; d\big(p^{{\rm in}}_*(2\sqrt r\muin^{\mathrm y})\big) 
= \int_{ p^{\rm in} (\Phi_{\mathrm x \rightarrow \mathrm y} \mathcal B)} b \circ \pi_{\Gamma_{\mathrm y}} \; d\big(\Phi^{{\mathrm x}\rightarrow {\mathrm y}}_* p^{{\rm out}}_* (2\sqrt r\muout^{\mathrm x})\big) \nonumber \\
&=  \int_{ p^{\rm out} (\mathcal B )} b \circ \pi_{\Gamma_{\mathrm y}} \circ \Phi^{\mathrm x \rightarrow \mathrm y} \; d\big(p^{{\rm out}}_*(2\sqrt r \muout^{\mathrm x})\big) 
= \int_{ \mathcal B } b \circ \pi_{\Gamma_{\mathrm y}} \circ \Phi^{\mathrm x\rightarrow \mathrm y} \circ p^{\rm out} \; d(2\sqrt r\muout^{\mathrm x}) \nonumber \\
&= \int_{ \mathcal B } b \circ \Phi^{\rm out}_{\mathrm x \rightarrow \mathrm y}\; d(2\sqrt r\muout^{\mathrm x}) = \int_{ \mathcal B } 2\sqrt r \times b \circ \Phi^{\rm out}_{\mathrm x \rightarrow \mathrm y}\; d\muout^{\mathrm x} ,
\end{align}
where we used the change-of-variable formula (\ref{eq:puschchange}), the fact that
$$
\Phi^{\mathrm x\rightarrow \mathrm y}\Big(p^{\rm out} \mathcal B \Big) = p^{\rm in} (\Phi_{\mathrm x \rightarrow \mathrm y} \mathcal B),
$$
and
$$
\pi_{\Gamma_{\mathrm y}} \circ \Phi^{{\mathrm x}\rightarrow {\mathrm y}} \circ p^{\rm out}=  \Phi^{\rm out}_{\mathrm x \rightarrow \mathrm y}.
$$
Now, as the Hamiltonian flow $\varphi_t$ consists of straight lines and $\Gamma_{\mathrm x}$ and $\Gamma_{\mathrm y}$ are parallel (straight) segments, we have, using the fact that $r$ depends only on $\xi'$,
$$
2\sqrt r \times b \circ \Phi^{\rm out}_{\mathrm x \rightarrow \mathrm y} = (2\sqrt r b) \circ \Phi^{\rm out}_{\mathrm x \rightarrow \mathrm y} = a \circ \Phi^{\rm out}_{\mathrm x \rightarrow \mathrm y} 
,$$
and the first part of the result follows by (\ref{eq:inv2}). The second part follows again from the fact that the Hamiltonian flow $\varphi_t$ consists of straight lines and $\Gamma_{\mathrm x}$ and $\Gamma_{\mathrm y}$ are parallel segments.
\end{proof}

To finish this section, we define auxiliary measures in the following way.

\begin{definition} \label{def:aux}
If $v$ admits defect measures and boundary defect measures, we define the auxiliary measure $\widetilde \mu$ as
$$
\widetilde \mu := (\xi_1 + 1)^2 \mu.
$$
In addition, we let

\noindent\begin{minipage}{0.5\linewidth}
$$
\widetilde \mu_{\rm out}^{\mathrm x} :=
\begin{cases}
( + \sqrt r + 1)^2 \mu_{\rm out}^{\leftsubscript} & \text{ for } {\mathrm x} = \leftsubscript, \\
( - \sqrt r + 1)^2 \mu_{\rm out}^{\mathrm x} & \text{ for } {\mathrm x} \in \{\middlesubscript, \rightsubscript\},
\end{cases}
$$
\end{minipage}%
\begin{minipage}{0.5\linewidth}
$$
\widetilde \mu_{\rm in}^{\mathrm x} :=
\begin{cases}
( - \sqrt r + 1)^2 \mu_{\rm in}^{\leftsubscript} & \text{ for } {\mathrm x} = \leftsubscript, \\
( + \sqrt r + 1)^2 \mu_{\rm in}^{\mathrm x} & \text{ for } {\mathrm x} \in \{\middlesubscript, \rightsubscript\},
\end{cases}
$$
\end{minipage}\par\vspace{\belowdisplayskip}
\noindent and define, for ${\mathrm x} \in \{\leftsubscript,  \middlesubscript, \rightsubscript\}$, $\widetilde \nu^{\mathrm x}_{d}$ to be the 
Dirichlet measure associated to $(\hsc D_{x_1}+1)u$ on $\Gamma_{\mathrm x}$.
\end{definition}

The heuristic is that these measures correspond to measures of $(\hsc D_{x_1} + 1)u$. Although they are not defined in this way through Theorem \ref{th:ex_defect}, we now show that they satisfy all the expected properties.

\begin{lem} \label{lem:tilde_good}
Assume that $v$ is outgoing near $\Gamma \in \{ \Gamma'_{s}, \Gamma_{s} \}$ and admits defect measures and boundary defect measures.
\begin{enumerate}
\item \label{i:same} The auxiliary measures $\widetilde \mu^{\mathrm x}_{\rm out/in}$ and $\widetilde \nu^{\mathrm x}_{d}$ satisfy the conclusions of Lemma \ref{lem:muinlup} and \ref{lem:ltoi}.
\item \label{i:dir_bc} For ${\mathrm x} \in \{\leftsubscript,  \rightsubscript\}$, if $\widetilde \nu^{\mathrm x}_d = 0$ on some Borel set $\mathcal B \subset \mathcal H$, then, on $\mathcal B$, $\widetilde \mu_{\rm in}^{\mathrm x}= \widetilde \mu_{\rm out}^{\mathrm x}$. 
\item \label{i:nu_out} If $\widetilde \mu_{\rm in}=0$ on some Borel set $\mathcal{B}\subset \mathcal{H}$, then, on $\mathcal B$, 
$
\widetilde \mu_{\rm out}= 2\sqrt{r}\widetilde \nu_{d}
$.
Similarly, if $\widetilde \mu_{\rm out}=0$ on some Borel set $\mathcal{B}\subset \mathcal{H}$, then, on $\mathcal B$,  
$
\widetilde \mu_{\rm in}=  2\sqrt{r}\widetilde \nu_{d}
$.
\end{enumerate}
\end{lem}

\begin{proof}
We first show (1).
Observe that
$\operatorname{supp}\mu \subset \operatorname{supp}\widetilde \mu \cup \{ \xi_1 = - 1 \}$.
We now show that $\{\xi_1=-1\}\cap \operatorname{supp}\mu=\emptyset$. 
If $v$ is outgoing near $\Gamma_{s}$, this follows directly from Corollary \ref{cor:suppmu}.
If $v$ is outgoing near $\Gamma'_{s}$, by Lemma \ref{lem:Millretations}, $\mu_{\rm out}^{\rightsubscript} = \big| \frac{\sqrt r - 1}{\sqrt r + 1}\big|^2\mu_{\rm in}^{\rightsubscript}$, 
and thus $\mu_{\rm out}^\rightsubscript (\sqrt{r}=1)=0$, and by the interpretation of boundary measures given by Lemma \ref{lem:interpr} we have $\{\xi_1=-1\}\cap\operatorname{supp}\mu =\emptyset$ as well. It follows that $\operatorname{supp}\mu = \operatorname{supp}\widetilde \mu$. Hence,
$\widetilde \mu$ satisfies the conclusions of Corollary \ref{cor:suppmu}. By construction, it also satisfies the conclusions of Lemma \ref{lem:interpr}. Therefore, since these were the only properties used in the proofs of Lemma \ref{lem:muinlup} and \ref{lem:ltoi}, the auxiliary measures satisfy the conclusions of Lemma \ref{lem:muinlup} and \ref{lem:ltoi} (with the exact same proofs as for $\mu$).

We now show (2), for example with $\mathrm x = \leftsubscript$. By definition of the defect measures, on $\mathcal B$
$$
0 = \widetilde \nu^{\leftsubscript}_d = \nu^{\leftsubscript}_d - 2 \Re\nu_j^{\leftsubscript} + \nu_n^{\leftsubscript}.
$$
By the Cauchy-Schwarz inequality and a similar reasoning as in the proof of \cite[Lemma 3.3]{GSW},
$$
\left|\nu_j^{\leftsubscript}\right| \leq \left|\sqrt{\nu^{\leftsubscript}_d }\right|\left|\sqrt{\nu_n^{\leftsubscript}}\right|,
$$
Therefore, on   $\mathcal B$, $\nu^{\leftsubscript}_n = \nu^{\leftsubscript}_d = - \nu^{\leftsubscript}_j$.
Hence, from Lemma \ref{lem:Millretations} (with $\alpha = 2$), on $\mathcal B$,
$$
 \mu^{\leftsubscript}_{\rm out} = \bigg| \frac{\sqrt r - 1}{\sqrt r + 1}\bigg|^2 \mu^{\leftsubscript}_{\rm in},
$$
which, by Definition \ref{def:aux}, is
 $\widetilde \mu_{\rm in}^{\leftsubscript}= \widetilde \mu_{\rm out}^{\leftsubscript}$. % by definition.

Finally, (3) is a direct consequence of the similar property for $\mu$ from Lemma \ref{lem:Millretations} together with the
fact that $\{\xi_1=-1\}\cap \operatorname{supp}\mu=\emptyset$.
\end{proof} 

\section{Proof of result involving \ref{eq:model_cell}, i.e., Theorem \ref{th:model1}} \label{sec:6}

We fix an admissible solution operator $S : g \mapsto u$ to (\ref{eq:modelh}).
The aim of this section is to prove the following theorem, which implies Theorem \ref{th:model1}.

\begin{theorem} \label{th:lowerupper}
Let $\iota \in \{ 1, -1\}$. Then
\begin{enumerate}
\item \label{i:1_up} For any $\epsilon>0$, there exists $\hsc_0(\epsilon)>0$ such that, for any $\hsc$--family of data $(g(\hsc))_{\hsc>0}$, if $u(\hsc)=S(\hsc)g(\hsc)$ is the associated solution to (\ref{eq:modelh}), then for all $0<\hsc\leq h_0$
\begin{equation} \label{eq:uppermain}
\Vert (\hsc D_{x_1}  + \iota)u(\hsc)\Vert_{L^2(\Gamma_{\middlesubscript})} \leq  \Bigg( \frac{1+ \iota \cos (\theta_{\rm max})}{1 + \cos (\theta_{\rm max})} + \epsilon \Bigg)
 \Vert g(\hsc) \Vert_{L^2}.
\end{equation} 
\item \label{i:1_low} For all $0 < \epsilon < \theta_{\rm max}$, there exists $g(\hsc) \in L^2(\Gamma_{\leftsubscript})$ so that, the associated solution $u(\hsc)=S(\hsc)g(\hsc)$ to (\ref{eq:modelh}) satisfies
\begin{equation} \label{eq:lowermain}
\lim_{\hsc\rightarrow 0} \frac{\Vert (\hsc D_{x_1} + \iota)u(\hsc) \Vert_{L^2(\Gamma_{\middlesubscript})} }{\Vert g(\hsc) \Vert_{L^2(\Gamma_{\leftsubscript})}} = \frac{1 + \iota\cos(\theta_{\rm max} - \epsilon)}{1+\cos(\theta_{\rm max} - \epsilon)}.
\end{equation}
\end{enumerate}
\end{theorem}

Both in this section and in \S\ref{sec:7}, we use the following consequence of the results of \S\ref{sec:5} (see \S\ref{ss:int_trace}).

\begin{lem} \label{lem:good_mes}
Let $g(\hsc)\in L^2(\Gamma_l)$ be 
such that $\Vert g \Vert_{L^2}\leq C$, and $u(\hsc) = S(\hsc)g(\hsc)$ the associated solution to (\ref{eq:modelh}) (resp.~(\ref{eq:modelh_timp})). Then, up to a subsequence, $u(\hsc)$ has a defect measure
and boundary defect measures on $\Gamma_l$ and $\Gamma_i^+$ (resp.~$\Gamma_l$, $\Gamma_i^+$ and $\Gamma_r$) satisfying $\Vert (\hsc D_{x_1} u + \iota) u \Vert_{L^2(\Gamma_i)} \rightarrow \widehat \nu_d^i({T^*\Gamma_i})$
(i.e., (\ref{eq:osc_Gammai})  below), where  $\widehat \nu_d^i := \nu^{\middlesubscript}_d - 2 \iota \Re\nu_j^{\middlesubscript} + \nu_n^{\middlesubscript}$ is the associated Dirichlet boundary measure of $\widehat u := (\hsc D_{x_1} u + \iota)u$, $\iota \in \{ -1, 1\}$.
\end{lem}

\subsection{The impedance trace on $\Gamma_{\middlesubscript} $ in terms of measures for \ref{eq:model_cell}} \label{subsec:measth}

\begin{prop} \label{prop:moregeneral}
Assume that $u$ satisfies
$$
(-\hsc^2\Delta - 1)  u = 0  \text{ in }\domain;
$$
and, for $\iota \in \{ 1,-1 \}$, let $\widehat u := (\hsc D_{x_1}+\iota)u$. If $u$ is outgoing near $\Gamma_s$ and 
admits defect measure and
boundary defect measures, then $\widehat u$ admits a Dirichlet boundary measure on $\Gamma_{\middlesubscript}^+$, and in ${T^* \Gamma_i}$ we have
\beq\label{eq:JJS0}
\widehat \nu_d^{\middlesubscript} = \frac{1}{2\sqrt r}\bigg( \frac{\iota + \sqrt r}{1+\sqrt r}\bigg)^2 \widetilde \mu_{\rm out}^{\leftsubscript} \circ \Phi^{\rm in}_{\middlesubscript\rightarrow \leftsubscript}.
\eeq
\end{prop}

\begin{proof}
By Part (1) of Lemma \ref{lem:muinlup},  $ \mu^{\middlesubscript}_{\rm out} = 0$. Thus, by Part (3) of Lemma \ref{lem:Millretations}, 
\begin{equation} \label{eq:s1}
 \nu^{\middlesubscript}_n = \frac 12 \sqrt{r}  \mu^{\middlesubscript}_{\rm in}, \quad  \Re  \nu^{\middlesubscript}_j = -\frac 12 \mu^{\middlesubscript}_{\rm in}, \quad\tand\quad  \nu^{\middlesubscript}_d = \frac{1}{2\sqrt{r}}  \mu^{\middlesubscript}_{\rm in} \hspace{0.5cm}\text{ in } \mathcal H.
\end{equation}
On the other hand, for $\widetilde u := (\hsc D_{x_1}+1)u$, the definition of $\widetilde u$ and $\widehat u$ in terms of $u$ imply that $\widetilde u$ and $\widehat u$ have Dirichlet boundary defect measures on $\Gamma_i^+$ given by (this notation coincide with Definition \ref{def:aux} in the case of $\widetilde \nu^{\middlesubscript}_d$)
\begin{align*}
\widetilde \nu^{\middlesubscript}_d = \nu_d^{\middlesubscript} - 2 \Re\nu_j^{\middlesubscript} + \nu_n^{\middlesubscript} \quad\tand\quad \widehat \nu^{\middlesubscript}_d = \nu^{\middlesubscript}_d - 2 \iota \Re\nu_j^{\middlesubscript} + \nu_n^{\middlesubscript}
\end{align*}
(where we recall that $s=-x_1$ on $\Gamma_{\middlesubscript}$).
Hence, by (\ref{eq:s1}), in $\mathcal H$
\begin{align*}
\widetilde \nu^{\middlesubscript}_d = \frac{(\sqrt r + 1)^2}{2\sqrt{r}} \mu^{\middlesubscript}_{\rm in}\quad\tand\quad
\widehat \nu^{\middlesubscript}_d = \frac{(\sqrt r + \iota)^2}{2\sqrt{r}} \mu^{\middlesubscript}_{\rm in},
\end{align*}
and thus, in $\mathcal H$
\begin{equation} \label{eq:s2}
\widehat \nu^{\middlesubscript}_d = \bigg( \frac{\sqrt r + \iota}{\sqrt r + 1} \bigg)^2 \widetilde \nu^{\middlesubscript}_d.
\end{equation}
On the other hand, by Part (\ref{i:same}) of Lemma \ref{lem:tilde_good} and
Part (1) of  Lemma \ref{lem:muinlup} again,  $ \widetilde \mu^{\middlesubscript}_{\rm out} = 0$ and thus, by 
Part (\ref{i:nu_out} of Lemma \ref{lem:tilde_good}, 
\beq\label{eq:JJS1}
\widetilde \nu^{\middlesubscript}_d = \frac{1}{2\sqrt{r}}  \widetilde \mu^{\middlesubscript}_{\rm in} \quad\tin \mathcal{H}.
\eeq
By Lemma \ref{lem:traceH}, $\widehat \nu^{\middlesubscript}_d$ is supported in $\mathcal H$;  
thus the combination of \eqref{eq:s2} and \eqref{eq:JJS1} implies that
\beqs
\widehat \nu^{\middlesubscript}_d = \bigg( \frac{\sqrt r + \iota}{\sqrt r + 1} \bigg)^2 \frac{1}{2\sqrt{r}}  \widetilde \mu^{\middlesubscript}_{\rm in}.
\eeqs
Finally, by propagation of defect measures given by Lemma \ref{lem:ltoi} together with Lemma \ref{lem:muinlup}, Part (\ref{i:muinlupa}) (recalling Part (\ref{i:same}) of Lemma \ref{lem:tilde_good}), with $(\mathrm x, \mathrm y) = (i, l)$, 
$$
 \widetilde \mu^{\middlesubscript}_{\rm in}= \widetilde \mu_{\rm out}^{\leftsubscript} \circ \Phi^{\rm in}_{\middlesubscript\rightarrow \leftsubscript}
 \quad\tin T^* \Gamma_i,
 $$
 and the result \eqref{eq:JJS0} follows.
\end{proof}

A first application of Proposition \ref{prop:moregeneral} is the following.

\begin{cor} \label{cor:meas} Assume that $u = Sg$ solves (\ref{eq:modelh}) with $\Vert g \Vert_{L^2}\leq C$, and admits defect measures and boundary 
defect measures given by Lemma \ref{lem:good_mes},
 and let $\widetilde u := (\hsc D_{x_1}+1)u$. Then,
$$
\lim_{\ell \rightarrow \infty} \Vert (\hsc_\ell D_{x_1} + \iota )u(\hsc_\ell)\Vert^2_{L^2(\Gamma_{\middlesubscript})} = \int_{\mathcal B^{\leftsubscript}_{\leftsubscript\rightarrow \middlesubscript}} \Big( \frac{\iota + \sqrt r}{1+\sqrt r}\Big)^2\; d\widetilde \nu^{\leftsubscript}_d, \hspace{0.3cm}\iota = 1,-1.
$$
\end{cor} 
\begin{proof}
Part (\ref{i:same}) of Lemma \ref{lem:tilde_good} and Part (1) of Lemma \ref{lem:muinlup} imply that  $\widetilde \mu_{\rm in}^{\leftsubscript} = 0$, and thus 
Part (\ref{i:nu_out}) of Lemma \ref{lem:tilde_good} 
implies that, in $\mathcal H$,
${\widetilde {\mu}}_{\rm out}^{\leftsubscript} = 2\sqrt{r} {\widetilde \nu^{\leftsubscript}_d}$; the result then follows by Proposition \ref{prop:moregeneral} together with (\ref{eq:osc_Gammai}) from Lemma \ref{lem:good_mes}.
\end{proof}

\subsection{Proof of Theorem \ref{th:lowerupper}}
\begin{proof}
{We first prove Point (\ref{i:1_up}) of Theorem \ref{th:lowerupper} (i.e. the upper bound \eqref{eq:uppermain}), then Point (\ref{i:1_low}) (i.e. the lower bound \eqref{eq:lowermain})}. Let $\iota \in \{ -1, 1\}$ {and let $\epsilon>0$}.
\subsection*{The upper bound}
If the upper bound \eqref{eq:uppermain} does not hold, then there exists $\hsc_\ell \rightarrow 0$, $g_\ell(h_\ell) \in L^2$ and $u(\hsc_\ell) = Sg(\hsc)$ solving  (\ref{eq:modelh})  and
\begin{equation} \label{eq:lower_contradictmegen}
\Vert (\hsc D_{x_1}  + \iota)u(\hsc_\ell)\Vert_{L^2(\Gamma_{\middlesubscript})} {>}  \Bigg(\frac{1+ \iota \cos (\theta_{\rm max})}{1 + \cos (\theta_{\rm max})} + \epsilon \Bigg)
 \Vert g_\ell(\hsc_\ell) \Vert_{L^2}.
\end{equation}
By rescaling, we can assume that
\begin{equation} \label{eq:low:normggen}
 \Vert g_\ell(\hsc_\ell) \Vert_{L^2} = 1.
\end{equation}
Then, by Lemma \ref{lem:good_mes}, up to extraction of a subsequence, $u(\hsc_\ell)$ admits defect measure and boundary defect measures.
Now, by Corollary \ref{cor:meas}, 
\begin{equation} \label{eq:uppfinfin}
\lim_{\ell \rightarrow \infty} \Vert (\hsc_\ell D_{x_1} + \iota1)u(\hsc_\ell)\Vert^2_{L^2(\Gamma_{\middlesubscript})} = \int_{\mathcal B^{\leftsubscript}_{\leftsubscript\rightarrow \middlesubscript}} \Big( \frac{\iota+ \sqrt r}{1+\sqrt r}\Big)^2\; d\widetilde \nu^{\leftsubscript}_d
\end{equation}
But, 
$$
\bigg( \frac{\sqrt{r} + \iota}{\sqrt{r}+1} \bigg)^2 = \bigg( \frac{\sqrt{1- \xi'^2} + \iota}{\sqrt{1- \xi'^2}+1} \bigg)^2 
$$
which is non-decreasing for $\xi' \in [0,1]$ and
$$
|\xi'| \leq \sin (\theta_{\rm max}) \hspace{0.3cm}\text{on } \mathcal B^{\leftsubscript}_{\leftsubscript\rightarrow \middlesubscript}.
$$
Therefore, by (\ref{eq:uppfinfin})
\begin{align*}
\lim_{\ell \rightarrow \infty} \Vert (\hsc_\ell D_{x_1} + \iota)u(\hsc_\ell)\Vert^2_{L^2(\Gamma_{\middlesubscript})} &\leq \bigg( \frac{\iota  +  \cos (\theta_{\rm max})}{1  + \cos (\theta_{\rm max})} \bigg)^2 \widetilde \nu^{\leftsubscript}_d
(\mathcal B^{\leftsubscript}_{\middlesubscript\rightarrow \leftsubscript}) \\
&\leq \bigg( \frac{\iota  +  \cos (\theta_{\rm max})}{1  + \cos (\theta_{\rm max})} \bigg)^2 \widetilde \nu^{\leftsubscript}_d
(T^*\Gamma_{\leftsubscript}) \\
&\leq \bigg( \frac{1  + \iota \cos (\theta_{\rm max})}{1  + \cos (\theta_{\rm max})} \bigg)^2 \bigg(\limsup_{\ell \to \infty}\Vert g_\ell(\hsc_\ell) \Vert_{L^2}^2\bigg)\\
&= \bigg( \frac{1  + \iota \cos (\theta_{\rm max})}{1  + \cos (\theta_{\rm max})} \bigg)^2, 
\end{align*}
where we used (\ref{eq:low:normggen}) on the last line. {This last inequality is contradicted by taking the limit $\ell \rightarrow \infty$ in (\ref{eq:lower_contradictmegen})}, 
and thus the upper bound \eqref{eq:uppermain} holds.

\subsection*{The lower bound}
{We now want to show Point (\ref{i:1_low}) of Theorem \ref{th:lowerupper}. To do so, 
we construct explicitly an $\hsc$-family of functions $g(\hsc)$ so that the associated solution $u(\hsc)=S(\hsc)g(\hsc)$ to (\ref{eq:modelh}) satisfies  (\ref{eq:lowermain}).}
Let $y_0 = (0, z_0) \in \mathring \Gamma_{\leftsubscript}$ and $\xi'_0 \in T^*\Gamma_{\leftsubscript}$ to be fixed later. {The idea is to chose the data $g(\hsc)$ specifically so that the Dirichlet measure of $\widetilde{u} := (\hsc D_{x_1} +1)u$ is a Dirac delta function in $(y_0, \xi_0)$, and then chose $(y_0, \xi_0)$ to saturate the upper bound, i.e., obtain (\ref{eq:lowermain}).} Let $\eta>0$ be such that $z_0 \in  [\eta, L-\eta]$, and $\chi \in C^\infty_c (\mathbb R)$ be such that $\chi = 1$ near $z_0$ and $\operatorname{supp } \chi \subset [\frac 12 \eta, L- \frac 12 \eta]$.  We define $g(\hsc) \in L^2(\Gamma_{\leftsubscript})$, $\operatorname{supp }g \Subset \Gamma_{\leftsubscript}$ by
\beq\label{eq:data_g}
g(\hsc) := (\pi \hsc)^{-1/4} e^{i (y-y_0)\cdot \xi'_0/\hsc} e^{-i(y-y_0)^2/2\hsc} \chi(y).
\eeq
Let $u$ be the associated solution of (\ref{eq:modelh}), and let $\widetilde u := (\hsc D_{x_1} +1)u$.

By Lemma \ref{lem:good_mes} (using the fact that $\Vert g \Vert_{L^2} \leq C$) 
any sequence $u(\hsc_\ell)$ admits, up to extraction of a subsequence, defect measure and boundary defect measures on $\Gamma_{\leftsubscript}$ and $\Gamma_{\middlesubscript}^+$. Integrating by parts (see for example by \cite[Page 102, Example 1]{Zw:12}), 
{$g$ has defect meausre $\delta_{\xi' = \xi'_0, y=y_0}$, 
and thus, by the boundary condition on $\widetilde u$,} we have that,
\begin{equation} \label{eq:gaussnu}
\widetilde {\nu}^{\leftsubscript}_d = \delta_{\xi' = \xi'_0, y=y_0},
\end{equation}
and in addition, as $\hsc \rightarrow 0$,
\begin{equation} \label{eq:gaussmass}
\Vert g(\hsc) \Vert_{L^2} \rightarrow 1.
\end{equation}

We show that, for any sequence $\hsc_\ell \rightarrow 0$, up to a subsequence,
\begin{equation} \label{eq:lowermain2}
\lim_{\ell\rightarrow\infty} \frac{\Vert (\hsc_\ell D_{x_1} + \iota)u(\hsc_\ell) \Vert_{L^2(\Gamma_{\middlesubscript})} }{\Vert g(\hsc_\ell) \Vert_{L^2(\Gamma_{\leftsubscript})}} = \frac{1+ \iota\cos(\theta_{\rm max} - \epsilon)}{1+\cos(\theta_{\rm max} - \epsilon)},
\end{equation}
from which {(\ref{eq:lowermain})} follows; indeed, if (\ref{eq:lowermain}) fails, there exists a sequence $\hsc_\ell \rightarrow 0$ so that
$$
\liminf_{\ell\rightarrow \infty} \left| \frac{\Vert (\hsc_\ell D_{x_1} + \iota)u(\hsc_\ell) \Vert_{L^2(\Gamma_{\middlesubscript})} }{\Vert g(\hsc_\ell) \Vert_{L^2(\Gamma_{\leftsubscript})}} - \frac{1+ \iota\cos(\theta_{\rm max} - \epsilon)}{1+\cos(\theta_{\rm max} - \epsilon)} \right| > 0,
$$
which contradicts (\ref{eq:lowermain2})

Take an arbitrary sequence $\hsc_\ell \rightarrow 0$. 
By Corollary \ref{cor:meas}  together with (\ref{eq:gaussnu}), 
\begin{equation} \label{eq:upperfinfin}
\lim_{\ell \rightarrow \infty} \Vert (\hsc_\ell D_{x_1} + \iota)u(\hsc_\ell)\Vert^2_{L^2(\Gamma_{\middlesubscript})} = \int_{\mathcal B^{\leftsubscript}_{\leftsubscript\rightarrow \middlesubscript}} \Big( \frac{\iota+\sqrt r}{1+\sqrt r}\Big)^2\; d \delta_{\xi' = \xi'_0, y=y_0}.
\end{equation}
Let
$$
\xi'_0 := \sin \theta \quad\text{ with }\quad \theta := \theta_{\rm max} - \epsilon.
$$
Since $\theta<\theta_{\rm max}$, there exists $y'_0 \in \mathring \Gamma_{\leftsubscript}$ so that $(\xi'_0, y'_0) \in \mathcal B^{\leftsubscript}_{\middlesubscript\rightarrow \leftsubscript}$. 
With the data $g$ given by \eqref{eq:data_g} with these $(\xi'_0,y'_0)$, 
the desired result 
\eqref{eq:lowermain2} follows from (\ref{eq:upperfinfin}), (\ref{eq:gaussmass}), and the fact that $r(\xi') = 1 - \xi'^2$.
\end{proof}

\section{Proof of the results involving \ref{eq:model_cell_timp}; i.e. Theorem \ref{th:model2} and Theorem \ref{th:compo}}\label{sec:7}

Similarly to the start of \S\ref{sec:6}, we fix admissible solution operators $S(\height , \mathsf {d_\leftsubscript}^+, \mathsf {d_{\rightsubscript}}^+) : g \mapsto u$ and $S(\height , \mathsf {d_\leftsubscript}^-, \mathsf {d_{\rightsubscript}}^-) : g \mapsto u$ to (\ref{eq:modelh_timp}) (with the two sets of parameters $(\height , \mathsf {d_\leftsubscript}^\pm, \mathsf {d_{\rightsubscript}}^\pm)$) throughout this section.

\subsection{The impedance trace on $\Gamma_{\middlesubscript}$ in term of measures for \ref{eq:model_cell_timp}}

\begin{prop} \label{prop:trace_meas2}
Assume that $u$ satisfies (\ref{eq:modelh_timp}) and has defect measure and boundary defect measures.
For $\iota \in \{ 1,-1 \}$, let $\widehat u := (\hsc D_{x_1}+\iota)u$. Then, up to a subsequence, $\widehat u$ admits a Dirichlet boundary measure $\widehat \nu_d^{\middlesubscript}$  on $\Gamma_{\middlesubscript}^+$ and
\begin{equation} \label{eq:supphat}
\operatorname{supp} \widehat \nu_d^{\middlesubscript} \, \cap\, {T^* \Gamma_i} \subset \operatorname{supp} \big(\widetilde \mu_{\rm out}^{\leftsubscript} \circ \Phi^{\rm in}_{\middlesubscript\rightarrow \leftsubscript} \big)\cup \operatorname{supp} \big(\widetilde \mu_{\rm out}^{\rightsubscript} \circ \Phi^{\rm in}_{\middlesubscript\rightarrow \rightsubscript}\big).
\end{equation}
If, in addition,
\begin{equation} \label{eq:mdisjsupp}
\operatorname{supp} \big(\widetilde \mu_{\rm out}^{\leftsubscript} \circ \Phi^{\rm in}_{\middlesubscript\rightarrow \leftsubscript} \big)\cap \operatorname{supp} \big(\widetilde \mu_{\rm out}^{\rightsubscript} \circ \Phi^{\rm in}_{\middlesubscript\rightarrow \rightsubscript}\big) \,\cap\,{T^* \Gamma_i} = \emptyset,
\end{equation}
then, in ${T^* \Gamma_i}$,
\begin{equation} \label{eq:exprhat}
\widehat \nu_d^{\middlesubscript} = \frac{1}{2\sqrt r}\Big( \frac{\iota + \sqrt r}{1+\sqrt r}\Big)^2 \widetilde \mu_{\rm out}^{\leftsubscript} \circ \Phi^{\rm in}_{\middlesubscript\rightarrow \leftsubscript} +  \frac{1}{2\sqrt r}\Big( \frac{ - \iota + \sqrt r}{-1+\sqrt r}\Big)^2 \widetilde \mu_{\rm out}^{\rightsubscript} \circ \Phi^{\rm in}_{\middlesubscript\rightarrow \rightsubscript}.
\end{equation}
\end{prop}
The support property (\ref{eq:supphat}) reflects the fact that there are two contributions to the impedance trace on $\Gamma_\middlesubscript$:~mass coming from the left boundary $\Gamma_\leftsubscript$, and mass coming from the right boundary $\Gamma_\rightsubscript$. Both of these contributions a priori interact. However, when they \emph{don't} interact, 
i.e., under the disjoint supports assumption (\ref{eq:mdisjsupp}), the impedance trace on $\Gamma_\middlesubscript$ corresponds precisely to the sum of these contributions, as expressed by (\ref{eq:exprhat}), which is the analogue of (\ref{eq:JJS0}) in the present case of (\ref{eq:modelh_timp}).

\begin{proof}
By the definition of $\widehat u$ in terms of $u$ and the definition of defect measures, $\widehat u$ has a Dirichlet boundary measure in $\Gamma_{\middlesubscript}^+$ and
\begin{equation} \label{eq:np1}
\widehat \nu^{\middlesubscript}_d = \nu^{\middlesubscript}_d - 2 \iota \Re\nu_j^{\middlesubscript} + \nu_n^{\middlesubscript}
\end{equation}
(exactly the same as at the beginning of the proof of Proposition \ref{prop:moregeneral}).
Hence,
$$
\operatorname{supp}\widehat \nu^{\middlesubscript}_d \subset \operatorname{supp} \nu^{\middlesubscript}_d \cup \operatorname{supp}  \nu_j^{\middlesubscript} \cup\operatorname{supp} \nu_n^{\middlesubscript}.
$$
But, by Lemma \ref{lem:Millretations}, 
$$
\Big( \operatorname{supp} \nu^{\middlesubscript}_d \cup \operatorname{supp} \nu_j^{\middlesubscript} \cup \operatorname{supp} \nu_n^{\middlesubscript} \Big) \cap \mathcal H \subset \operatorname{supp} \mu^{\middlesubscript}_{\rm out} \cup \operatorname{supp} \mu^{\middlesubscript}_{\rm in}. 
$$
{The well-posedness results in \S\ref{sec:wp} (whose proofs are delayed to \S\ref{sec:5}) imply that}
$\nu^{\middlesubscript}_d, \nu^{\middlesubscript}_j$, and $\nu^{\middlesubscript}_n$ are all supported in $\mathcal H$ {(see Lemma \ref{lem:traceH} below)}, and thus
\begin{equation} \label{eq:np2}
\operatorname{supp}\widehat \nu^{\middlesubscript}_d \subset \operatorname{supp} \mu^{\middlesubscript}_{\rm out} \cup \operatorname{supp} \mu^{\middlesubscript}_{\rm in}.
\end{equation}
We now seek to relate $\mu^{\middlesubscript}_{\rm out/in}$ to $\widetilde{\mu}^{\middlesubscript}_{\rm out/in}$ using the relationship between $\mu$ and $\widetilde{\mu}$ and Lemma \ref{lem:interpr}.
Since $ \widetilde \mu = (\xi_1+1)^2 \mu$, 
$\operatorname{supp}\mu \subset \operatorname{supp}\widetilde \mu \cup \{ \xi_1 = - 1 \}$.
We now show that $\{\xi_1=-1\}\cap \operatorname{supp}\mu=\emptyset$. Indeed, by Lemma \ref{lem:Millretations}, $\mu_{\rm out}^{\rightsubscript} = \big| \frac{\sqrt r - 1}{\sqrt r + 1}\big|^2\mu_{\rm in}^{\rightsubscript}$, 
and thus $\mu_{\rm out}^\rightsubscript (\sqrt{r}=1)=0$. Therefore, by the interpretation of boundary measures given by Lemma \ref{lem:interpr} that $\{\xi_1=-1\}\cap\operatorname{supp}\mu =\emptyset$.
 Hence $\operatorname{supp}\mu \subset \operatorname{supp}\widetilde \mu$, from which, by Lemma \ref{lem:interpr} again, $\operatorname{supp}\mu^{\middlesubscript}_{\rm in / out} \subset \operatorname{supp}\widetilde \mu^{\middlesubscript}_{\rm in / out}$. Therefore, from (\ref{eq:np2}) we get
\begin{equation} \label{eq:np2p}
\operatorname{supp}\widehat \nu^{\middlesubscript}_d \subset \operatorname{supp} \widetilde \mu^{\middlesubscript}_{\rm out} \cup \operatorname{supp} \widetilde \mu^{\middlesubscript}_{\rm in}.
\end{equation}
Finally, observe that, by Part (\ref{i:same}) of Lemma \ref{lem:tilde_good} and propagation of defect measures given by Lemma \ref{lem:ltoi} together with Lemma \ref{lem:muinlup}, Parts (\ref{i:muinlupa2}) and (\ref{i:muinlupa}) with $(\mathrm x, \mathrm y) = (i, l)$,
\begin{equation} \label{eq:np3}
 \widetilde \mu^{\middlesubscript}_{\rm in}= \widetilde \mu_{\rm out}^{\leftsubscript} \circ \Phi^{\rm in}_{\middlesubscript\rightarrow \leftsubscript}, \hspace{0.5cm} \widetilde \mu^{\middlesubscript}_{\rm out}= \widetilde \mu_{\rm out}^{\rightsubscript} \circ \Phi^{\rm in}_{\middlesubscript\rightarrow \rightsubscript}\quad\tin T^*\Gamma_i.
\end{equation}
Then \eqref{eq:np3} combined with (\ref{eq:np2p}) gives (\ref{eq:supphat}).

We now assume that (\ref{eq:mdisjsupp}) holds and we show (\ref{eq:exprhat}). The previous support arguments show that
\begin{equation} \label{eq:np4}
\operatorname{supp} \mu^{\middlesubscript}_{\rm out} \cap \operatorname{supp} \mu^{\middlesubscript}_{\rm in} \,\cap\, T^* \Gamma_i = \emptyset.
\end{equation}
In other words
$$
\mu^{\middlesubscript}_{\rm out} = 0\,\, \text{ on }\operatorname{supp}\mu^{\middlesubscript}_{\rm in}\,\cap\, T^* \Gamma_i; \hspace{0.5cm}\mu^{\middlesubscript}_{\rm in} = 0\,\, \text{ on }\operatorname{supp}\mu^{\middlesubscript}_{\rm out}\,\cap\, T^* \Gamma_i.
$$
Therefore, by Lemma \ref{lem:Millretations},
\begin{equation}\label{eq:np5}
 \nu^{\middlesubscript}_n = \frac 12 \sqrt{r}  \mu^{\middlesubscript}_{\rm in}, \quad  \Re  \nu^{\middlesubscript}_j = -\frac 12 \mu^{\middlesubscript}_{\rm in}, \quad   \nu^{\middlesubscript}_d = \frac{1}{2\sqrt{r}}  \mu^{\middlesubscript}_{\rm in} \hspace{0.5cm}\text{ on }\operatorname{supp}\mu^{\middlesubscript}_{\rm in}\,\cap\, T^* \Gamma_i,
\end{equation}
and
\begin{equation}\label{eq:np6}
 \nu^{\middlesubscript}_n = \frac 12 \sqrt{r}  \mu^{\middlesubscript}_{\rm out}, \quad  \Re  \nu^{\middlesubscript}_j = \frac 12 \mu^{\middlesubscript}_{\rm out}, \quad   \nu^{\middlesubscript}_d = \frac{1}{2\sqrt{r}}  \mu^{\middlesubscript}_{\rm out} \hspace{0.5cm}\text{ on }\operatorname{supp}\mu^{\middlesubscript}_{\rm out}\,\cap\, T^* \Gamma_i.
\end{equation}
In addition, observe that, by the definition of $\widehat u$ in terms of $u$ and the definition of defect measures,
\begin{equation}\label{eq:np7}
\widetilde \nu^{\middlesubscript}_d = \nu^{\middlesubscript}_d - 2 \Re\nu_j^{\middlesubscript} + \nu_n^{\middlesubscript}.
\end{equation}
From (\ref{eq:np1}) and (\ref{eq:np7}) together with (\ref{eq:np5}) we get
$$
\widehat \nu_d^{\middlesubscript} = \Big( \frac{\sqrt r + \iota }{\sqrt r + 1} \Big)^2 \widetilde \nu_d^{\middlesubscript}\hspace{0.5cm}\text{ on }\operatorname{supp}\mu^{\middlesubscript}_{\rm in}\,\cap\, T^* \Gamma_i.
$$
Therefore, as in the same way $\widetilde \nu_d^{\middlesubscript} = \frac{1}{2\sqrt r} \widetilde \mu_{\rm in}^{\middlesubscript}$ on $\operatorname{supp}\mu^{\middlesubscript}_{\rm in}\subset \operatorname{supp}\widetilde{\mu}^{\middlesubscript}_{\rm in}$, 
$$
\widehat \nu_d^{\middlesubscript} = \Big( \frac{\sqrt r + \iota }{\sqrt r +1} \Big)^2 \frac{1}{2\sqrt r}\widetilde \mu^{\middlesubscript}_{\rm in}\hspace{0.5cm}\text{ on }\operatorname{supp}\mu^{\middlesubscript}_{\rm in}\,\cap\, T^* \Gamma_i,
$$
which gives, by (\ref{eq:np3}), 
\begin{equation}\label{eq:np8}
\widehat \nu_d^{\middlesubscript} = \Big( \frac{\sqrt r + \iota }{\sqrt r +1} \Big)^2 \frac{1}{2\sqrt r}\widetilde \mu_{\rm out}^{\leftsubscript} \circ \Phi^{\rm in}_{\middlesubscript\rightarrow \leftsubscript} \hspace{0.5cm}\text{ on }\operatorname{supp}\mu^{\middlesubscript}_{\rm in}\,\cap\, T^* \Gamma_i.
\end{equation}
In the same way, using (\ref{eq:np1}) and (\ref{eq:np7}) together with (\ref{eq:np6}), then (\ref{eq:np3}), we obtain
\begin{equation}\label{eq:np9}
\widehat \nu_d^{\middlesubscript} = \Big( \frac{\sqrt r - \iota }{\sqrt r - 1} \Big)^2 \frac{1}{2\sqrt r}\widetilde \mu_{\rm out}^{\rightsubscript} \circ \Phi^{\rm in}_{\middlesubscript\rightarrow \rightsubscript}.\hspace{0.5cm}\text{ on }\operatorname{supp}\mu^{\middlesubscript}_{\rm out}\,\cap\, T^* \Gamma_i.
\end{equation}
Together with  (\ref{eq:np4}) and (\ref{eq:np2}), (\ref{eq:np8}) and (\ref{eq:np9}) give (\ref{eq:exprhat}), and the proof is complete.
\end{proof}

Proposition \ref{prop:trace_meas2} is our key tool in the rest of this section. Throughout the section, we repeatedly use the fact that if $u=Sg$ solves (\ref{eq:modelh_timp}) with $\Vert g \Vert_{L^2}\leq C$, then Lemma \ref{lem:good_mes} 
allows us to extract defect measures and boundary defect measures for $u$.

\subsection{Effect of the composite impedance map on data concentrating as a Dirac: Proof of Theorem \ref{th:compo}, (\ref{i:compo:lower})}

We say that a family of Diracs on $T^*\Gamma_{\leftsubscript}$ is \emph{ $(\height , \mathsf {d_\leftsubscript}, \mathsf {d_{\rightsubscript}})$-non interacting} if none can be obtained as the image under the flow of another one, and if the family produce rays whose
intersections with $T^*\Gamma_{\middlesubscript}$ are all distinct. (Strictly speaking, it is the beam-like Helmholtz solutions corresponding to the Dirac data on the boundary that are non-interacting, but since we work on the boundary we talk about the family of Diracs as non-interacting.)
 To introduce such a definition, we denote, for $(x',  \xi')\in T^*(\{0\} \times \mathbb R)$ and $q\geq 1$,
$$
\begin{cases}
\Phi^1_{\rightarrow \middlesubscript} (x', \xi') := \Phi^{\rm out}_{\leftsubscript\rightarrow \middlesubscript}  (x', \xi'), \\
\Phi^{q+1}_{\rightarrow \middlesubscript} (x', \xi') := \Phi^{q}_{\rightarrow \middlesubscript} \circ \Phi^{\rm out}_{\rightsubscript\rightarrow \leftsubscript} \circ \Phi^{\rm out}_{\leftsubscript\rightarrow \rightsubscript} (x', \xi'), 
\end{cases}
\begin{cases}
\Phi^1_{\middlesubscript \leftarrow} (x', \xi') :=  \Phi^{\rm out}_{\rightsubscript\rightarrow \middlesubscript} \circ \Phi^{\rm out}_{\leftsubscript\rightarrow \rightsubscript}  (x', \xi'), \\
\Phi^{q+1}_{\middlesubscript \leftarrow} (x', \xi') := \Phi^{q}_{\middlesubscript \leftarrow} \circ \Phi^{\rm out}_{\rightsubscript\rightarrow \leftsubscript} \circ \Phi^{\rm out}_{\leftsubscript\rightarrow \rightsubscript} (x', \xi').
\end{cases}
$$
In other words, $\Phi^{q}_{\rightarrow \middlesubscript}(x',\xi')$ is the point of $T^*\Gamma_{\middlesubscript}$ attained by $(x',\xi')\in T^*\Gamma_l$ after $q-1$ reflections on the left and right boundaries of the domain and coming from the left, whereas $\Phi^{q}_{\leftarrow \middlesubscript}(x',\xi')$ is the point of $T^*\Gamma_{\middlesubscript}$ attained by $(x',\xi')\in T^*\Gamma_l$ after $q$ reflections on the left and right boundaries of the domain and coming from the right.
We can now define the notion of $(\height , \mathsf {d_\leftsubscript}, \mathsf {d_{\rightsubscript}})$-non interacting Diracs.

\begin{definition}
We say that  $(\delta_{x'_{\ell}, \xi'_\ell})_{0\leq \ell \leq N}$, for 
$(x'_{\ell}, \xi'_\ell) \in T^*\Gamma_\leftsubscript$, is a family of $(\height , \mathsf {d_\leftsubscript}, \mathsf {d_{\rightsubscript}})$-non interacting
Diracs if
\begin{gather} \label{eq:collision1}
\tfa  \; 0\leq \ell_1 \neq \ell_2 \leq N, \; \tfa  n\geq 0, \; (x'_{\ell_1}, \xi'_{\ell_1}) \neq  \big(\Phi^{\rm in}_{\rightsubscript\rightarrow \leftsubscript} \circ \Phi^{\rm in}_{\leftsubscript\rightarrow \rightsubscript} \big)^n (x'_{\ell_2}, \xi'_{\ell_2}),
\end{gather}
and
\begin{align}\nonumber
&\tfa \diamond_1, \diamond_2 \in \{ \rightarrow \middlesubscript , \middlesubscript \leftarrow\}, \; 0\leq \ell_1, \ell_2 \leq N, \; q_1, q_2\geq 0 \text{ such that } \\
&\hspace{5cm}\Phi^{q_1}_{\diamond_1} (x'_{\ell_1}, \xi'_{\ell_1}) \in T^*\Gamma_{\middlesubscript} \text{ and } \Phi^{q_2}_{\diamond_2} (x'_{\ell_2}, \xi'_{\ell_2}) \in T^*\Gamma_{\middlesubscript}, \nonumber \\
&\hspace{1cm}(\diamond_1, \ell_1, q_1) \neq  (\diamond_2, \ell_2, q_2) \implies  \Phi^{q_1}_{\diamond_1} (x'_{\ell_1}, \xi'_{\ell_1}) \neq \Phi^{q_2}_{\diamond_2} (x'_{\ell_2}, \xi'_{\ell_2})
 \label{eq:collision2} 
\end{align}
\end{definition}

The following proposition is a consequence of Proposition \ref{prop:trace_meas2}. It describes exactly the impedance-to-impedance maps in
the high frequency limit when applied to data concentrating as a sum of non-interacting Diracs.
\begin{prop} \label{prop:imDirac}
Let $\iota \in \{ 1, -1 \}$. Assume that $u=Sg$ solves (\ref{eq:modelh_timp}) with $\Vert g \Vert_{L^2} \leq C$ and let $\widetilde u : = (\hsc D_{x_1}+1) u$. If 
$$
\widetilde \nu^{\leftsubscript}_d = \sum_{\ell = 0}^N a_\ell(\xi') \delta_{x'_{\ell}, \xi'_\ell},
$$
where $a_\ell \in C^\infty(\mathbb R)$, $a_\ell \geq 0$ and $(\delta_{x'_{\ell}, \xi'_\ell})_{0\leq \ell \leq N}$ is a family of $(\height , \mathsf {d_\leftsubscript}, \mathsf {d_{\rightsubscript}})$-non interacting
Diracs,
then $\widehat u := (\hsc D_{x_1} + \iota)u$ admits Dirichlet boundary measure on $\Gamma_{\middlesubscript}^+$, satisfying, in $T^* \Gamma_i$,
\begin{multline} \label{eq:itDirac}
\widehat \nu_d^{\middlesubscript} = \bigg( \frac{\iota + \sqrt r}{1+\sqrt r}\bigg)^2 \sum_{\substack{\ell = 0 \cdots N\\ q \geq 1\\ \Phi^q_{\rightarrow \middlesubscript} (x'_{\ell}, \xi'_\ell) \in T^*\Gamma_{\middlesubscript}}} a_\ell \bigg( \frac{1-\sqrt r}{1+\sqrt r} \bigg)^{2(q-1)} \delta_{\Phi^q_{\rightarrow \middlesubscript} (x'_{\ell}, \xi'_\ell)} \\
+ \bigg( \frac{ - \iota + \sqrt r}{-1+\sqrt r}\bigg)^2 \sum_{\substack{\ell = 0 \cdots N \\ q \geq 1\\ \Phi^q_{\middlesubscript \leftarrow} (x'_{\ell}, \xi'_\ell) \in T^*\Gamma_{\middlesubscript}}} a_\ell \bigg( \frac{1-\sqrt r}{1+\sqrt r} \bigg)^{2q} \delta_{\Phi^q_{\middlesubscript \leftarrow} (x'_{\ell}, \xi'_\ell)}.
\end{multline}
\end{prop}

\begin{proof}
Let  $\Phi_{\leftsubscript\rightarrow \rightsubscript} := \Phi^{\rm in}_{\leftsubscript\rightarrow \rightsubscript}$, $\Phi_{\rightsubscript\rightarrow \leftsubscript} := \Phi^{\rm in}_{\rightsubscript\rightarrow \leftsubscript}$, 
and
$$
\Phi :=  \Phi_{\rightsubscript\rightarrow \leftsubscript} \circ \Phi_{\leftsubscript\rightarrow \rightsubscript}.
$$
We claim that
\begin{equation} \label{eq:thbeam2}
\frac{1}{2\sqrt r}\widetilde \mu_{\rm out}^{\leftsubscript} = \sum_{\ell = 0 \cdots N}  \sum_{k\geq 0} \; a_\ell \bigg| \frac{\sqrt r - 1}{\sqrt r + 1}\bigg|^{2k} \delta_{\Phi^k(y_\ell, \xi'_\ell)} \mathbf 1_{\Phi^k(y_\ell, \xi'_\ell) \in T^* \Gamma_{\leftsubscript}},
\end{equation} 
and 
\begin{equation} \label{eq:thbeam3}
\frac{1}{2\sqrt r}\widetilde \mu_{\rm out}^{\rightsubscript} =  \sum_{\ell = 0 \cdots N} \sum_{k\geq 0} \; a_\ell \Big| \frac{\sqrt r - 1}{\sqrt r + 1}\Big|^{2(k+1)} \delta_{\Phi_{\leftsubscript\rightarrow \rightsubscript}\circ \Phi^k (y_\ell, \xi'_\ell)} \mathbf 1_{\Phi_{\leftsubscript\rightarrow \rightsubscript}\circ\Phi^k(y_\ell, \xi'_\ell) \in T^* \Gamma_{\rightsubscript}}.
\end{equation} 
Once this claim is established, the result follows directly from 
Proposition \ref{prop:trace_meas2}, as the non-interaction assumption implies the 
disjoint-support assumption (\ref{eq:mdisjsupp}). We therefore only need to show (\ref{eq:thbeam2}) and  (\ref{eq:thbeam3}). 
As $ (- \hsc D_{x_1} + 1)\widetilde u = 0 \text{ on }\Gamma_{\rightsubscript}$,  Lemma \ref{lem:Millretations} together with Definition \ref{def:aux} implies that
\begin{equation} \label{eq:thbeamref1}
\widetilde \mu^{\rightsubscript}_{\rm out} = \bigg| \frac{\sqrt r - 1}{\sqrt r + 1}\bigg|^2 \widetilde \mu^{\rightsubscript}_{\rm in}. 
\end{equation} 
In addition, observe that if $\mathcal{B} \cap (\bigcup_{\ell}(x_\ell, \xi'_\ell)) = \emptyset$,
 then $\widetilde{\nu}^{\leftsubscript}_d(\mathcal B) = 0$ and thus also $\widetilde{\nu}^{\leftsubscript}_j(\mathcal B)=0$ by the Cauchy--Schwarz inequality. Hence by Lemma \ref{lem:tilde_good}, (\ref{i:nu_out}), 
\begin{equation}  \label{eq:thbeamref2}
\widetilde \mu^{\leftsubscript}_{\rm out}(\mathcal B) = \widetilde \mu^{\leftsubscript}_{\rm in}(\mathcal B), \hspace{0.5cm} \tfa  \mathcal B \subset T^*\Gamma_{\leftsubscript} \text{ with } \mathcal B \cap  \bigcup_{\ell = 0 \cdots N}(x_\ell, \xi'_\ell)  = \emptyset.
\end{equation}
Finally, by Lemma \ref{lem:tilde_good}, (\ref{i:same}) and Lemma \ref{lem:ltoi}, for any $a\in C^\infty(T^*\mathbb R)$ not depending on $x$,
\begin{equation}  \label{eq:thbeamref3}
a\widetilde\mu^{\leftsubscript}_{\rm in}(\mathcal B) = a\widetilde\mu^{\rightsubscript}_{\rm out}(\Phi_{\leftsubscript\rightarrow \rightsubscript} \mathcal B), \hspace{0.3cm} \tfa  \mathcal B \subset \mathcal B^{\leftsubscript}_{\rightsubscript\rightarrow \leftsubscript},
\end{equation}
and
\begin{equation}  \label{eq:thbeamref4}
a\widetilde\mu^{\rightsubscript}_{\rm in}(\mathcal B) = a\widetilde\mu^{\leftsubscript}_{\rm out}(\Phi_{\rightsubscript\rightarrow \leftsubscript} \mathcal B ), \hspace{0.3cm} \tfa  \mathcal B \subset \mathcal B^{\rightsubscript}_{\leftsubscript\rightarrow \rightsubscript}.
\end{equation}

Observe that we have the disjoint union
$$
T^*\Gamma_{\leftsubscript} \cap \mathcal H = \bigsqcup_{0\leq k \leq \infty} R^k, \hspace{0.5cm} R^k := \Big\{ \rho \in T^*\Gamma_{\leftsubscript}\cap \mathcal H, \; \inf \big\{ \ell \,:\, \Phi^{\ell}(\rho) \cap T^*\Gamma_{\leftsubscript}  = \emptyset \big\} = k\Big\};
$$
i.e., each $R^k$ consists of points of $T^*\Gamma_{\leftsubscript} \cap \mathcal H$ that are no longer in $T^*\Gamma_{\leftsubscript}$ after exactly $k$ iterations of the map $\Phi$ (going from left to right and back again).
Let now $\mathcal B \subset R^{k_0}$. We first assume that $k_0 < \infty$. Then, by definition,
$$
\Phi^{k_0}(\mathcal B) \cap T^*\Gamma_{\leftsubscript} = \emptyset.
$$
It follows that $\Phi_{\leftsubscript\rightarrow \rightsubscript}\circ \Phi^{k_0-1}(\mathcal B) \cap \mathcal B^{\rightsubscript}_{\leftsubscript\rightarrow \rightsubscript} = \emptyset$. As $\operatorname{supp} \widetilde\mu^{\rightsubscript}_{\rm in} \subset \mathcal B^{\rightsubscript}_{\leftsubscript\rightarrow \rightsubscript} $ by Part (\ref{i:muinlupa}) of Lemma \ref{lem:muinlup} (recall Lemma \ref{lem:tilde_good}, (\ref{i:same})), 
we have using (\ref{eq:thbeamref1}) that
$$
\widetilde\mu^{\rightsubscript}_{\rm out} (\Phi_{\leftsubscript\rightarrow \rightsubscript}\circ \Phi^{k_0-1}(\mathcal B)) = 0.
$$
If, for all $k$, $\Phi^k(\mathcal B)  \cap \bigcup_{\ell = 0 \cdots N}(x_\ell, \xi'_\ell) = \emptyset$, it follows by (\ref{eq:thbeamref1}), (\ref{eq:thbeamref2}), (\ref{eq:thbeamref3}) and (\ref{eq:thbeamref4}) that 
\begin{multline} \label{eq:beam_follow_{\middlesubscript}ay}
0 = \widetilde\mu^{\rightsubscript}_{\rm out} (\Phi_{\leftsubscript\rightarrow \rightsubscript}\circ \Phi^{k_0-1}(\mathcal B))  = \widetilde\mu^{\leftsubscript}_{\rm in} (\Phi^{k_0-1}(\mathcal B)) =  \widetilde\mu^{\leftsubscript}_{\rm out} (\Phi^{k_0-1}(\mathcal B)) 
=  \widetilde\mu^{\rightsubscript}_{\rm in} (\Phi_{\leftsubscript\rightarrow \rightsubscript}\circ \Phi^{k_0-2}(\mathcal B)) \\ = \bigg| \frac{\sqrt r - 1}{\sqrt r + 1}\bigg|^{-2} \widetilde\mu^{\rightsubscript}_{\rm out} (\Phi_{\leftsubscript\rightarrow \rightsubscript}\circ \Phi^{k_0-2}(\mathcal B)) = \cdots =  \bigg| \frac{\sqrt r - 1}{\sqrt r + 1}\bigg|^{-2k_0} \widetilde\mu^{\leftsubscript}_{\rm out}(\mathcal B),
\end{multline}
hence
\begin{equation} \label{eq:beam_follow_{\middlesubscript}ay1bis}
\text{if, for all }  k, \,\,\Phi^k(\mathcal B) \cap \bigcup_{\ell = 0 \cdots N}(x_\ell, \xi'_\ell)=\emptyset, \hspace{0.3cm} \text{ then }\widetilde\mu^{\leftsubscript}_{\rm out}(\mathcal B) = 0.
\end{equation} 
Assume now that there exists $k_1$ so that $\Phi^{k_1}(\mathcal B) \cap \bigcup_{\ell = 0 \cdots N}(x_\ell, \xi'_\ell)\neq\emptyset$. 
By the non-interaction assumption, if $\mathcal B$ is small enough,
such a $k_1$ is unique, and there exists a unique $\ell$ so that  $\Phi^{k_1}(\mathcal B) \ni (y_\ell, \xi'_\ell)$. Following the same argument as in (\ref{eq:beam_follow_{\middlesubscript}ay}) up to $k = k_1$, we get
$\widetilde \mu^{\leftsubscript}_{\rm in}(\Phi^{k_1}\mathcal B) = 0$. It follows by Lemma \ref{lem:Millretations} that $\frac{1}{2\sqrt r}\widetilde \mu^{\leftsubscript}_{\rm out}(\Phi^{k_1}\mathcal B) =  \widetilde \nu^{\leftsubscript}_d(\Phi^{k_1}\mathcal B)$. Hence $\frac{1}{2\sqrt r}\widetilde \mu^{\leftsubscript}_{\rm out}(\Phi^{k_1}\mathcal B) = a_\ell(\xi'_\ell)$. We now follow the same argument as in (\ref{eq:beam_follow_{\middlesubscript}ay}) from $k=k_1$ down to $k=0$. We get
\begin{multline} \label{eq:beam_follow_{\middlesubscript}ay2}
a_\ell(\xi'_\ell) = \frac{1}{2\sqrt r} \widetilde \mu^{\leftsubscript}_{\rm out}(\Phi^{k_1}\mathcal B)=  \frac{1}{2\sqrt r} \widetilde \mu^{\rightsubscript}_{\rm in} (\Phi_{\leftsubscript\rightarrow \rightsubscript}\circ \Phi^{k_1-1}(\mathcal B)) = \bigg| \frac{\sqrt r - 1}{\sqrt r + 1}\bigg|^{-2} \frac{1}{2\sqrt r} \widetilde \mu^{\rightsubscript}_{\rm out} (\Phi_{\leftsubscript\rightarrow \rightsubscript}\circ \Phi^{k_1-1}(\mathcal B)) \\
 = \frac{1}{2\sqrt r}\bigg| \frac{\sqrt r - 1}{\sqrt r + 1}\bigg|^{-2} \frac{1}{2\sqrt r} \widetilde \mu^{\leftsubscript}_{\rm in} (\Phi^{k_1-1}(\mathcal B)) = \cdots =   \bigg| \frac{\sqrt r - 1}{\sqrt r + 1}\bigg|^{-2k_1} \frac{1}{2\sqrt r} \widetilde \mu^{\leftsubscript}_{\rm out}(\mathcal B).
\end{multline}
From (\ref{eq:beam_follow_{\middlesubscript}ay1bis}) and (\ref{eq:beam_follow_{\middlesubscript}ay2}), it follows that (\ref{eq:thbeam2}) is verified for any $\mathcal B \subset \bigsqcup_{0\leq k < \infty} R^k$. It  remains to verify (\ref{eq:thbeam2}) for $\mathcal B \in  R^\infty$. Observe that such a set satisfies $\mathcal B \subset \{ \sqrt r = 1\}$. Hence,
using  (\ref{eq:thbeamref4}) and (\ref{eq:thbeamref1})
$$
\widetilde \mu^{\leftsubscript}_{\rm in} (\mathcal B) = \widetilde \mu^{\rightsubscript}_{\rm out}(\Phi_{\leftsubscript\rightarrow \rightsubscript} (\mathcal B)) = 0 \times \widetilde \mu^{\rightsubscript}_{\rm in}(\Phi_{\leftsubscript\rightarrow \rightsubscript} (\mathcal B)) = 0.
$$
It follows by Lemma  \ref{lem:Millretations} that $\widetilde \mu^{\leftsubscript}_{\rm out}(\mathcal B) = 2\sqrt r \widetilde\nu^{\leftsubscript}_d(\mathcal B) = 2\sqrt r \delta_{\xi' = \xi'_0, y=y_0}(\mathcal B) $. Thus
(\ref{eq:thbeam2}) is verified also for $\mathcal B \subset R^\infty$, hence (\ref{eq:thbeam2}) holds. The proof of (\ref{eq:thbeam3}) follows the same lines.
\end{proof}

To prove Theorem \ref{th:compo}, Part (\ref{i:compo:lower}), we iterate Proposition \ref{prop:imDirac} by taking a boundary
data producing a Dirac Dirichlet measure at high frequency, whose image with mass one is still in the domain after $n$ iterations. 
To do so, we need to know that a suitable family of $(\height , \mathsf {d_\leftsubscript}, \mathsf {d_{\rightsubscript}})$-non interacting
Diracs exists; this is given by the following lemma, in which we let $\overleftarrow\pi$ be the canonical projection from $T^*\Gamma_{\middlesubscript}$ to $T^*\Gamma_{\leftsubscript}$ (using that $T^*\Gamma_{\middlesubscript}\simeq T^*(\{0\}\times (0,\height ))= T^*\Gamma_{\leftsubscript}$).

\begin{lem}
\label{lem:nonintrays}
For any $N\geq 1$ and any $\sigma \in  \{ \rightarrow \middlesubscript, \middlesubscript \leftarrow\}^N$, there exists $(x'_\sigma, \xi'_\sigma) \in T^*\Gamma_{\leftsubscript}$ with $\xi'_\sigma \neq 0$ so that
\begin{enumerate}
\item \label{i:gr1} 
$$
\Phi^1_{\sigma(N)} \prod_{\ell = 1 \cdots N-1} \overleftarrow\pi \Phi^1_{\sigma(\ell)} \;(x'_\sigma, \xi'_\sigma) \in T^*\Gamma_{\middlesubscript}
$$
(where the product denotes composition of the maps)
\item \label{i:gr2} for any $0\leq \ell \leq N-1$, the family $X_\ell$ of points of $T^*\Gamma_{\leftsubscript}$ defined by 
$$
\begin{cases}
X_0 := \Big\{ (x'_\sigma, \xi'_\sigma) \Big\}, \\
X_{\ell + 1} := \Big\{ \overleftarrow\pi \Phi^q_{\diamond} (x', \xi'), \text{ s.t. }  (x', \xi') \in X_\ell, \; q \geq 1, \; \diamond \in \{ \rightarrow \middlesubscript, \middlesubscript \leftarrow\},\; \Phi^q_{\diamond} (x', \xi') \in T^*\Gamma_{\middlesubscript} \Big\}
\end{cases}
$$
forms a corresponding family of $(\height , \mathsf {d_\leftsubscript}, \mathsf {d_{\rightsubscript}})$-non interacting
Diracs, and $X_\ell \notin \{ 0, \height  \}$ for any $\ell$.
\end{enumerate}
\end{lem}

\begin{proof}
We take (for example) $x'_\sigma := \frac 12 \height $.
Observe that the property (\ref{i:gr1}) is satisfied as soon as $|\xi'_\sigma|$ is small enough (depending on $N$). Hence, it remains to show  that we can construct $\xi'_\sigma$ so that the property (\ref{i:gr2}) holds as well. To do so, it suffices to show that, given a finite set of points $(x'_j)_{1\leq j \leq n} \in \Gamma_{\leftsubscript}$, the set
$$
\mathcal A := \Big\{ (\xi'_j)_{1\leq j \leq n} \subset B(0,1)^n, \text{ so that } (x'_j, \xi'_j)_{1\leq j \leq n} \text{ does \emph{not} satisfy }\mathcal P\Big\}
$$
is discrete, where we say that a set $X$ of points of $T^*\Gamma_{\leftsubscript}$ satisfies $\mathcal P$ if
$$
\Big\{ \overleftarrow\pi \Phi^q_{\diamond} (x', \xi'), \text{ s.t. }  (x', \xi') \in X, \; q \geq 1, \; \diamond \in \{ \rightarrow \middlesubscript , \middlesubscript \leftarrow\},\; \Phi^q_{\diamond} (x', \xi') \in T^*\Gamma_{\middlesubscript} \Big\}
$$
forms a corresponding family of $(\height , \mathsf {d_\leftsubscript}, \mathsf {d_{\rightsubscript}})$-non interacting Diracs and $X_\ell \notin \{ 0, \height  \}$ for any $\ell$. Indeed, if it is the case, the set of $\xi' \neq 0$ so that $(x'_\sigma, \xi')$ does not satisfy
 (\ref{i:gr2}) is discrete, and hence it suffices to take $\xi'_\sigma$  small enough.

We therefore show that for any $n$ and any set of points  $X = (x'_j)_{1\leq j \leq n} \in \Gamma_{\leftsubscript}$, $\mathcal A $ is discrete. Assume that, for some $(\xi'_j)_{1\leq j \leq n}$,  $(x'_j, \xi'_j)_{1\leq j \leq n}$  does not satisfies $\mathcal P$. Let $\mathcal C_1(X)$ be the subset of indices of $1\leq \cdots \leq n$ so that (\ref{eq:collision1}) fails, $\mathcal C_2(X)$ so that (\ref{eq:collision2}) fails, and $\mathcal C_3$ so that $X_\ell \in \{0, \height \}$. If $j \in \mathcal C_1$ (resp. $\mathcal C_2$  or $\mathcal C_3$), taking $\widetilde \xi'_j \neq  \xi'_j$ with $|\widetilde \xi'_j -  \xi'_j|$ small enough and $\widetilde X := X \backslash (x'_j, \xi'_j) \cup (x'_j, \widetilde \xi'_j)$, we have $C_1(\widetilde X) \subset C_1(X)\backslash \{ j \}$, $C_2(X) \subset C_2(\widetilde X) $ and $C_3(X) \subset C_3(\widetilde X) $ (resp. $C_2(\widetilde X) \subset C_2(X)\backslash \{ j \}$ with $C_1(X) \subset C_1(\widetilde X) $ and $C_3(X) \subset C_3(\widetilde X) $; or $C_3(\widetilde X) \subset C_3(X)\backslash \{ j \}$ with $C_1(X) \subset C_1(\widetilde X) $ and $C_2(X) \subset C_2(\widetilde X) $), which shows, acting repetitively, that $\mathcal A$ is indeed discrete.
\end{proof}

We can now conclude.
\begin{proof}[Proof of Theorem \ref{th:model2} and Theorem \ref{th:compo}, Part (\ref{i:compo:lower})]
Let $N \geq 1$ and $\sigma \in \{ +, -\}^N$ be arbitrary; in the case of Theorem \ref{th:model2}, $N := 1$.
Let $\widetilde \sigma \in\{ \rightarrow \middlesubscript , \middlesubscript \leftarrow\}^N$ be associated to $\sigma$ by
$$
\widetilde \sigma(n) = \begin{cases}
\middlesubscript \leftarrow & \text{ if }\sigma(n) = +, \\
\rightarrow \middlesubscript  & \text{ if }\sigma(n) = -,
\end{cases}
$$
and let $(x'_{\widetilde \sigma}, \xi'_{\widetilde \sigma}) \in T^*\Gamma_{\leftsubscript}$ be given by Lemma \ref{lem:nonintrays}. 
Moreover, let $\chi \in C^\infty_c(\Gamma_l)$ be equal to one near $x'_{\widetilde \sigma}$.
We define $g(\hsc) \in L^2(\Gamma_{\leftsubscript})$ by 
$$
g(\hsc)(y) :=  \chi(y) (\pi h)^{-1/4} e^{i (y-x'_{\widetilde \sigma})\cdot \xi'_{\widetilde \sigma}/\hsc} e^{-i(y-x'_{\widetilde \sigma})^2/2\hsc}
$$
{(compare to \eqref{eq:data_g})}, 
 and $(u^n)_{1 \leq n \leq N}$ to be the cascade of solutions of (\ref{eq:modelh_timp})  associated with $\sigma$ and $g$ by $S(\height , \mathsf {d_\leftsubscript}^\pm, \mathsf {d_{\rightsubscript}}^\pm)$, that is so that
$$
\begin{cases}
(-\hsc^2 \Delta - 1) u^n = 0 \text{ in } \domain \\
(\hsc D_{x_1}+1) u^n = \gamma_{\Gamma_{\middlesubscript}(\mathsf {d}^{\sigma(n-1)})} (\hsc_\ell D_{x_1} \pm_{\sigma(n-1)} 1) u^{n-1}_\ell \; \text{ on } \Gamma_{\leftsubscript}, \\
\text{$u^n$ is outgoing near $\Gamma'_s$}, \\
(- \hsc D_{x_1}+1) u^n = 0 \; \text{ on } \Gamma_{\rightsubscript}(\mathsf {d_{\rightsubscript}}^{\sigma(n)}), \\
\end{cases}
$$
where we denoted $\pm_{\sigma(n)} := \sigma(n)$ to improve readability, and
$$
\begin{cases}
(-\hsc^2 \Delta - 1) u^1 = 0 \text{ in } \domain \\
(\hsc D_{x_1}+1) u^1 = g(\hsc) \; \text{ on } \Gamma_{\leftsubscript}, \\
\text{$u^1$ is outgoing near $\Gamma'_s$}, \\
(- \hsc D_{x_1}+1) u^1 = 0 \; \text{ on } \Gamma_{\rightsubscript}(\mathsf {d_{\rightsubscript}}^{\sigma(1)}). \\
\end{cases}
$$
In addition, denote
$$
 \widetilde u^n_\ell : = (\hsc_\ell D_{x_1}+1) u^n_\ell,\qquad \widehat u^n_\ell : = (\hsc_\ell D_{x_1} \pm_{\sigma(n)}1) u^n_\ell.
$$
We claim that there is a sequence $\hsc_\ell$ so that, as $\ell \rightarrow \infty$
\begin{equation} \label{eq:beamsclaim}
\Vert \widehat u^N_\ell(\hsc_\ell) \Vert_{L^2(\Gamma_{\middlesubscript})} \rightarrow c \geq 1.
\end{equation}
Since $\Vert g(\hsc) \Vert_{L^2(\Gamma_l)} \rightarrow 1$ (by direct computation), showing (\ref{eq:beamsclaim}) ends the proof.
As $g$ is locally uniformly bounded in $L^2$, by Lemma \ref{lem:good_mes}, there is a sequence $\hsc_\ell$ so that $u^1_\ell$ admits defect measures and boundary defect measures. In addition, direct computation shows that
$$
\widetilde{\nu}_d^{1,\leftsubscript}
= \delta_{(x'_{\widetilde \sigma}, \xi'_{\widetilde \sigma}) }.
$$
Hence by Proposition \ref{prop:imDirac} {used with $\ell = 0$, $a_0 = 1$, and $(x'_0, \xi'_0) = (x'_{\widetilde \sigma}, \xi'_{\widetilde \sigma})$}, {$\widehat \nu^{1,\middlesubscript}_d$} exists and is given by (\ref{eq:itDirac}), {that is
\begin{multline} \label{eq:12_it1}
\widehat \nu_d^{1,\middlesubscript} =  \sum_{\substack{q \geq 1\\ \Phi^q_{\rightarrow \middlesubscript} (x'_0, \xi'_0) \in T^*\Gamma_{\middlesubscript}}} \bigg( \frac{\iota + \sqrt r}{1+\sqrt r}\bigg)^2 \bigg( \frac{1-\sqrt r}{1+\sqrt r} \bigg)^{2(q-1)} \delta_{\Phi^q_{\rightarrow \middlesubscript} (x'_0, \xi'_0)} \\
+  \sum_{\substack{ q \geq 1\\ \Phi^q_{\middlesubscript \leftarrow} (x'_0, \xi'_0) \in T^*\Gamma_{\middlesubscript}}} \bigg( \frac{ - \iota + \sqrt r}{-1+\sqrt r}\bigg)^2 \bigg( \frac{1-\sqrt r}{1+\sqrt r} \bigg)^{2q} \delta_{\Phi^q_{\middlesubscript \leftarrow} (x'_0, \xi'_0)}.
\end{multline}
Observe in particular that at least one of the coefficients in front of the Diracs in \eqref{eq:12_it1} is equal to one:~this coefficient corresponds to the first ($q=1$) element of the first sum if $\iota = 1$, or to the first  element of the second sum if $\iota = - 1$. Hence, denoting $(x_1', \xi'_1)$ the corresponding point, which is indeed in $T^* \Gamma_i$ thanks to  Lemma \ref{lem:nonintrays}, Part (\ref{i:gr1}), we have
$$
\widehat \nu_d^{1,\middlesubscript} \geq \delta_{(x_1', \xi'_1)}
$$
in the sense of measures (i.e., testing by any positive function).
We now iterate the argument. The existence of $\widehat \nu^{1,\middlesubscript}_d$ implies that, up to a subsequence, $\gamma_{\Gamma_{\middlesubscript}(\mathsf {d}^{\sigma(1)})} (\hsc_\ell D_{x_1} \pm_{\sigma(1)} 1) u^{1}_\ell$ is locally uniformly bounded in $L^2$, and, by the boundary condition
 satisfied by our cascade of solutions, $\widetilde{\nu}_d^{2,\leftsubscript} = \widehat{\nu}_d^{1,\middlesubscript}$; hence $\widetilde{\nu}_d^{2,\leftsubscript}$ is given by the right-hand side of (\ref{eq:12_it1}). By the non self-interaction property of Lemma  \ref{lem:nonintrays}, Part (\ref{i:gr2}), the family of weighted Diracs corresponding to (\ref{eq:12_it1})
satisfies the assumptions of Proposition \ref{prop:imDirac},
 hence we can iterate the argument:~$\widetilde{\nu}_d^{3,\leftsubscript} = \widehat{\nu}_d^{2,\middlesubscript}$ exists and is given by (\ref{eq:itDirac}). 
Repeated use of Proposition \ref{prop:imDirac} combined with Lemma  \ref{lem:nonintrays}, Part (\ref{i:gr2}) (to ensure that the family of Diracs obtained at each step indeed satisfies the non-interraction assumption of Proposition \ref{prop:imDirac}) gives in this way
$$
\widehat \nu^{N,\middlesubscript}_{d} \geq\delta_{{(x'_N, \xi'_N)}}, \hspace{0.3cm}{(x'_N, \xi'_N)} := \Phi^1_{\sigma(N)} \prod_{\ell = 1 \cdots N-1} \overleftarrow\pi \Phi^1_{\sigma(\ell)} \;(x'_{\widetilde \sigma}, \xi'_{\widetilde \sigma}) \in T^*\Gamma_{\middlesubscript},
$$
where $(x'_N, \xi'_N) \in T^*\Gamma_i$ thanks to Lemma \ref{lem:nonintrays}, Part (\ref{i:gr1}).}
Hence (\ref{eq:beamsclaim}) follows thanks to (\ref{eq:osc_Gammai}) (from Lemma \ref{lem:good_mes}).
\end{proof}

\subsection{High-frequency nilpotence of the impedance-to-impedance map away from zero-frequency: Proof of Theorem \ref{th:compo}, Part (\ref{i:compo:upper})}
The following Proposition reflects the fact that, in the high-frequency limit, the impedance-to-impedance map associated with (\ref{eq:modelh_timp}) pushes the mass emanating from $\Gamma_{\leftsubscript}$ with an angle $\lambda$ to the horizontal up and down by a distance proportional to $\lambda^{-1}$.
We denote $\overrightarrow\pi$ the canonical projection from $T^*\Gamma_{\leftsubscript}$ to $T^*\Gamma_{\middlesubscript}$.

\begin{prop} \label{prop:pushsupp}
Let $\iota \in \{ 1, -1 \}$,  $0<\lambda<1$, and  $X^+_\lambda, X^-_\lambda \subset T^*\Gamma_{\leftsubscript}$ so that
$$
X^+_\lambda \subset \{ \xi' \geq \lambda\}, \hspace{0.5cm}X^-_\lambda \subset \{ \xi' \leq -\lambda\}.
$$
Assume that $u=Sg$ solves (\ref{eq:modelh_timp}) with $\Vert g \Vert_{L^2} \leq C$ and let $\widetilde u : = (\hsc D_{x_1}+1) u$. If 
$$
\operatorname{supp} \widetilde \nu^{\leftsubscript}_d \subset X^+_\lambda \cup  X^-_\lambda,
$$
then, up to a subsequence $\widehat u := (\hsc D_{x_1} + \iota)u$ admits Dirichlet defect measure on $\Gamma_{\middlesubscript}^+$, satisfying
$$
 \operatorname{supp} \widehat \nu^{\middlesubscript}_d \,\cap\, {T^* \Gamma_i}\subset \bigg(\overrightarrow \pi X^+_\lambda + \bigg(\frac{\mathsf {d_\leftsubscript} \lambda}{\sqrt{1-\lambda^2}}, 0\bigg) \bigg) \cup \bigg( \overrightarrow \pi X^-_\lambda + \bigg(-\frac{\mathsf {d_\leftsubscript} \lambda}{\sqrt{1-\lambda^2}}, 0\bigg) \bigg).
$$
\end{prop}

\begin{proof}
Observe that, by Proposition \ref{prop:trace_meas2}, 
$$
\operatorname{supp}\widehat \nu_d^{\middlesubscript} \,\cap\, {T^* \Gamma_i} \subset  \operatorname{supp} \widetilde \mu_{\rm out}^{\leftsubscript} \circ \Phi^{\rm in}_{\middlesubscript\rightarrow \leftsubscript} \cup  \operatorname{supp} \widetilde \mu_{\rm out}^{\rightsubscript} \circ \Phi^{\rm in}_{\middlesubscript\rightarrow \rightsubscript}.
$$
Therefore, the result follows if we can show that
\begin{multline} \label{eq:suppUP}
 \big( \operatorname{supp} \widetilde \mu_{\rm out}^{\leftsubscript} \circ \Phi^{\rm in}_{\middlesubscript\rightarrow \leftsubscript} \cup  \operatorname{supp} \widetilde \mu_{\rm out}^{\rightsubscript} \circ \Phi^{\rm in}_{\middlesubscript\rightarrow \rightsubscript} \big)
 \,\cap\, {T^* \Gamma_i} \\
 \subset \bigg(\overrightarrow \pi X^+_\lambda + \bigg(\frac{\mathsf {d_\leftsubscript} \lambda}{\sqrt{1-\lambda^2}}, 0\bigg) \bigg) \cup \bigg( \overrightarrow \pi X^-_\lambda + \bigg(-\frac{\mathsf {d_\leftsubscript} \lambda}{\sqrt{1-\lambda^2}}, 0\bigg) \bigg).
\end{multline}
To show (\ref{eq:suppUP}), we focus on the part of the statement concerning $\widetilde \mu_{\rm out}^{\leftsubscript}$; the proof for $\widetilde \mu_{\rm out}^{\rightsubscript} $ is 
similar. To do so, let $\mathcal B_0 \subset {T^* \Gamma_i}$ be such that
\begin{equation} \label{eq:UPB}
\mathcal B_0 \cap \bigg( \overrightarrow \pi X^+_\lambda + \bigg(\frac{\mathsf {d_\leftsubscript} \lambda}{\sqrt{1-\lambda^2}}, 0\bigg) \bigg) \cup \bigg( \overrightarrow \pi X^-_\lambda + \bigg(-\frac{\mathsf {d_\leftsubscript} \lambda}{\sqrt{1-\lambda^2}}, 0\bigg)\bigg) = \emptyset.
\end{equation}
Our goal is to show that $ \widetilde \mu_{\rm out}^{\leftsubscript}  ( \Phi^{\rm in}_{\middlesubscript\rightarrow \leftsubscript} \mathcal B_0) = 0$.
We denote $\Phi_{\leftsubscript\rightarrow \rightsubscript} := \Phi^{\rm in}_{\leftsubscript\rightarrow \rightsubscript}$, $\Phi_{\rightsubscript\rightarrow \leftsubscript} := \Phi^{\rm in}_{\rightsubscript\rightarrow \leftsubscript}$, and
$$
\Phi :=  \Phi_{\rightsubscript\rightarrow \leftsubscript} \circ \Phi_{\leftsubscript\rightarrow \rightsubscript}.
$$
By simple geometric optics, a ray starting from $(x',\xi')\in T^*\Gamma$ makes an angle $\sin \xi'$ to the horizontal. Using that 
$\tan \arcsin \lambda = \frac{\lambda}{\sqrt{1-\lambda^2}}$, we then have that
\begin{multline*} 
(x',\xi') \notin \bigg(\overrightarrow \pi X^+_\lambda +\bigg (\frac{\mathsf {d_\leftsubscript} \lambda}{\sqrt{1-\lambda^2}}, 0\bigg) \bigg) \cup \Big( \overrightarrow \pi X^-_\lambda + \bigg(-\frac{\mathsf {d_\leftsubscript} \lambda}{\sqrt{1-\lambda^2}}, 0\bigg) \bigg) \\
\implies \tfa  n \geq 0, \; \Phi^n \circ \Phi^{\rm in}_{\middlesubscript\rightarrow \leftsubscript} (x', \xi') \notin X^+_\lambda \cup X^-_\lambda.
\end{multline*}
Hence, (\ref{eq:UPB}) together with the support assumption gives that
\begin{equation} \label{eq:UPBzero}
\tfa  n \geq 0, \hspace{0.3cm} \widetilde \nu_d^{\leftsubscript} \big( \Phi^n \circ \Phi^{\rm in}_{\middlesubscript\rightarrow \leftsubscript} \mathcal B_0 \big) = 0.
\end{equation}
As $ (- \hsc D_{x_1} + 1)\widetilde u = 0 \text{ on }\Gamma_{\rightsubscript}$, by Lemma \ref{lem:Millretations} together with Definition \ref{def:aux},
\begin{equation} \label{eq:614_2}
\widetilde \mu^{\rightsubscript}_{\rm out} = \bigg( \frac{\sqrt r - 1}{\sqrt r + 1}\bigg)^2 \widetilde \mu^{\rightsubscript}_{\rm in}.
\end{equation}
In addition, by Lemma \ref{lem:tilde_good}, (\ref{i:nu_out}), 
\begin{equation}  \label{eq:thbeamref2b}
\widetilde \mu^{\leftsubscript}_{\rm out}(\mathcal B) = \widetilde \mu^{\leftsubscript}_{\rm in}(\mathcal B), \hspace{0.5cm} \tfa  \mathcal B \subset T^*\Gamma_{\leftsubscript} \text{ s.t. }\widetilde\nu^{\leftsubscript}_d(\mathcal B) = 0. 
\end{equation}
Finally, by Lemma \ref{lem:ltoi} (recall Lemma \ref{lem:tilde_good}, (\ref{i:same})), for any $a\in C^\infty(T^*\mathbb R)$ not depending on $x$,
\begin{equation}  \label{eq:thbeamref3b}
a\widetilde \mu^{\leftsubscript}_{\rm in}(\mathcal B) = a\widetilde \mu^{\rightsubscript}_{\rm out}(\Phi_{\leftsubscript\rightarrow \rightsubscript} \mathcal B), \hspace{0.3cm} \tfa  \mathcal B \subset \mathcal B^{\leftsubscript}_{\rightsubscript\rightarrow \leftsubscript},
\end{equation}
and
\begin{equation}  \label{eq:thbeamref4b}
a\widetilde \mu^{\rightsubscript}_{\rm in}(\mathcal B) = a\widetilde \mu^{\leftsubscript}_{\rm out}(\Phi_{\rightsubscript\rightarrow \leftsubscript} \mathcal B ), \hspace{0.3cm} \tfa  \mathcal B \subset \mathcal B^{\rightsubscript}_{\leftsubscript\rightarrow \rightsubscript}.
\end{equation}

We denote $\widetilde {\mathcal B}_0 := \Phi^{\rm in}_{\middlesubscript\rightarrow \leftsubscript} \mathcal B_0$, consider again the disjoint union
$$
T^*\Gamma_{\leftsubscript} \cap \mathcal H = \bigsqcup_{0\leq k \leq \infty} R^k, \hspace{0.5cm} R^k := \Big\{ \rho \in T^*\Gamma_{\rightsubscript}\cap \mathcal H, \; \inf\big \{ \ell, \; \Phi^{\ell}(\rho) \cap T^*\Gamma_{\leftsubscript}  = \emptyset\big \} = k\Big\},
$$
and decompose
$$
\widetilde {\mathcal B}_0 = \bigsqcup_{0\leq k \leq \infty}  \widetilde{ \mathcal B} ^k_0.
$$
We show that  
\begin{equation}\label{eq:UPgoal}
\tfa  \; 0\leq k \leq \infty, \hspace{0.5cm} \widetilde \mu_{\rm out}^{\leftsubscript}  ( \widetilde{ \mathcal B} ^k_0) = 0, 
\end{equation}
which completes the proof. We first assume
that $k < \infty$.
Then,
$$
\Phi^{k}(  \widetilde{ \mathcal B} ^k_0 ) \cap T^*\Gamma_{\leftsubscript} = \emptyset.
$$
It follows that $\Phi_{\leftsubscript\rightarrow \rightsubscript}\circ \Phi^{k-1}(  \widetilde{ \mathcal B} ^k_0) \cap \mathcal B^{\rightsubscript}_{\leftsubscript\rightarrow \rightsubscript} = \emptyset$. As, by Lemma \ref{lem:muinlup}, (\ref{i:muinlupa}) (and Lemma \ref{lem:tilde_good}, (\ref{i:same}) again), $\operatorname{supp} \widetilde \mu^{\rightsubscript}_{\rm in} \subset \mathcal B^{\rightsubscript}_{\leftsubscript\rightarrow \rightsubscript} $, we have using (\ref{eq:614_2})
$$
\widetilde \mu^{\rightsubscript}_{\rm out} (\Phi_{\leftsubscript\rightarrow \rightsubscript}\circ \Phi^{k-1}(  \widetilde{ \mathcal B} ^k_0)) = 0.
$$
It follows by (\ref{eq:614_2}), (\ref{eq:thbeamref2b}), (\ref{eq:thbeamref3b}) and (\ref{eq:thbeamref4b}) combined with (\ref{eq:UPBzero}) that 
\begin{multline*} 
0 = \widetilde \mu^{\rightsubscript}_{\rm out} (\Phi_{\leftsubscript\rightarrow \rightsubscript}\circ \Phi^{k-1}(\widetilde{ \mathcal B} ^k_0))  = \widetilde \mu^{\leftsubscript}_{\rm in} (\Phi^{k-1}(\widetilde{ \mathcal B} ^k_0)) =  \widetilde \mu^{\leftsubscript}_{\rm out} (\Phi^{k-1}(\widetilde{ \mathcal B} ^k_0)) 
=  \widetilde \mu^{\rightsubscript}_{\rm in} (\Phi_{\leftsubscript\rightarrow \rightsubscript}\circ \Phi^{k-2}(\widetilde{ \mathcal B} ^k_0)) \\ = \bigg( \frac{\sqrt r - 1}{\sqrt r + 1}\bigg)^{-2} \widetilde \mu^{\rightsubscript}_{\rm out} (\Phi_{\leftsubscript\rightarrow \rightsubscript}\circ \Phi^{k-2}(\widetilde{ \mathcal B} ^k_0)) = \cdots =    \bigg( \frac{\sqrt r - 1}{\sqrt r + 1}\bigg)^{-2k} \widetilde \mu^{\leftsubscript}_{\rm out}(\widetilde{ \mathcal B} ^k_0),
\end{multline*}
hence (\ref{eq:UPgoal}) holds for all $k<\infty$. It remains to show that $\widetilde \mu_{\rm out}^{\leftsubscript}  ( \widetilde{ \mathcal B} ^\infty_0) = 0$.
Observe that such a set satisfies $ \widetilde{ \mathcal B} ^\infty_0 \subset \{ \sqrt r = 1\}$. Hence,
using  (\ref{eq:thbeamref4b}) and (\ref{eq:614_2})
$$
\widetilde \mu^{\leftsubscript}_{\rm in} (\widetilde{ \mathcal B} ^\infty_0) = \widetilde \mu^{\rightsubscript}_{\rm out}(\Phi_{\leftsubscript\rightarrow \rightsubscript} (\widetilde{ \mathcal B} ^\infty_0)) = 0 \times \widetilde \mu^{\rightsubscript}_{\rm in}(\Phi_{\leftsubscript\rightarrow \rightsubscript} (\widetilde{ \mathcal B} ^\infty_0)) = 0,
$$
from which, by (\ref{eq:thbeamref2b}) combined with (\ref{eq:UPBzero})
$$
\widetilde \mu_{\rm out}^{\leftsubscript}  ( \widetilde{ \mathcal B} ^\infty_0) = 0,
$$
which ends the proof.
\end{proof}

We can now conclude.

\begin{proof}[Proof of Theorem \ref{th:compo}, Part (\ref{i:compo:upper})] 
Let $\sigma \in \{+, -\}^{\mathbb N}$ and $\lambda>0$ and assume that the claim fails. Then, for any $A>0$, there exists $n \geq A$, $\epsilon >0$, $g_\ell \in L^2(\Gamma_{\leftsubscript})$ with $\Vert g_\ell \Vert_{L^2} = 1$ and $\hsc_\ell \rightarrow 0$ so that
$$
\tfa  \ell, \hspace{0.3cm} \Vert \mathcal I ^{\sigma_n}(\height , \mathsf{d^+_\leftsubscript}, \mathsf {d_{\rightsubscript}}^+, \mathsf{d^-_\leftsubscript}, \mathsf {d_{\rightsubscript}}^-, \hsc_\ell^{-1}) \Pi^{\hsc_\ell^{-1}}_\lambda g_\ell \Vert_{L^2} \geq \epsilon.
$$
Let 
$$
f_\ell := \Pi^{\hsc_\ell^{-1}}_\lambda g_\ell.
$$
We show that, for any $n$ big enough, we have, up to a subsequence in $\ell$, as $\ell \rightarrow\infty$
\begin{equation}\label{eq:nil:contrme}
\Vert \mathcal I ^{\sigma_n}(\height , \mathsf{d^+_\leftsubscript}, \mathsf {d_{\rightsubscript}}^+, \mathsf{d^-_\leftsubscript}, \mathsf {d_{\rightsubscript}}^-, \hsc_\ell^{-1}) f_\ell \Vert_{L^2} \longrightarrow 0,
\end{equation}
which gives a contradiction and completes the proof. We denote $(u_\ell^n)_{n\geq 1}$ the cascade of solutions of (\ref{eq:modelh_timp})  associated with $\sigma$ and $f_\ell$ by $S(\height , \mathsf {d_\leftsubscript}^\pm, \mathsf {d_{\rightsubscript}}^\pm)$, that is so that
$$
\begin{cases}
(-\hsc_\ell^2 \Delta - 1) u^n_\ell = 0 \text{ in } \domain \\
(\hsc_\ell D_{x_1}+1) u^n_\ell = \gamma_{\Gamma_{\middlesubscript}(\mathsf {d}^{\sigma(n-1)})} (\hsc_\ell D_{x_1} \pm_{\sigma(n-1)} 1) u^{n-1}_\ell \; \text{ on } \Gamma_{\leftsubscript}, \\
\text{$u^n_\ell$ is outgoing near $\Gamma'_s$}, \\
(- \hsc_\ell D_{x_1}+1) u^n_\ell = 0 \; \text{ on } \Gamma_{\rightsubscript}(\mathsf {d_{\rightsubscript}}^{\sigma(n)}), \\
\end{cases}
$$
where we denoted $\pm_{\sigma(n)} := \sigma(n)$, and
$$
\begin{cases}
(-\hsc_\ell^2 \Delta - 1) u^1_\ell = 0 \text{ in } \domain \\
(\hsc_\ell D_{x_1}+1) u^1_\ell =f_\ell \; \text{ on } \Gamma_{\leftsubscript}, \\
\text{$u^1_\ell$ is outgoing near $\Gamma'_s$}, \\
(- \hsc_\ell D_{x_1}+1) u^1_\ell = 0 \; \text{ on } \Gamma_{\rightsubscript}(\mathsf {d_{\rightsubscript}}^{\sigma(1)}). \\
\end{cases}
$$
In addition, denote
$$
 \widetilde u^n_\ell : = (\hsc_\ell D_{x_1}+1) u^n_\ell, \hspace{0.3cm} \widehat u^n_\ell : = (\hsc_\ell D_{x_1} \pm_{\sigma(n)}1) u^n_\ell.
$$
As $f_\ell$ is uniformly bounded in $L^2$, by Lemma \ref{lem:good_mes}, $u^1_\ell$ admit defect measures and boundary defect measures. As $\operatorname{supp} \widetilde \nu^{1,\leftsubscript}_d \subset X^+_\lambda \cup X^-_\lambda$ by definition of $\Pi^{\hsc_\ell^{-1}}_\lambda$, it follows by Proposition \ref{prop:pushsupp} that, up to a subsequence $\widehat u^1_\ell$ admit Dirichlet boundary defect measure on $\Gamma_{\middlesubscript}$, satisfying
$$
 \operatorname{supp} \widehat \nu^{1,\middlesubscript}_d \subset X^+_{\lambda,1} \cup X^-_{\lambda, 1}, \qquad  X^\pm_{\lambda,1} := \overrightarrow \pi X^\pm_{\lambda} + \bigg(\pm \frac{\mathsf d^{\sigma(1)}_\leftsubscript \lambda}{\sqrt{1-\lambda^2}}, 0\bigg).
$$
In addition, observe that the existence of $\widehat \nu^{1,\middlesubscript}_d$ implies that, up to a subsequence, $\gamma_{\Gamma_{\middlesubscript}(\mathsf {d}^{\sigma(1)})} (\hsc_\ell D_{x_1} \pm_{\sigma(1)} 1) u^{1}_\ell$ is locally uniformly bounded in $L^2$, hence we can iterate the argument. We obtain that, for any $n$, up to a subsequence in $\ell$,
$\widehat u^n_\ell$ admits Dirichlet boundary defect measure on $\Gamma_{\middlesubscript}$, satisfying
$$
 \operatorname{supp} \widehat \nu^{n,\middlesubscript}_d \subset X^+_{\lambda,n} \cup X^-_{\lambda, n}, \qquad  X^\pm_{\lambda,n} := X^\pm_{\lambda, n-1} + \bigg(\pm\frac{\mathsf d^{\sigma(n)}_\leftsubscript \lambda}{\sqrt{1-\lambda^2}}, 0\bigg).
$$
It follows that 
$$
\tfa  \,n \geq \frac{\height }{\min (\mathsf{d^+_\leftsubscript}, \mathsf{d^-_\leftsubscript})}\frac{\sqrt{1-\lambda^2}}{\lambda},
\qquad
\operatorname{supp} \widehat \nu^{n,\middlesubscript}_d \cap {T^* \Gamma_{\middlesubscript}} = \emptyset.
$$
Thanks to (\ref{eq:osc_Gammai}), this implies (\ref{eq:nil:contrme}) and therefore ends the proof. 
\end{proof}

\section{The wellposedness results i.e.~proof of Lemma \ref{lem:modelwp} and \ref{lem:modelHF}, and interior trace bounds} \label{sec:5}

\subsection{Definition of the admissible solution operators}

Let us first define extensions of our models. For $\epsilon > 0$, let
$$
\domain_\epsilon^1 := (0, \mathsf {d_{\leftsubscript}} + \mathsf {d_{\rightsubscript}} + \epsilon) \times (-\epsilon,\height+\epsilon), \hspace{0.3cm}\domain_\epsilon^2 := (0, \mathsf {d_{\leftsubscript}} + \mathsf {d_{\rightsubscript}}) \times (-\epsilon,\height+\epsilon),
$$
let $\Gamma^{1,2}_{\mathsf x, \epsilon}$ denote the corresponding parts of the boundaries of $\domain_\epsilon^{1,2}$, and $g\in L^2(\Gamma_l)$ being given, define $g_\epsilon$ as its extension by zero to $L^2(\Gamma^1_{l,\epsilon})=L^2(\Gamma^2_{l,\epsilon})$. Denote
also $\Gamma^1_{b, \epsilon} := \Gamma^1_{l,\epsilon}$ and $\Gamma^2_{b, \epsilon} := \Gamma^2_{l,\epsilon} \cup \Gamma^2_{r,\epsilon} $. 
We define the extended models (\ref{eq:modelh_ext}) and (\ref{eq:modelh_timp_ext})
in the following way.

\noindent\begin{minipage}{0.5\linewidth}
\begin{equation} \label{eq:modelh_ext} \tag{M1$_\epsilon$}
\begin{cases}
(-\hsc^2 \Delta - 1) u = 0 \text{ in }\domain^1_\epsilon \\
(\hsc D_{x_1}+1) u = g_\epsilon \; \text{ on } \Gamma^1_{\leftsubscript, \epsilon}, \\
\text{$u$ is outgoing near $\Gamma^1_{s, \epsilon}$}.
\end{cases}
\end{equation}
\end{minipage}
\begin{minipage}{0.5\linewidth}
\begin{equation} \label{eq:modelh_timp_ext} \tag{M2$_\epsilon$}
\begin{cases}
(-\hsc^2 \Delta - 1) u = 0 \text{ in }\domain^2_\epsilon \\
(\hsc D_{x_1}+1) u = g_\epsilon \; \text{ on } \Gamma^2_{\leftsubscript, \epsilon}, \\
\text{$u$ is outgoing near $\Gamma'^2_{s, \epsilon}$}, \\
(- \hsc D_{x_1}+1) u = 0 \; \text{ on } \Gamma^2_{\rightsubscript, \epsilon}. \\
\end{cases}
\end{equation}
\end{minipage}\par\vspace{\belowdisplayskip}
\noindent We also define
$$
\underline \domain_\epsilon^1 := (-\infty, \mathsf {d_{\leftsubscript}} + \mathsf {d_{\rightsubscript}} + \epsilon) \times (-\epsilon,\height+\epsilon), \hspace{0.3cm}\underline \domain_\epsilon^2 := (-\infty, \infty) \times (-\epsilon,\height+\epsilon),
$$
and, if $u \in L^2(U)$, with $U \subset \mathbb R^2$, we denote by $\underline u$ its extension by zero to $\mathbb R^2$.

\begin{definition} \label{def:adm}
A solution operator $S$ associated to model (\ref{eq:modelh}), resp.~(\ref{eq:modelh_timp}), is said to be $\epsilon$-admissible if there exists $\epsilon, C, \hsc_0 > 0$ such that for any $g\in L^2(\Gamma_l)$, $Sg$ can be extended to a solution of (\ref{eq:modelh_ext}), resp.~ (\ref{eq:modelh_timp_ext}), that
\begin{enumerate}
\item \label{i:bdd} is bounded: $\Vert  u \Vert_{L^2(D^{1,2}_{\epsilon})}\leq C \Vert g \Vert_{L^2(\Gamma_l)}$ for all $0<\hsc\leq \hsc_0$, and
\item \label{i:bdd_trace} has bounded traces: $\Vert  u \Vert_{L^2(\Gamma^{1,2}_{b,\epsilon})} + \Vert \hsc \partial_n u \Vert_{L^2(\Gamma^{1,2}_{b,\epsilon})} \leq C \Vert g \Vert_{L^2(\Gamma_l)}$ for all $0<\hsc\leq \hsc_0$.
\end{enumerate}

\end{definition}

In \S\ref{ss:wpend} 
 and \S\ref{ss:int_trace} below, we use some results about semiclassical pseudo-differential calculus in $\mathbb R^d$, 
in particular, the elliptic parametrix construction. We refer for example to \cite[Appendix E]{DyZw:19} for a precise definition of the pseudo-differential operator classes 
and the statement of this result (alternatively, a brief introduction to these techniques can be found in \cite[\S4]{LSW3}).

\subsection{Proof of Lemma \ref{lem:modelwp} and \ref{lem:modelHF}} \label{ss:wpend}

\begin{proof}[Proof of Lemma \ref{lem:modelwp}]
We prove existence for (\ref{eq:modelh_timp}); the proof for (\ref{eq:modelh}) is similar.

\subsection*{Existence of a solution to (\ref{eq:modelh_timp_ext})}
For $\epsilon>0$, let $\Theta = \Theta_1 \cup \Theta_2$ be the union of two smooth, connected, convex subsets $\Theta_1, \Theta_2 \subset \mathbb R^2$ so that 
\begin{equation} \label{eq:Theta_def}
\Theta \cap \big((0, \mathsf{d_r} + \mathsf{d_l}) \times \mathbb R\big) = \emptyset, \quad \Theta_1 \cap \big(\{0\}\times \mathbb R\big) = \Gamma^2_{l,\epsilon}, \quad\tand\quad
\Theta_2 \cap \big(\{\mathsf{d_l} + \mathsf{d_r} \}\times \mathbb R\big) = \Gamma^2_{r, \epsilon}.
\end{equation}
We extend $g_\epsilon$ to $\partial \Theta$ by $0$, as a new function $\widetilde g \in L^2(\partial \Theta)$. Now, let $w$ be solution the exterior impedance problem with the Sommerfeld radiation condition in $\Omega := \mathbb R^2 \backslash \overline{\Theta}$
\begin{equation} \label{eq:Rdpb_aux}
\begin{cases}
(-\hsc^2 \Delta - 1) w = 0 &\text{ in } \Omega, \\
(-\hsc D_{n}+1) w = \widetilde g  &\text{ on } \partial \Theta, \\
(-\hsc D_r + 1 ) w = o(r^{-\frac{1}{2}}) &\text{ as } r \rightarrow \infty,
\end{cases}
\end{equation}
where $n$ is the normal pointing into $\Theta$ (i.e., the normal pointing out of $\domain_\epsilon^+$ on $\Gamma_{\leftsubscript, \epsilon}$ and $\Gamma_{\rightsubscript, \epsilon}$). Noting the sign of the normal, the solution of \eqref{eq:Rdpb_aux} is unique by, e.g., \cite[Theorem 3.37]{CoKr:83}, and exists by Fredholm theory for all $\hsc>0$. In addition, observe that $w$ is $\hsc$-tempered with 
\begin{equation} \label{eq:euan_temp}
\Vert w \Vert_{H^1_\hsc(K \cap (\mathbb R^2 \backslash \Theta))} \leq C \hsc^{-1/4}\Vert g \Vert_{L^2}
\end{equation}
for any compact $K \subset \mathbb R^2$ by \cite[Theorem 1.8]{Sp:14}.
We now claim that
\begin{equation} \label{eq:goto_suppg}
\operatorname{WF}_\hsc w \, \cap \, \mathcal T^c \subset \big\{ \rho \in T^*( \mathbb R^2 \backslash \Theta), \; \exists t<0, \,\pi_x \varphi_t(\rho) \in \operatorname{supp} \widetilde g\big\},
\end{equation}
where $\mathcal T$ consists of the trapped rays between $\Theta_1$ and $\Theta_2$ (observe that $\mathcal T \subset \{ \xi \cdot n(x) = 0 \}$ in the notation of Definition \ref{def:outgo}), and $\varphi_t(x, \xi)$ denotes the point of the phase-space attained following, for a time $t$, the ray of geometrical optics in $\mathbb R^2 \backslash \Theta$ starting from $(x, \xi)$.
Observe that the above shows the existence of a solution to (\ref{eq:modelh_timp_ext}) by taking 
$$
u := w|_{ \domain_\epsilon^2}
$$
and any $\domain_\epsilon^+ \subset  (0, \mathsf{d_r} + \mathsf{d_l}) \times \mathbb R$ satisfying $\domain_\epsilon^+ \supset \domain_\epsilon^2 \cup \Gamma_\epsilon^2$ so that $\partial \domain_\epsilon^+ \cap(\partial\domain_\epsilon^2 \backslash \Gamma_\epsilon^2) = \partial\domain_\epsilon^2 \backslash \Gamma_\epsilon^2$.

We therefore show (\ref{eq:goto_suppg}). 
Observe that, for $A>0$ so that $B(0,A)\Supset \Theta$
\begin{equation} \label{eq:outgo_WF}
\operatorname{WF}_{\hsc} w \cap T^*B(0,A)^c \subset \big\{ \xi \cdot x > 0\big\}.
\end{equation}
The inclusion (\ref{eq:outgo_WF}) is proven similarly as in \cite{Burq}, 
observing that, for $\Psi \in C^\infty(\mathbb R^2)$ such that $\Psi = 0$ near $\Theta$ and $\Psi=1$ near $B(0,A)^c$, $\Psi w$ is outgoing (in the standard sense) and solves the following Helmholtz equation in $\mathbb R^2$,
$$
(-\hsc^2\Delta - 1)\Psi w = -[\hsc^2\Delta, \Psi]w,
$$
hence (\ref{eq:outgo_WF}) follows from the analogous property for the \emph{free} resolvent; see \cite[Proposition 2.2]{Burq}.
From $(\ref{eq:outgo_WF})$ and invariance of the wavefront set by the Hamiltonian flow in the interior (this follows, for example, from propagation 
of singularities, \cite[Theorem E.47]{DyZw:19}),
\begin{equation} \label{eq:goto_1}
\operatorname{WF}_\hsc w \subset \big\{ (x, \xi) \in T^* (\mathbb R^2 \backslash \Theta), \; \exists t<0, x + t \xi  \in \partial \Theta \big\}.
\end{equation}
On the other hand, from propagation of singularities up to the boundary (see, e.g.,~\cite{MS2} and Remark \ref{rem:prop} below), we have, for any $\tau > 0$
\begin{equation} \label{eq:goto_2}
\rho \notin \operatorname{WF}_\hsc w\, \text{ and } \,\big\{ \pi_x \varphi_t(\rho), \; t\in [0, \tau] \big\} \cap \operatorname{supp} \widetilde g = \emptyset \implies
\varphi_\tau(\rho) \notin \operatorname{WF}_\hsc w.
\end{equation}
By (\ref{eq:goto_1}) and (\ref{eq:goto_2}), we obtain (\ref{eq:goto_suppg}). We define $Sg := u_{|D}$, and it remains to show the points  
(\ref{i:bdd}) and (\ref{i:bdd_trace}) of Definition \ref{def:adm}.

\subsection*{The trace bound in Part (\ref{i:bdd_trace}) of Definition \ref{def:adm}. }

If we prove that
\begin{equation} \label{eq:trace_gen_ext}
\Vert w \Vert_{L^2(\partial \Theta)} + \Vert \hsc \partial_n w \Vert_{L^2(\partial \Theta)} \lesssim \Vert g\Vert_{L^2(\Gamma_l)}.
\end{equation}
then this implies the bound in Part (\ref{i:bdd_trace}) of Definition \ref{def:adm}.

For $R\gg 1$, let $\Omega_{R} := B(0, R) \backslash \Theta$. 
Pairing the equation by $v$ and integrating by parts on $\Omega_R$, we get
\begin{equation} \label{eq:extIPP}
\hsc^2 \N{\nabla w}^2_{L^2(\Omega_R)} - \N{w}^2_{L^2(\Omega_R)} = \hsc i \N{w}^2_{L^2(\partial \Omega_R)}  + \hsc \big\langle \hsc D_n w - w, w \rangle_{L^2(\partial \Omega_R)}.
\end{equation}
Taking the imaginary part,
we obtain by Cauchy-Schwarz inequality and the inequality $2ab\leq a^2 + b^2$ for $a, b \geq 0$,
$$
\Vert w \Vert^2_{L^2(\partial \Omega_R)} \leq \Vert  \hsc D_n w - w \Vert_{L^2(\partial B(0,R))} \Vert w \Vert_{L^2(\partial B(0,R))} + \frac 12 \Vert g  \Vert_{L^2}^2 +\frac 12 \Vert w \Vert_{L^2(\partial \Theta)}^2
$$
Using the Sommerfeld radiation condition, we obtain that 
\begin{align*}
\Vert w \Vert^2_{L^2(\partial \Omega_R)} &\leq 2 \Vert  \hsc D_n w - w \Vert_{L^2(\partial B(0,R))} \Vert w \Vert_{L^2(\partial B(0,R))} + \Vert g  \Vert_{L^2}^2 \\
&\leq 2 \epsilon_{\hsc}(R) R^{-\frac 12} \Vert  1 \Vert_{L^2(\partial B(0,R))} \Vert w \Vert_{L^2(\partial B(0,R))} + \Vert g  \Vert_{L^2}^2 \\
&\leq 2 \epsilon_{\hsc}(R) \Vert w \Vert_{L^2(\partial B(0,R))} + \Vert g  \Vert_{L^2}^2,
\end{align*}
where $\epsilon_{\hsc}(R) \rightarrow 0$ as $R\rightarrow \infty$ for fixed $\hsc>0$.
If, $\hsc >0$ being fixed, $\Vert w \Vert_{L^2(\partial B(0,R))} \leq \Vert w \Vert^2_{L^2(\partial B(0,R))}$, (\ref{eq:trace_gen_ext}) follows by taxing $R$ fixed big enough so that $2 \epsilon_{\hsc}(R) \leq \frac 12$. If is not the case, then $\Vert w \Vert_{L^2(\partial B(0,R))} \leq 1$, and (\ref{eq:trace_gen_ext}) follows by letting $R \rightarrow \infty$, since $\Vert w \Vert^2_{L^2(\partial \Theta)} \leq \Vert w \Vert^2_{L^2(\partial \Omega_R)}$.

\subsection*{An oscillatory property.}

Before showing Part (\ref{i:bdd_trace}) of Definition \ref{def:adm}, we obtain an oscillatory bound ((\ref{eq:trace_gen_ext}) below) on solutions $w$. Let $\chi \in C^\infty_c(\mathbb R^2)$, $\psi \in C^\infty(\mathbb R^2)$ be bounded and so that $\psi = 0$ in $B(0,2)$, and denote by $\underline w$ the extension of $w$ to $\mathbb R^2$ by zero.

First, observe that, as $(-\hsc ^2 \Delta - 1)$ is semiclassicaly elliptic on $\operatorname{WF}_\hsc \psi (\hsc D_x)$, 
there exists $E\in \Psi^{-2}_\hsc$ so that
\begin{equation} \label{eq:basic_para}
\psi (\hsc D_x) = E(-\hsc ^2 \Delta - 1) + O(\hsc^\infty)_{\Psi^{-\infty}}.
\end{equation}
Now,
\begin{align}
(-\hsc^2 \Delta - 1)\chi \underline w &= 1_{x\in D^2_\epsilon} (-\hsc^2 \Delta - 1) \chi w + \hsc\big(-\delta(x_1)\otimes \hsc\partial_{x_1} (\chi w) _{|x_1 = 0} + h\delta'(x_1)\otimes (\chi w) _{|x_1 = 0}\big) \nonumber \\ 
&= 1_{x\in D^2_\epsilon} [-\hsc^2 \Delta, \chi] w + \hsc\big(-\delta(x_1)\otimes \hsc\partial_{x_1} (\chi w) _{|x_1 = 0} + h\delta'(x_1)\otimes (\chi w) _{|x_1 = 0}\big) \nonumber \\
&=  [-\hsc^2 \Delta, \chi] \underline w - \hsc \delta(x_1)\otimes (\partial_{x_1}\chi) w_{|x_1 = 0} \nonumber  \\
&\hspace{1.5cm}+ \hsc(-\delta(x_1)\otimes \hsc\partial_{x_1} (\chi w) _{|x_1 = 0} + h\delta'(x_1)\otimes \chi w _{|x_1 = 0}). \label{eq:osc_int}
\end{align}
Let 
$$
z := \hsc \delta(x_1)\otimes (\partial_{x_1}\chi) w_{|x_1 = 0} + \hsc\big(-\delta(x_1)\otimes \hsc\partial_{x_1} \chi w _{|x_1 = 0} + h\delta'(x_1)\otimes (\chi w) _{|x_1 = 0}\big).
$$
Pairing with a test function $\varphi$ and using the Cauchy-Schwarz inequality and then trace inequalities, we find that
\begin{align*}
|\langle z, \varphi \rangle| &\leq C  \big( \hsc \Vert \varphi \Vert_{L^2(\{ x_1 = 0\})} + \hsc^2 \Vert \partial_n \varphi \Vert_{L^2(\{ x_1 = 0\})} \big) \big(\Vert w \Vert_{L^2(\partial \Theta)} + \Vert \hsc \partial_n w \Vert_{L^2(\partial \Theta)} \big)\\
&\leq C \big( \hsc \Vert \varphi \Vert_{H^{\frac 12 + \eta}} + \hsc^2 \Vert  \varphi \Vert_{H^{\frac 32 + \eta}} \big) \big(\Vert w \Vert_{L^2(\partial \Theta)} + \Vert \hsc \partial_n w \Vert_{L^2(\partial \Theta)} \big) \\
&\leq C \hsc^{\frac 12 - \eta} \Vert  \varphi \Vert_{H_\hsc^{\frac 32 + \eta}} \big(\Vert w \Vert_{L^2(\partial \Theta)} + \Vert \hsc \partial_n w \Vert_{L^2(\partial \Theta)} \big),
\end{align*}
for $\eta >0$. 
Hence
$$
\Vert z \Vert_{H_\hsc^{-\frac 32-\eta}} \leq C  \hsc^{\frac 12 - \eta} \big(\Vert w \Vert_{L^2(\partial \Theta)} + \Vert \hsc \partial_n w \Vert_{L^2(\partial \Theta)} \big).
$$
By combining the above with (\ref{eq:osc_int}), (\ref{eq:basic_para}) and (\ref{eq:euan_temp}), and using the fact that $E [-\hsc^2 \Delta, \chi] \in \hsc \Psi_\hsc^{-1} \subset \hsc \Psi_\hsc^{0}$, we get
\begin{equation} \label{eq:prel_oscc}
\Vert \psi(\hsc D_{x})\chi \underline w  \Vert_{L^2} \leq C  \hsc^{\frac 12 - \eta} \big(\Vert w \Vert_{L^2(\partial \Theta)} + \Vert \hsc \partial_n w \Vert_{L^2(\partial \Theta)} + \Vert g\Vert_{L^2(\Gamma_l)} \big).
\end{equation}
Together with (\ref{eq:trace_gen_ext}), this gives the oscillatory bound
\begin{equation} \label{eq:osccc}
\Vert \psi(\hsc D_{x})\chi \underline w  \Vert_{L^2} \leq C  \hsc^{\frac 12 - \eta} \Vert g\Vert_{L^2(\Gamma_l)}.
\end{equation}

\subsection*{The resolvent estimate in Part (\ref{i:bdd}) of Definition \ref{def:adm}.}

We show that, for any $\chi \in C^\infty_c(\mathbb R^2)$, there is $C>0$ so that for $\hsc_0>0$ small enough and any $0<\hsc\leq \hsc_0$
$$
\Vert \chi w\Vert_{L^2} \leq C \Vert g \Vert_{L^2}.
$$
In particular, this gives Part (\ref{i:bdd}) of Definition \ref{def:adm}.
Without loss of generality, we assume that $\Theta \Subset \operatorname{supp}\chi$.
If the estimate fails, there exists sequences $h_\ell \rightarrow 0$, $g_\ell \in L^2(\Gamma_\ell)$, and $w_\ell (\hsc_\ell)$ (that we will denote $w (\hsc_\ell)$ to lighten the notations)  so that
\begin{equation} \label{eq:asstilde1b}
\Vert \chi w (\hsc_\ell)   \Vert_{L^2} \geq \ell \Vert g_\ell \Vert_{L^2}.
\end{equation}
Rescaling, we may assume that
\begin{equation} \label{eq:asstilde2b}
\Vert \chi w (\hsc_\ell)   \Vert_{L^2}= 1.
\end{equation}
As $\Theta \Subset \operatorname{supp}\chi$, it follows from (\ref{eq:asstilde2b}) and the properties of the free outgoing resolvent that 
$w(\hsc_\ell)$ is bounded in $L^2_{\rm loc}$ (see for example \cite[Lemma 3.1]{Burq}). 
Together with (\ref{eq:asstilde1b}), the trace bound (\ref{eq:trace_gen_ext}) and the analogue of Theorem \ref{th:ex_defect} in $\mathbb R^2 \backslash \Theta$ (e.g.~\cite[Theorem 2.3]{GLS1}), up to extracting a subsequence, we can assume that $ u (\hbar_\ell)$ admits defect measure and boundary defect measures. 

Thanks to the oscillatory property (\ref{eq:prel_oscc}), it follows from (\ref{eq:asstilde2b}) that $\mu(\chi^2) = 1$ (see, e.g., \cite[Proof of Lemma 4.2]{GSW}).
To obtain a contradiction, we show that $\mu = 0$. We
focus on the most challenging part of the phase space $T^*\overline \domain^2_\epsilon\cap \domain_\epsilon^+$, that contains the trapped trajectories $\mathcal T$ (recall (\ref{eq:Theta_def})), and show that $\mu(T^*\overline \domain^2_\epsilon\cap \domain^+)=0$. The fact that $\mu$ vanishes on the other parts of the phase space is obtained in the same way, in the spirit of \cite{Burq} and \cite{GSW}, by combining the invariance of the measure with the outgoingness of the wavefront-set (with the invariance of the measure on the glancing set 
$\mathcal G$ in our case coming from the facts that, in our particular setting of the exterior of two convex obstacles, the generalized bicharacteristics through $\mathcal G$ are bicharacteristics of $H_p$ and $H_p\mu = \delta(x_1)\otimes(\mu^{\rm in} - \mu^{\rm out})$ everywhere, e.g., by \cite[Lemmas 2.10, 2.13, 2.14]{GLS1}). 

We therefore work in $\domain^2_\epsilon$ from now on. Observe that 
all the results of \S\ref{sec:3} and \S\ref{sec:4} apply to solutions of (\ref{eq:modelh_timp_ext}) in $D^2_{\epsilon}$, and we use these results in the corresponding case without further mention.
We will show that
\begin{equation} \label{eq:resImpmuout}
\mu_{\rm out}^{\leftsubscript} = \mu_{\rm in}^{\leftsubscript} = 0, \hspace{0.3cm}\mu_{\rm out}^{\rightsubscript} = \mu_{\rm in}^{\rightsubscript} = 0.
\end{equation}
Observe that, by (\ref{eq:asstilde1b}), 
$$
\Vert g_\ell\Vert_{L^2(\Gamma_{\leftsubscript})} \rightarrow 0.
$$
It follows from the above together with the boundary condition on $\Gamma_{l, \epsilon}$ that
$$
\nu^{\leftsubscript}_d + 2 \Re\nu_j^{\leftsubscript} + \nu_n^{\leftsubscript} = 0.
$$
Similarly as in the proof of Lemma \ref{lem:tilde_good}, this implies by Lemma \ref{lem:Millretations} that
\begin{equation} \label{eq:tileresrefll}
 \mu^{\leftsubscript}_{\rm out} = \bigg| \frac{\sqrt r - 1}{\sqrt r + 1}\bigg|^2 \mu^{\leftsubscript}_{\rm in},
\end{equation} 
and in the same way, as $ (- \hsc D_{x_1} + 1) u = 0 \text{ on }\Gamma_{\rightsubscript}$, 
\begin{equation} \label{eq:tileresrefl}
\mu^{\rightsubscript}_{\rm out} = \bigg| \frac{\sqrt r - 1}{\sqrt r + 1}\bigg|^2 \mu^{\rightsubscript}_{\rm in}. 
\end{equation}
Furthermore, by Lemma \ref{lem:ltoi}, denoting $\Phi_{\leftsubscript\rightarrow \rightsubscript} := \Phi^{\rm in}_{\leftsubscript\rightarrow \rightsubscript}$, $\Phi_{\rightsubscript\rightarrow \leftsubscript} := \Phi^{\rm in}_{\rightsubscript\rightarrow \leftsubscript}$, 
for any $a\in C^\infty(T^* \mathbb R)$ not depending on $x$,
\begin{equation}  \label{eq:thbeamref_res}
a \mu^{\leftsubscript}_{\rm in}(\mathcal B) = a \mu^{\rightsubscript}_{\rm out}(\Phi_{\leftsubscript\rightarrow \rightsubscript} \mathcal B), \hspace{0.3cm} \tfa  \mathcal B \subset \mathcal B^{\leftsubscript}_{\rightsubscript\rightarrow \leftsubscript},
\end{equation}
and
\begin{equation}  \label{eq:thbeamref_res2}
a \mu^{\rightsubscript}_{\rm in}(\mathcal B) = a \mu^{\leftsubscript}_{\rm out}(\Phi_{\leftsubscript\rightarrow \rightsubscript} \mathcal B ), \hspace{0.3cm} \tfa  \mathcal B \subset \mathcal B^{\rightsubscript}_{\leftsubscript\rightarrow \rightsubscript}.
\end{equation}
We have the disjoint union
$$
T^*\Gamma_{\leftsubscript} \cap \mathcal H = \bigsqcup_{0\leq k \leq \infty} R^k, \hspace{0.5cm} R^k := \Big\{ \rho \in T^*\Gamma_{\leftsubscript}\cap \mathcal H, \; \inf \big\{ \ell, \; \Phi^{\ell}(\rho) \cap T^*\Gamma_{\leftsubscript}  = \emptyset \big\} = k\Big\},
$$
where $\Phi := \Phi_{\rightsubscript\rightarrow \leftsubscript} \circ \Phi_{\leftsubscript\rightarrow \rightsubscript}$.
Let now $\mathcal B \subset R^{k_0}$. We first assume that $k_0 < \infty$. Then, $\Phi^{k_0}(\mathcal B) \cap T^*\Gamma_{\leftsubscript} = \emptyset$.
It follows that $\Phi_{\leftsubscript\rightarrow \rightsubscript}\circ \Phi^{k_0-1}(\mathcal B) \cap \mathcal B^{\rightsubscript}_{\leftsubscript\rightarrow \rightsubscript} = \emptyset$. As, by Lemma \ref{lem:muinlup}, (\ref{i:muinlupa}), $\operatorname{supp} \mu^{\rightsubscript}_{\rm in} \subset \mathcal B^{\rightsubscript}_{\leftsubscript\rightarrow \rightsubscript} $, we have $ \mu^{\rightsubscript}_{\rm in} (\Phi_{\leftsubscript\rightarrow \rightsubscript}\circ \Phi^{k_0-1}(\mathcal B)) = 0$, and using (\ref{eq:tileresrefll}), (\ref{eq:tileresrefl}), (\ref{eq:thbeamref_res}) and (\ref{eq:thbeamref_res2}) repetitively, it follows that
$\mu_{\rm out}^{\leftsubscript}(\mathcal B)=0$. Therefore
$$
\mu^{\leftsubscript}_{\rm out}\bigg( \bigsqcup_{0\leq k < \infty} R^k\bigg) = 0.
$$
In addition, as $\sqrt r = 1$ on $R^\infty$, (\ref{eq:tileresrefll}) 
 yields to $\mu^{\leftsubscript}_{\rm out}( R^\infty) = 0$ as well. Hence $\mu^{\leftsubscript}_{\rm out} = 0$. The same arguments lead to  $\mu^{\rightsubscript}_{\rm out} = 0$, and hence, using (\ref{eq:thbeamref_res}) and (\ref{eq:thbeamref_res2}) again, we conclude that (\ref{eq:resImpmuout}) holds.
Using Lemma \ref{lem:interpr} together with Corollary \ref{cor:suppmu}, it follows from (\ref{eq:resImpmuout}) that $\mu(T^*(\overline \domain^2_\epsilon\cap \domain_\epsilon^+)) = 0$.
\end{proof}

\begin{proof}[Proof of Lemma \ref{lem:modelHF}]
Similarly as in the proof of Lemma \ref{lem:modelwp}, since $S_1 g -S_2 g$ solves (\ref{eq:modelh}), resp  (\ref{eq:modelh_timp}) with zero impedance boundary condition on $\Gamma_b$,
propagation of singularities together with the outgoing condition give
$$
\operatorname{WF}_\hsc(S_1 g -S_2 g) \cap T^* D = \emptyset.
$$
The result follows by definition of the wavefront set together with Part (\ref{i:osc_inter_2}) of Lemma \ref{lem:osc_inter}.
\end{proof}

\begin{remark} \label{rem:prop}
The propagation results of Melrose and Sj\"ostrand \cite{MS1, MS2} are proved for homogeneous (that is, non semiclassical) second-order operators, with wavefront-set defined from homogeneous pseudo-differential operators. In the spirit of \cite{Leb}, to use their results in our semiclassical setting, for any sequences $\hsc = (\hsc_k)_{k\geq 1}$, $u=(u_k)_{k \geq 1}$, 
with $(-\hsc_k^2 \Delta - 1)u_k = 0$, if we define $\Theta(u) := \sum_k e^{-it\hsc_k^{-1}} u_k$, then  
for any $(x, \xi, t, \tau)\in {T^* \Omega} \times T^*\mathbb R$, $(x, \xi) \in {\operatorname{WF}_\hsc}(u)$ iff $(x, \xi, t, \tau) \in {\operatorname{WF} {(\Theta(u))}}$ (see \cite{Leb}), and $\Theta(u)$ solves the homogeneous wave equation $(\partial_t^2 - \Delta) \Theta(u) = 0$ (with damping boundary condition in the setting of this paper), to which \cite{MS1, MS2} apply; we therefore obtain propagation of singularities with $\hsc$-wavefront-sets for $u$.
For a more general presentation of semiclassical propagation of singularities up to the boundary, we refer the interested reader for example to \cite{I1, I2}.
\end{remark}

\subsection{Interior trace estimates} \label{ss:int_trace}

The purpose of this paragraph is to show the three following lemmas about the trace of the solution on the interior surface $\Gamma_i$.

\begin{lem}\label{lem:trace_Gammai}
Let $S$ be an $\epsilon$-admissible solution operator to (\ref{eq:modelh}), resp.~(\ref{eq:modelh_timp}), and $0\leq \epsilon' < \epsilon$. Then, there exists $\hsc_0>0$ and $C>0$ such that if $0<\hsc\leq \hsc_0$ then, given $g \in L^2(\Gamma_\leftsubscript)$, 
 any solution $u=Sg$ of (\ref{eq:modelh}), resp.~(\ref{eq:modelh_timp}) satisfies,
$$
 \Vert  u  \Vert_{L^2(\Gamma_{\middlesubscript, \epsilon'})} + \Vert \partial_n  u \Vert_{L^2(\Gamma_{\middlesubscript, \epsilon'})} \leq C \Vert g \Vert_{L^2(\Gamma_\leftsubscript)}.
$$
\end{lem}

\begin{lem}\label{lem:osc_Gammai}
Let $S$ be an $\epsilon$-admissible solution operator to (\ref{eq:modelh}), resp.~(\ref{eq:modelh_timp}),
 $\iota \in \{ -1, 1\}$, $g \in L^2(\Gamma_l)$ be such that $\Vert g \Vert_{L^2} \leq C$, $u=Sg$ and $\widehat \nu_d^i = \nu^{\middlesubscript}_d - 2 \iota \Re\nu_j^{\middlesubscript} + \nu_n^{\middlesubscript}$ a Dirichlet boundary measure of $\widehat u := (\hsc D_{x_1} u + \iota)u$ on $\Gamma_{i, \epsilon'}$ for some
 $0<\epsilon'<\epsilon$. Then, 
\begin{equation} \label{eq:osc_Gammai}
\Vert (\hsc D_{x_1} u + \iota) u \Vert_{L^2(\Gamma_i)} \rightarrow \widehat \nu_d^i({T^*\Gamma_i}).
\end{equation}
\end{lem}

\begin{lem} \label{lem:traceH}
If $v$ is outgoing near $\Gamma'_s$, and admits defect measures $\nu^{\middlesubscript}_d$, $\nu^{\middlesubscript}_n$, $\nu^{\middlesubscript}_j$ on $\Gamma_{\middlesubscript}^+$, then
they are supported in $\mathcal H$.
\end{lem}

In particular, 
Lemma \ref{lem:good_mes} is 
a consequence of Lemma \ref{lem:trace_Gammai} and \ref{lem:osc_Gammai} together with Part (\ref{i:bdd}) of Definition \ref{def:adm}. 

\begin{proof}[Proof of Lemma \ref{lem:good_mes}]
Existence of the defect measure and boundary measures is a 
direct consequence Lemma \ref{lem:trace_Gammai}, Parts (\ref{i:bdd}) and (\ref{i:bdd_trace}) of Definition \ref{def:adm}, and Theorem \ref{th:ex_defect}.
The relationship \eqref{eq:osc_Gammai} follows from Lemma \ref{lem:osc_Gammai} using the fact that, given an $\epsilon$-admissible solution operator, $D^+$ 
can be taken such
that $\Gamma_i^+\Subset \Gamma_{i,\epsilon}$; indeed, for $0<\epsilon'<\epsilon$ so that $\Gamma_i^+\Subset \Gamma_{i,\epsilon'}$, the boundary measures of $u$ on $\Gamma_i^+$ and of $u$ on $\Gamma_{i, \epsilon'}$ coincide on $T^*\Gamma_i$.
\end{proof}

We need the following result, quantifying in the interior of the domain the fact that an Helmholtz solution is oscillating at frequency
$\sim \hsc^{-1}$.

\begin{lem} \label{lem:osc_inter}
Let $U \subset \mathbb R^d$ be an open set and $v \in L_{\rm loc}^2(U)$ be so that $(-\hsc ^2 \Delta - 1) v = 0$ in $U$.
Then, for any $\chi \in C^\infty_c (U)$, and any $\widetilde \chi \in C^\infty_c (U)$ so that $\widetilde \chi = 1$ on $\operatorname{supp} \chi$,
the following holds.
\begin{enumerate}
\item \label{i:osc_inter_1} For any $m\geq 0$ there exists $C_m>0$ such that 
$$
\Vert \chi u \Vert_{H^m_\hsc} \leq C_m \Vert \widetilde \chi u \Vert_{L^2}.
$$
\item \label{i:osc_inter_2} For any bounded function $\psi \in C^\infty(\mathbb R^d)$ so that $\psi = 0$ in $B(0,2)$ and any $N\geq 1$ there exists $C_N>0$ such that 
$$
\Vert \psi (\hsc D_x) \chi u \Vert_{H^N_\hsc} \leq C_N \hsc^N \Vert \widetilde \chi u \Vert_{L^2}.
$$ 
\end{enumerate}
\end{lem}

\begin{proof}
Let $\chi_1 \in C^\infty_c(U)$ be so that $\chi_1 = 1$ on $\operatorname{supp} \chi$ 
and $\widetilde \chi = 1$ on $\operatorname{supp} \chi_1$.
Observe that $\chi_1(-\hsc ^2 \Delta - 1)$ is semiclassically elliptic on $\operatorname{WF}_\hsc \psi (\hsc D_x) \chi$, hence, by the elliptic parametrix construction, there exists $E\in \Psi_\hsc^{-2}$ so that
$$
\psi (\hsc D_x) \chi = E\chi_1(-\hsc ^2 \Delta - 1) + O(\hsc^\infty)_{\Psi^{-\infty}}.
$$
Applying $\widetilde \chi$ to the right of the above, we get
$$
\psi (\hsc D_x) \chi = E\chi_1(-\hsc ^2 \Delta - 1) + O(\hsc^\infty)_{\Psi^{-\infty}} \widetilde \chi,
$$
and hence (2) follows applying this identity to $v1_{U}$. 
Let now $\psi_1 \in C^\infty(\mathbb R^d)$ be such that $1 - \psi_1$ is compactly supported and $\psi_1 = 0$ on $B(0,2)$.
Since $(1-\psi_1)(\hsc D_x) \in \Psi^{-\infty}_\hsc$, 
$$
\Vert (1-\psi_1)(\hsc D_x)\chi  v \Vert_{H^m_\hsc} \leq C_m \Vert \chi  v \Vert_{L^2}.
$$
Together with (2), this gives (1).
\end{proof}

\begin{proof}[Proof of Lemma \ref{lem:trace_Gammai}]
By \cite[Lemma 3.7]{GLS1}, if $\Gamma$ is a smooth closed curve and $A\in \Psi_\hsc^m(\mathbb{R}^2)$ is such that $\operatorname{WF}_\hsc A \cap S^*\Gamma=\emptyset$, there exists $C>0$ such that
$$
\|Af\|_{L^2(\Gamma)}\leq C\|Af\|_{L^2} +Ch^{-1}\|(-\hsc^2\Delta - 1)Af\|_{L^2}+O(\hsc^\infty)\|f\|_{L^2}.
$$
Let $\epsilon'', \delta >0$ be so that $0<\epsilon' <\epsilon'' < \epsilon'' + 4 \delta< \epsilon$. Next, let $\Gamma \subset [ \frac{\mathsf d_l}{2}, \frac{\mathsf d_l + \mathsf d_r}{2} ] \times [- \epsilon'' - 4\delta, \mathsf h + \epsilon'' + 4 \delta] $ be a smooth closed curve so that
and $\Gamma \cap \big( \{ \mathsf d_l \} \times \mathbb R \big) = \{ \mathsf d_l \} \times [- \epsilon'' - 3 \delta, \mathsf h + \epsilon'' + 3\delta]$.
Next, let $\chi \in C^\infty_c(\mathbb R^2)$ be so that $\operatorname{supp} \chi \cap \Gamma =  \{ \mathsf d_l \} \times[- \epsilon'' - \delta, \mathsf h + \epsilon'' + \delta]$ and $\chi = 1$ near $\Gamma_{i,\epsilon'}$; and $\widetilde \chi \in C^\infty_c(\mathbb R^2)$ be so that $\widetilde \chi = 1$ on
$\operatorname{supp} \chi$ and $\operatorname{supp} \widetilde \chi \cap \Gamma =  \{ \mathsf d_l \} \times[- \epsilon'' - 2\delta, \mathsf h + \epsilon'' + 2 \delta]$.
Applying the above to $\chi f$ and using the pseudo-locality of the semiclassical pseudo-differential calculus, we get
\begin{equation}\label{eq:tr_gls}
\| A \chi f\|_{L^2(\Gamma_{i, \epsilon'})}\leq C\|\widetilde \chi A \chi f\|_{L^2} +Ch^{-1}\|\widetilde \chi (-\hsc^2\Delta - 1)A\chi f\|_{L^2}+O(\hsc^\infty)\|\chi f\|_{L^2}.
\end{equation}
By Lemma \ref{lem:hyperint}, there exists $A\in \Psi^0(\mathbb{R}^2)$ so that  $\operatorname{WF}_\hsc A\cap S^*\Gamma=\emptyset$ and
 $ \operatorname{WF}_\hsc (\chi \hsc D_{x_1} u) \cap \operatorname{WF}_\hsc (I-A) = \emptyset$.
In addition, let $\psi \in C^\infty(\mathbb R^d)$ so that $\psi = 0$ in $B(0,2)$ and $\psi = 1$ outside $B(0,3)$. Observe that
$$
\chi \hsc D_{x_1} u = A \chi \hsc D_{x_1} u + (1-A)\psi(\hsc D_x) \chi \hsc D_{x_1} u + (1-A)(1-\psi(\hsc D_x)) \chi \hsc D_{x_1} u,
$$
hence, by definition of the wavefront-set, as $(1-A)(1-\psi(\hsc D_x))\chi$ has a compactly supported symbol, and using Part (2) of Lemma \ref{lem:osc_inter}, 
$$
\chi \hsc D_{x_1} u = A \chi \hsc D_{x_1} u + O(\hsc^\infty) \widetilde \chi u
$$
in any $H_\hsc^m$ space. Hence, applying (\ref{eq:tr_gls}) to $\hsc D_{x_1} u$ (and using a trace estimate to deal with the $O(\hsc^\infty)$ error), we obtain the trace bound on $\hsc\partial_n u$ using the fact that $[-\hsc^2\Delta, A\chi] \in \hsc \Psi_\hsc^{1}$,
Part (1) of Lemma \ref{lem:osc_inter}, and the resolvent bound of Part (\ref{i:bdd}) of Definition \ref{def:adm}. The trace bound on $u$ is obtained in the 
same way.
\end{proof}

\begin{proof}[Proof of Lemma \ref{lem:osc_Gammai}]
Let $\psi \in C^\infty(\mathbb R)$ be so that $\psi = 0$ in $B(0,4)$ and $\psi = 1$ outside $B(0,5)$, 
and $\widetilde \psi \in C^\infty(\mathbb R^2)$ be so that $\widetilde \psi = 0$ in $B(0,2)$ and $\widetilde \psi = 1$ outside $B(0,3)$.
By definition of the Dirichlet boundary measure, it suffices to show that
$$
\psi(\hsc D_{x'}) \chi \gamma_{\Gamma_{i, \epsilon'}}(\hsc D_{x_1} + \iota) u \rightarrow 0.
$$
for any $\chi \in C^\infty_c(\mathbb R)$, such that $\chi = 1$ near $\{ \mathsf d_l \} \times [0, \mathsf h]$ and (say) $\operatorname{supp} \chi \subset \{ \mathsf d_l \} \times [-\frac {\epsilon'} 4, \mathsf h + \frac  {\epsilon'} 4]$. Let $\widetilde \chi \in C^\infty_c(\mathbb R)$, such that $\widetilde \chi = 1$ near $\operatorname{supp} \chi$ and $\operatorname{supp} \widetilde \chi \subset \{ \mathsf d_l \} \times [-\frac  {\epsilon'} 2, \mathsf h + \frac  {\epsilon'} 2]$. By pseudo-locality of the semiclassical pseudo-differential calculus (e.g.,~together with a trace estimate and
Part \ref{i:bdd} of Definition \ref{def:adm}), 
 it suffices to show that
$$
\widetilde \chi \psi(\hsc D_{x'}) \chi \gamma_{\Gamma_{i, \epsilon}}(\hsc D_{x_1}  + \iota) u \rightarrow 0.
$$
Extend $\widetilde \chi$ and $\chi$ as elements of $C^\infty_c(\mathbb R^2)$, so that $\widetilde \chi =1$ on $\operatorname{supp}{\chi}$, $\chi = 1$ near $\Gamma_i$, and  $\operatorname{supp}\widetilde \chi \subset \domain_{ \epsilon'}$. As $\widetilde \chi \psi(\hsc D_{x'}) \chi$ is a tangential operator, it commutes with the trace and it suffices to show that
$$
\Vert \widetilde \chi \psi(\hsc D_{x'}) \chi (\hsc D_{x_1}  + \iota) u \Vert_{L^2(\Gamma_{i, \epsilon})} \rightarrow 0.
$$
Now, by a trace estimate, 
$$
\Vert \widetilde \chi \psi(\hsc D_{x'}) \chi (\hsc D_{x_1}  + \iota) u \Vert_{L^2(\Gamma_{i, \epsilon})} \lesssim \hsc^{- \frac 1 2} \Vert \psi(\hsc D_{x'}) \chi (\hsc D_{x_1}  + \iota) u \Vert_{H^1_\hsc}.
$$
Decompose
$$
\psi(\hsc D_{x'}) \chi (\hsc D_{x_1}  + \iota) u = \psi(\hsc D_{x'}) \widetilde \psi(\hsc D_{x}) \chi (\hsc D_{x_1}  + \iota) u + \psi(\hsc D_{x'}) (1-\widetilde \psi(\hsc D_{x})) \chi (\hsc D_{x_1}  + \iota) u.
$$
By Part (2) of Lemma \ref{lem:osc_inter},
$$
\hsc^{-\frac 12} \Vert \psi(\hsc D_{x'}) \widetilde \psi(\hsc D_{x}) \chi (\hsc D_{x_1}  + \iota) u \Vert_{H^1_\hsc} \rightarrow 0.
$$
On the other hand, as $\operatorname{supp}(1-\widetilde \psi(\xi)) \cap \operatorname{supp} \psi(\xi')=\emptyset$, we have $\psi(\hsc D_{x'}) (1-\widetilde \psi(\hsc D_{x})) = O(\hsc^\infty)_{\Psi^{-\infty}}$, and the result follows using the bound 
in Part (\ref{i:bdd}) of Definition \ref{def:adm}.
\end{proof}

\begin{proof}[Proof of Lemma \ref{lem:traceH}]
We show the result for $\nu_d^{\middlesubscript}$; the proofs for $\nu_d^{\middlesubscript}$ and $\nu_j^{\middlesubscript}$ are the same. 
Shrinking $\domain^+$ if necessary, let $\widetilde \domain^+ \supset \domain^+$
be such that Definition \ref{def:outgo} holds with $\domain^+$ replaced by $\widetilde \domain^+$, 
and let $\chi \in C^\infty_c$ be such that $\operatorname{supp} \chi \subset \widetilde \domain^+$ and $\chi =1$ near $\Gamma_{\middlesubscript}^+$.
By Lemma \ref{lem:hyperint}, there exists $c>0$ so that
\begin{equation} \label{eq:suppWFpreli}
\operatorname{WF}_\hsc ( \chi v) \subset \{ |\xi_1| \geq c \}.
\end{equation} 
Since, in addition, $\operatorname{WF}_\hsc (\chi v)  \subset \{ |\xi|^2 = 1 \}$, 
it follows from (\ref{eq:suppWFpreli}) that
\begin{equation} \label{eq:suppWF}
\operatorname{WF}_\hsc (\chi v) \subset \{ 1 \geq r(\xi') \geq c^2 \}.
\end{equation}
Now, let $\psi \in C^\infty_c$ be so that $\operatorname{supp }\psi \subset [c^2/2, 2]$ and $\psi = 1$
near $\{ 1 \geq z \geq c^2 \}$, and define $\Psi(\xi'):=\psi(r(\xi'))$. From (\ref{eq:suppWF}) together the definition of the wavefront-set and  Lemma \ref{lem:osc_inter}, (\ref{i:osc_inter_2}), for any $s$ and any $N$ there exists $C_s,N>0$ so that
$$
\Vert (1-\Psi(\hsc D_{x'})) \chi v \Vert_{H^s_\hsc} \leq C_{s,N} \hsc^N.
$$
Hence, by Sobolev trace estimates, $\Vert (1-\Psi(\hsc D_{x'})) v \Vert_{L^2(\Gamma_{\middlesubscript})} \rightarrow 0$.
It follows that, by definition of defect measures, $\nu_d^i\big((1-\Psi(\hsc D_{x'}))^2\big) = 0$. As $(1-\Psi(\hsc D_{x'}))^2 = 1$ near $\mathcal H^c$, the Lemma follows.

\end{proof}

\section*{Acknowledgements}

DL and EAS acknowledge support from EPSRC grant EP/R005591/1, thank Jeffrey Galkowski (University College London) for useful discussions, thank Ivan Graham (University of Bath) for reading and commenting on an earlier draft of the paper, {and thank the referee for their careful reading of the paper and helpful comments.}

\bibliographystyle{amsalpha}
\bibliography{ref}

\end{document}